\documentclass[reqno,12pt]{amsart}

\usepackage{graphicx, amsmath, amssymb, amscd, amsthm, euscript, psfrag, amsfonts,bm}

\newtheorem{theorem}{Theorem}[section]
\newtheorem{lemma}[theorem]{Lemma}
\newtheorem{proposition}[theorem]{Proposition}
\newtheorem{corollary}[theorem]{Corollary}

\theoremstyle{definition}
\newtheorem{definition}[theorem]{Definition}
\newtheorem{notation}[theorem]{Notation}

\newtheorem{remark}[theorem]{Remark}

\numberwithin{equation}{section}

\newcommand{\bH}{\mathbb H}
\newcommand{\LG}{\Lambda(\Gamma)}
\newcommand{\tS}{\tilde S}
\newcommand{\Euc}{\op{Euc}}

\newcommand{\e}{\epsilon}
\newcommand{\z}{\mathbb{Z}}

\newcommand{\N}{\mathbb{N}}

\newcommand{\br}{\mathbb{R}}
\newcommand{\R}{\mathbb{R}}

\newcommand{\gt}{{\Cal G}^t}

\newcommand{\gr}{{\Cal G}^r}
\newcommand{{\grinv}}{{\Cal G}^{-r}}

\newcommand{\p}{\mathbf{p}}
\newcommand{\ba}{\backslash}\newcommand{\bs}{\backslash}

\newcommand{\G}{\Gamma}

\newcommand{\g}{\gamma}

\newcommand{\Cal}{\mathcal}

\newcommand{\la}{\langle}
\newcommand{\ra}{\rangle}

\newcommand{\SL}{\operatorname{SL}}

\newcommand{\Int}{\operatorname{int}}

\newcommand{\bdp}{{\rm bp}}

\newcommand{\inv}{^{-1}}
\newcommand{\cH}{{\mathcal{H}}}

\newcommand{\SO}{\operatorname{SO}}

\newcommand{\T}{\operatorname{T}}
\newcommand{\Th}{\operatorname{T}^1(\bH^n)}
\newcommand{\bh}{\partial\mathbb{H}^n}

\newcommand{\sk}{\operatorname{sk}}

\newcommand{\Vol}{\op{Vol}}
\newcommand{\PSL}{\op{PSL}}

\newcommand{\cB}{{\mathcal B}}
\newcommand{\cD}{{\mathcal D}}

\newcommand{\cF}{{\mathcal F}}

\newcommand{\PS}{\rm{PS}}
\newcommand{\muPS}{\mu^{\PS}}

\newcommand{\norm}[1]{\lVert #1 \rVert}
\newcommand{\abs}[1]{\lvert #1 \rvert}

\newcommand{\op}{\operatorname}
\newcommand{\supp}{\operatorname{supp}}
\newcommand{\diag}{\operatorname{diag}}

\newcommand{\codim}{\operatorname{codim}}

\newcommand{\BR}{\operatorname{BR}}
\newcommand{\Leb}{\operatorname{Leb}}
\newcommand{\BMS}{\operatorname{BMS}}
\newcommand{\Riem}{\operatorname{Riem}}
\newcommand{\Hn}{{\mathbb H}^n}
\newcommand{\rank}{\operatorname{rank}}

\newcommand{\parcorank}{\operatorname{pb-corank}}
\newcommand{\diam}{\operatorname{diam}}
\newcommand{\V}{\operatorname{Viz}}
\newcommand{\Vinv}{\operatorname{V}\inv}

\newcommand{\tE}{\tilde E}

\newcommand{\cl}[1]{\overline{#1}}
\renewcommand{\muPS}{\mu^{\PS}}
\renewcommand{\setminus}{\smallsetminus}

\newcommand{\Span}{\operatorname{span}}

\newcommand{\Lie}{\operatorname{Lie}}

\newcommand{\gri}{{\mathcal G}^{r_i}}
\newcommand{\gro}{{\mathcal G}^{r_0}}

\newcommand{\bg}{{\bm{\g}}}
\newcommand{\bk}{{\bm{k}}}
\newcommand{\bss}{{\bm{s}}}
\newcommand{\bw}{{\bm{w}}}
\newcommand{\Z}{\z}
\newcommand{\cE}{{\mathcal E}}
\newcommand{\rmr}{{\rm r}}
\newcommand{\rmp}{{\rm p}}
\newcommand{\eucl}{{\rm eucl}}
\newcommand{\bfb}{\operatorname{b}}
\newcommand{\bfh}{\operatorname{h}}

\newcommand{\GmH}{\G_H\bs H}
\newcommand{\GmG}{\G\bs G}
\newcommand{\mBR}{{\bar m}^{\BR}}

\newcommand{\bfq}{{\mathbf q}}
\newcommand{\quo}{\bfq}

\begin{document}

\title[Equidistribution and counting]
{Equidistribution and counting for orbits of geometrically finite hyperbolic
groups}

\author{Hee Oh}
\address{Mathematics department, Brown university, Providence, RI 02912
and Korea Institute for Advanced Study, Seoul, Korea}
\email{heeoh@math.brown.edu}
\thanks{H. Oh was supported in part by NSF Grant \#0629322 and \#1068094.}

\author{Nimish A. Shah}
\address{Department of Mathematics, The Ohio State University, Columbus, OH 43210}
\email{shah@math.osu.edu}
\thanks{N. Shah was supported in part by NSF Grant \#1001654.}



\begin{abstract}
Let $G$ be the identity component of $\SO(n,1)$, $n\ge 2$, acting linearly on
 a finite dimensional real vector space $V$.
 Consider a vector $w_0\in V$ such that the stabilizer of $w_0$ is a symmetric subgroup of $G$
or the stabilizer of the line $\br w_0$ is a parabolic subgroup of $G$.
For any non-elementary discrete subgroup $\G$ of $G$ with its orbit $w_0\G$ discrete, we
 compute an asymptotic formula (as $T\to \infty$)
for the number of points in $w_0\G$ of norm at most $T$,
provided that the Bowen-Margulis-Sullivan measure on $\T^1(\G\ba \bH^n)$ and the $\G$-skinning size
of $w_0$ are finite.

The main ergodic ingredient in our approach is the description for the limiting distribution of the 
orthogonal translates of a totally geodesically immersed closed submanifold of 
$\G\ba\bH^n$. We also give a criterion on the finiteness of the $\G$-skinning size of $w_0$ for 
$\G$ geometrically finite.
\end{abstract}

\maketitle
\tableofcontents

\section{Introduction}

\subsection{Motivation and Overview}
\label{subsec:motivation}
Let $G$ denote the identity component of the special orthogonal group $\SO(n,1)$, $n\ge 2$,
and $V$ a finite dimensional real vector space on which $G$ acts linearly from the right.

A discrete subgroup of a locally compact group with finite co-volume is called a lattice.
For $v\in V$ and a subgroup $H$ of $G$, let $H_{v}=\{h\in H:vh=v\}$ denote the
stabilizer of $v$ in $H$.

A subgroup $H$ of $G$ is called
{\it symmetric} if there exists a nontrivial involutive automorphism $\sigma$ of $G$ such that
the identity component of $H$ is same as the identity component of $G^\sigma=\{g\in G:
\sigma(g)=g\}$.

\begin{theorem}[Duke-Rudnick-Sarnak \cite{DukeRudnickSarnak1993}] \label{drs}
Fix $w_0\in V$ such that $G_{w_0}$ is symmetric. Let $\G$ be a lattice in $G$ such that
 $ \G_{w_0}$ is a lattice in $G_{w_0}$.
Then for any norm $\|\cdot \|$ on $V$,
$$\lim_{T\to \infty}\frac{\#\{w\in w_0\G: \|w\|<T\}}{\op{vol}(B_T)} = \frac{\op{vol}(\G_{w_0}\ba
G_{w_0})}{\op{vol}(\G\ba G)}$$
where $B_T:=\{w\in w_0G:\|w\|<T\}$ and the volumes on
$G_{w_0}, G$ and $w_0G\simeq G_{w_0}\ba G$ are computed with respect to
the right invariant measures chosen compatibly.
\end{theorem}

Eskin and McMullen \cite{EskinMcMullen1993}
gave a simpler proof of Theorem~\ref{drs}
based on the mixing property of the geodesic flow of a hyperbolic
manifold with finite volume. It may be noted that this approach for counting via mixing was
used earlier by Margulis in his 1970 thesis \cite{Margulisthesis}.
We also refer to \cite{BeO} for a quantitative version of Theorem \ref{drs}.
\medskip

The group $G$ can be considered as the group of orientation preserving isometries of the $n$-dimensional
hyperbolic space $\bH^n$.
The main achievement of this paper lies in extending Theorem \ref{drs}
to a suitable class of discrete subgroups $\G$ of infinite covolume in $G$; namely, the groups
$\G$ with finite Bowen-Margulis-Sullivan measure
$m^{\BMS}$ on $\G\ba \bH^n$. In particular, this class contains all geometrically finite subgroups of
$G$.
The analogue of $\op{vol}(\G_{w_0}\ba G_{w_0})$ turns out to be a very interesting
quantity, which we will call the `skinning size' of $w_0$ relative to $\Gamma$ and denote by
$\sk_\G(w_0)$. In fact, $\sk_\G(w_0)$ will be the total mass of a Patterson-Sullivan type measure on the
unit normal bundle of a closed immersed submanifold of $\Gamma\ba\bH^n$
associated to $G_{w_0}$.  One of the important components of this work is to completely
determine when $\sk_\G(w_0)$ is finite (Theorem \ref{skfinite}). In particular, 
$\sk_\G(w_0)<\infty$ for any geometrically finite $\G$ whose critical exponent $\delta$ is greater than the
codimension of the associated submanifold.

\medskip

The main ergodic theoretic ingredient in the proof is the description for the limiting
distribution of the evolution of the smooth measure on the unit normal bundle of a closed
 totally geodesically immersed submanifold of $\G\ba\bH^n$ under the
geodesic flow. The corresponding equidistribution statement (Theorem~\ref{mainergint}) is
applicable to many other problems; for example, in \cite{OhShahcircle,OhShahsphere}, it has
been applied to the study of the asymptotic distributions in circle packings in the Euclidean
plane or a sphere, invariant under a non-elementary
group of Mobius transformations.

\subsection{Statement of main result} Our generalization of Theorem \ref{drs} for discrete
subgroups
which are not necessarily lattices involves
terms which can be best explained in the language of the hyperbolic geometry.
Let $\G < G$ be a torsion-free
discrete subgroup which is non-elementary, that is, $\G$ has no abelian subgroup of finite
index.  This is a standing assumption on $\G$ throughout the whole paper.
Now $\G$ acts properly discontinuously on $\bH^n$. Let $0 <\delta\le n-1$
be the critical exponent of $\G$ (see \S\ref{subsec:criticalexp}).  Let
$\{\nu_x\}_{x\in \bH^n}$ be a $\G$-invariant conformal density of dimension $\delta$ on the
geometric boundary
$\partial{\bH^n}$ (see~\eqref{eq:conformal})
which exists by Patterson~\cite{Patterson1976} and Sullivan~\cite{Sullivan1979}.
 Let $m^{\BMS}$ denote the Bowen-Margulis-Sullivan measure on the unit tangent bundle
$\T^1(\G\ba \bH^n)$ associated to $\{\nu_x\}$ (see~\eqref{eq:defBMS}).

For $u\in \T^1(\bH^n)$, we denote by $u^{\pm}\in \partial{\bH^n}$ the forward and the
backward endpoints of the geodesic
determined by $u$ respectively and by $\pi(u)\in \bH^n$ the base point of $u$. Let
$\p:\T^1(\bH^n)\to \T^1(\G\bs\bH^n)$ be the canonical quotient map.

\medskip

Let $V$ be a finite dimensional vector space on which $G$ acts linearly. Let $w_0\in V$ be such
that $G_{w_0}$ is a symmetric subgroup or
the stabilizer $G_{\br w_0}$ of the line $\br w_0$ is a parabolic subgroup.
We define a subset $\tilde E\subset \T^1(\bH^n)$ associated to the orbit $w_0\G$ in each case.

When $G_{w_0}$ is a symmetric subgroup  associated to an involution $\sigma$,
choose a Cartan involution $\theta$ of $G$ which commutes with $\sigma$, and let $o\in \bH^n$ be such that its stabilizer 
$G_{o}$ is the fixed group of  $\theta$.
 Then $\tS:=G_{w_0}.o$ is an isometric imbedding of $\bH^k$ in $\bH^n$ for some $0\le k\le n-1$, where the embeddings of 
$\bH^0$ and $\bH^1$ mean a point and a complete geodesic respectively. Let $\tE\subset\T^1(\bH^n)$ be the unit normal
 bundle of $\tS$.

In the case when $G_{\br w_0}$
 is parabolic, we fix any $o\in\bH^n$.
If $N$ is the unipotent radical of $G_{\br w_0}$, then $\tS:=N. o$ is a horosphere.
We set $\tilde E\subset  \T^1(\bH^n)$  to be the unstable horosphere over $\tS$.

Now in either case,
we define the following Borel measure on $\tilde E$:
 $$d\mu^{\PS}_{\tilde E}(v):=e^{\delta \beta_{v^+}(x,\pi(v))}\, d\nu_x(v^+)$$
for $x\in\bH^n$ and  $\beta_
\xi(x_1,x_2)$ denotes the value of the Busemann function, that is,
the signed distance between the horospheres based at $\xi$, one passing through $x_1$ and the other through $x_2$ 
(see \eqref{eq:busemann}).
 This definition of $\mu^{\PS}_{\tilde E}$ is independent of the choice of $x\in \bH^n$. Due to the $\G$-invariance property of 
the conformal density $\{\nu_x\}$, it induces a measure on $E:=\p(\tilde E)$ which we denote by $\mu^{\PS}_{E}$.

Fix any $X_0\in \tilde E$ based at $o$, and let $A=\{a_r:r\in \br\}$ be a one-parameter subgroup of $G$ consisting of 
$\R$-diagonalizable elements such that $r\mapsto a_r.X_0$ is a unit speed geodesic. Note that $A$ is contained in a copy of 
$\SO(2,1)\cong\PSL(2,\R)$ such that $a_{r}$ corresponds to $d_{r}=\diag(e^{r/2}, e^{-r/2})$. Any irreducible representation of 
$\PSL(2,\R)$ is given by the standard action of $\SL(2,\R)$ on homogeneous polynomials of degree $k$ in two variables such that the action of $-I$ is trivial, so $k$ is even and the largest eigenvalue of $d_{r}$ is $e^{(k/2)r}$. Therefore if $\lambda$ denotes the $\log$ of the largest eigenvalue of $a_1$ on $\R\text{-}\Span(w_0 G)$, then $\lambda\in\N$. We set
\[
w_0^{\lambda}:=\lim_{r\to\infty}{e^{-\lambda r}} { w_0a_{r}}\neq 0,
\quad \text{by \cite[Lemma~4.2]{GorodnikOhShah2009}.}
\]

\begin{theorem} \label{m11}
Let $\G <G$ be a non-elementary discrete subgroup with $\abs{m^{\BMS}}<\infty$.
Suppose that $w_0\G$ is discrete and that its skinning size $\op{sk}_\G(w_0):=\abs{\mu_E^{\PS}}$ is finite. 
Then for any $G_o$-invariant norm  $\norm{\cdot}$ on $V$, we have

\begin{equation} \label{eq:m11}
\lim_{T\to\infty}\frac{\#\{w\in w_0\G: \|w\|<T\}}{T^{\delta/\lambda}}=\frac{\abs{\nu_o} \cdot
\op{sk}_\G(w_0) }
{\delta\cdot \abs{m^{\BMS}} \cdot \norm{w_0^\lambda}^{\delta/\lambda}}\;.
\end{equation}
\end{theorem}

\begin{remark}
(1)  If $\G$ is convex cocompact, $\sk_\Gamma(w_0)<\infty$. In the case when $G_{\br w_0}$ is
parabolic, $\sk_\Gamma (w_0)<\infty$ as well. A finiteness criterion for $\sk_\Gamma (w_0)$
is provided in the section 1.4.

(2) Since $w_0\G$ is infinite, $\sk_{\G}(w_0)>0$ (Proposition~\ref{nempty}), and hence the
limit \eqref{eq:m11} is strictly positive.

(3) The description of the limit changes if we do not assume the $G_o$-invariance of the 
norm $\norm{\cdot}$; see Theorem~\ref{thm:counting1}, 
Remark~\ref{rem:main}(\ref{itm:71}-\ref{itm:73}), and 
Theorem~\ref{thm:count-cone}.

(4) If $G_{w_0}$ is symmetric and $\G$ is Zariski dense in $G$,
 then the condition $\abs{\mu_{E}^{\PS}}<\infty$ implies that $w_0\G$ is discrete, for
 by Theorem~\ref{thm:finite-closed} and Remark~\ref{rem:proper}, $w_{0}\G$ is closed in 
$w_0G$, and by \cite[Lemma~4.2]{GorodnikOhShah2009}, $w_{0}G$ is closed in $V$. Therefore $w_{0}\G$ is closed and hence discrete
in $V$.

(5) If $G_{\R w_0}$ is parabolic, then the condition $\abs{\muPS_E}<\infty$ implies that  $w_0\G$ is discrete.
To see this, note that if the horosphere $\tS$ is based at $\xi$, then $\partial\tS=\{\xi\}$ and 
by Theorem~\ref{thm:finite-closed}, $\G\tS$ is closed in $\bH^{n}$ 
and $w_{0}\G$ is closed in $w_0G=\cl{w_0G}\setminus\{0\}$.
If $w_0\G$ were not closed in $V$,  $w_{0}\g_{i}\to 0$ for a sequence $\{\gamma_i\}\subset \G$. Then  $\g_{i} o\to \xi$
and $\xi$ is a horospherical limit point of $\G$. Since $\abs{m^{\BMS}}$ is finite, the geodesic flow is mixing (Theorem~\ref{thm:mixing}) and hence by \cite[Thm.A \& Prop.B]{Dalbo2000}, $\G\tS$ is dense in $\pi(\{u:u^-\in\Lambda(\G)\})$, a contradiction to $\G\tS$ being closed. Therefore $w_{0}\G$ is closed and hence discrete in $V$. 

Thanks are due to the referee for the last two remarks. 
\end{remark}

 A discrete group $\G$ is called {\em geometrically finite},
if the unit neighborhood of its convex core\footnote{The {\em convex core} $C_\G\subset \G\bs\bH^n$ of $\G$ is the image of the minimal convex subset of $\bH^n$ which contains all geodesics connecting any two points in the limit set of $\G$.} 
has finite Riemannian volume (see also Theorem~\ref{bo}).  Any discrete group  admitting a finite sided polyhedron as a fundamental domain in $\bH^n$ is geometrically finite.

Sullivan \cite{Sullivan1979} showed that $\abs{m^{\BMS}}<\infty$ for all geometrically finite $\G$.
However Theorem \ref{m11} is not limited to those,
as Peign\'e \cite{Peigne2003} constructed a large class of geometrically infinite groups admitting a finite Bowen-Margulis-Sullivan measure.

We will provide a general criterion on the finiteness of $\sk_\G(w_0)$ in Theorem~\ref{thm:parcorank}.
For the sake of concreteness,
we first describe the results for the standard representation of $G$.

\subsection{Standard representation of $G$.}
 Let $Q$ be a real quadratic form of signature $(n,1)$ for $n\ge 2$ and $G$ the identity
component of the
special orthogonal group $\SO(Q)$.
Then $G$ acts on $\br^{n+1}$ by the matrix multiplication from the right, i.e., the standard representation.
For any non-zero $w_0\in \br^{n+1}$, up to conjugation and commensurability,
$G_{w_0}$ is  $\SO(n-1,1)$ (resp.\ $\SO(n)$) if $Q(w_0)>0$ (resp.\ if $Q(w_0)<0$).
If $Q(w_0)=0$, the stabilizer of the line $\br w_0$ is a parabolic subgroup.
Therefore Theorem~\ref{m11} is applicable for any non-zero $w_0\in \br^{n+1}$, provided $\sk_\G(w_0)<\infty$
(in this case, $\lambda=1$).

An element $\g\in \G$ is called {\em parabolic\/} if there exists a unique fixed point of $\g$ in
$ \partial{\bH^n}$.  For $\xi\in \partial{\bH^n}$, we denote by $\G_{\xi}$ the stabilizer of $\xi$ in $\G$
and call $\xi$ a {\em parabolic fixed point\/} of $\G$ if $\xi$ is fixed by a parabolic element of
$\G$.

Noting that $G_{w_0}$ is the isometry group of the codimension one totally geodesic subspace, say,
 $\tilde S_{w_0}$, when $Q(w_0)>0$, we give the following:

\begin{definition}  \rm  Let $w_0\G$ be discrete.
Then $w_0\in \br^{n+1}$ is said to be {\em externally $\G$-parabolic\/} if $Q(w_0)>0$ and there exists a parabolic 
fixed point $\xi \in \partial\tilde{S}_{w_0}$ for $\G$ such that $G_{w_{0}}\cap \G_{\xi}$ is trivial,
where $\partial\tS_{w_0}\subset\partial{\bH^n}$ denotes the boundary of $\tilde S_{w_0}$ in $\cl{\bH^n}$.
\end{definition}

{\begin{figure}
\begin{center}
 \includegraphics[width=3cm]{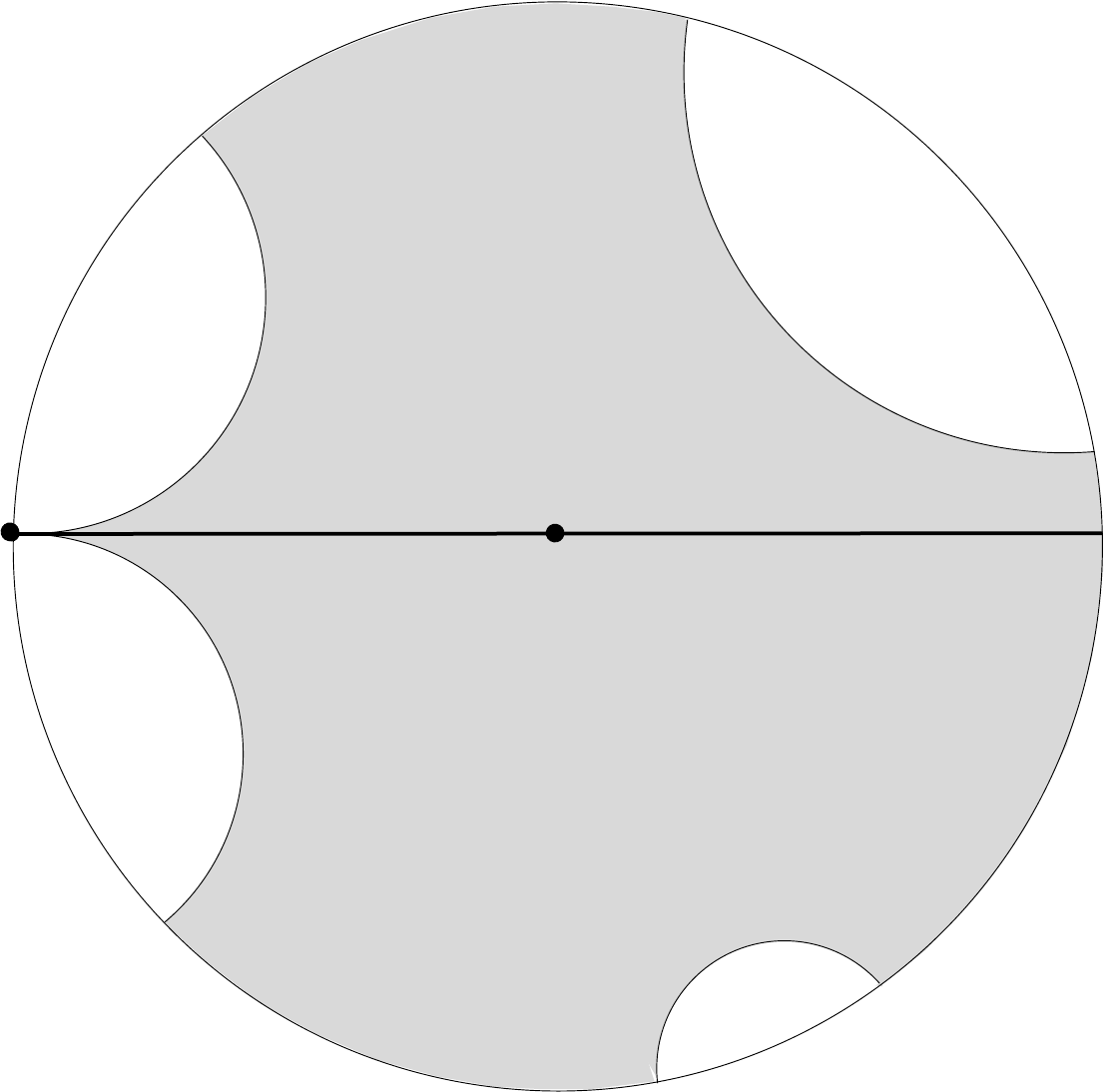}
    \caption{An externally $\G$-parabolic vector}
    \label{f1}
 \end{center}
\end{figure}}

For $n=2$, $w_0\in\R^3$ with $Q(w_0)>0$  is externally $\G$-parabolic if and only if the
projection of the geodesic $\tilde S_{w_0}$ in $\G\bs\bH^n$ is divergent in both directions, and at
least one end of $\tS_{w_0}$ goes into a cusp of a fundamental domain of $\G$ in $\bH^2$
(see Fig.~\ref{f1}).

\begin{theorem}[On the finiteness of $\sk_\G(w_0)$] \label{skfinite} Let $\G$ be geometrically finite and $w_0\G$ discrete.
\begin{enumerate}
\item If $\delta>1$, then $ \op{sk}_\G(w_0) <\infty.$

\item If $\delta \le 1$, then
$ \op{sk}_\G(w_0) =\infty $ if and only if $w_0$ is externally $\G$-parabolic.
\end{enumerate}
\end{theorem}

\begin{corollary}
Let $\G$ be geometrically finite and $w_0\Gamma $ discrete. If
either  $\delta>1$ or
 $w_0$ is not externally $\G$-parabolic, then \eqref{eq:m11} holds.
\end{corollary}

\begin{remark}
\rm (1) For geometrically finite $\Gamma$, if the Riemannian volume of $E$ is finite, then
$\op{sk}_\G(w_0)<\infty$ (Corollary~\ref{Cor:Leb:PS}).

 (2) It can be proved that if $\delta\le 1$ and $w_0$ is externally $\G$-parabolic,
the asymptotic count is of the order  $ T\log T$ if $\delta=1$
and of the order $T$ if $\delta<1$, instead
of $\T^{\delta}$ (cf.~\cite{OhShahlog}).

(3) When $Q(w_0)<0$, the orbital counting with respect to the hyperbolic metric balls was
obtained by Lax and
Phillips \cite{LaxPhillips} for $\G$ geometrically finite with $\delta>(n-1)/2$, by Lalley
\cite{Lalley1989}
for convex cocompact subgroups and by Roblin \cite{Roblin2003} for all groups with finite
Bowen-Margulis-Sullivan
measure.

(4) When $Q(w_0)=0$ and $\G$ is
geometrically finite with $\delta>(n-1)/2$, a version of Theorem \ref{m11} was
obtained in \cite{KontorovichOh}.
\end{remark}

\subsection{Equidistribution of expanding submanifolds} \label{subsec:Expanding}
In this section, we will describe the main ergodic theoretic ingredients used in the proof of
Theorem \ref{m11}. Let $\tE \subset \T^1(\bH^n)$ be one of the following:
\begin{enumerate}
\item an unstable horosphere over a horosphere $\tS$ in $\bH^n$;
\item the unit normal bundle of a complete proper connected totally geodesic subspace  $\tS$ of $
\bH^n$; that is,
 $\tS$ is  an isometric imbedding of $\bH^k$ in $\bH^n$ for some $0\le k\le n-1$.
\end{enumerate}

Let $\G$ be a discrete subgroup of $G$, and set $E:=\p(\tE)$ for the projection $\p:\T^1(\bH^n)\to \T^1(\Gamma\ba \bH^n)$.

 Recall that $\{\nu_x:x\in \bH^n\}$ denotes a Patterson-Sullivan
density of dimension $\delta$. Let $\{m_x:x\in \bH^n\}$ denote a $G$-invariant conformal density of dimension $(n-1)$.
We consider the following locally finite Borel measures on $\tilde E$:
$$d\mu^{\Leb}_{\tilde E}(v)=e^{(n-1)\beta_{v^+}(o,\pi(v))}dm_o(v^+),\quad
d\mu^{\PS}_{\tilde E}(v)=e^{\delta\beta_{v^+}(o,\pi(v))}d\nu_o(v^+),$$
where $o\in \bH^n$. Note that $\mu_{\tE}^{\Leb}$ is the measure associated to the Riemannian volume form on $\tE$.

The measures $\mu_{\tE}^{\PS}$ and $\mu_{\tE}^{\Leb}$ are invariant under 
$\G_{\tilde E}=\{\gamma\in \G: \gamma(\tilde E)=\tilde E\}$
and hence induce measures on $\G_{\tilde E}\ba \tilde E$.
We denote by $\mu_{E}^{\Leb}$ and $\mu_E^{\PS}$ respectively the projections of these measures on $E$ via
the projection map  $\G_{\tilde E}\ba \tilde E \to E$ induced by $\p$.

Let $m^{\BR}$ denote the Burger-Roblin
measure on   $\T^1(\G\ba \bH^n)$ associated to the conformal densities $\{\nu_x\}$ in the
backward
direction and $\{m_x\}$ in the forward direction (\cite{Burger1990}, \cite{Roblin2003}, see~
\eqref{eq:defBR}).

Let  $\{{\gt}\}$ denote the geodesic flow on $\T^1(\bH^n)$.

\begin{theorem}
\label{mainergint}
Suppose that $\abs{m^{\BMS}}<\infty$ and $\abs{\mu_E^{\PS}}<\infty$.
Let $F\subset E$ be a Borel subset with $\mu_E^{\PS}(\partial F)=0$.
 For any $\psi\in C_c(\T^1(\G\ba \bH^n))$,
\begin{equation}\label{tre}
\lim_{t\to +\infty} e^{(n-1-\delta)t}\cdot
\int_{F}\psi({\gt}(v))\; d\mu^{\Leb}_{E}(v)
=\frac{ \mu_{E}^{\PS}(F)}{\abs{m^{\BMS}}}
 \cdot m^{\BR} (\psi).
\end{equation} In particular, this holds for $F=E$.
\end{theorem}

See Theorem \ref{m2} for a version of Theorem~\ref{mainergint} without the finiteness assumption on
$\abs{\mu_E^{\PS}}$.

\begin{remark}\rm Theorem \ref{mainergint} applies to $F$ with $\mu_E^{\Leb}(F)=\infty$ as well,
provided $|\mu_E^{\PS}|< \infty$. The proof for this generality requires
greater care since it cannot be deduced from the cases of $F$ bounded.
 It is precisely this general nature of our equidistribution theorem
 which enabled us to state Theorem \ref{m11} for general groups $\G$ only
assuming the finiteness of the skinning size $\op{sk}_\G(w_0)=\abs{\mu_E^{\PS}}$ for a
suitable $E$.
\end{remark}

When $E$ is a horosphere and $F$ is bounded, Theorem \ref{mainergint} was obtained
earlier by Roblin \cite[P.52]{Roblin2003}.
We were motivated to formulate and prove the result from an independent view point; our
attention is especially on the case of $
\pi(E)$ being a totally geodesic immersion.   This case involves many new features,
observations, and applications (cf.~\cite{OhShahcircle}, \cite{OhShahsphere}).
The main key to our proof is the transversality theorem \ref{thm:transversal}, which was influenced by
the work of Schapira~\cite{Schapira2005}. The transversality
theorem provides a precise relation between the transversal intersections of geodesic
evolution of $F$ with a given piece, say
$T$, of a weak stable leaf and the transversal measure corresponding to the $m^{\BMS}$
measure on $T$.
\medskip

For $\G$ Zariski dense, we generalize Theorem \ref{mainergint} to $\psi\in C_c(\G\ba G)$.
To state the generalization, we fix $o\in \bH^n$ and $X_0\in \tilde E$ based at $o$. 
Then, for $K=G_o$ and $M=G_{X_0}$, we may identify $\bH^n$ and $\T^1(\bH^n)$ with
 $G/K$ and $G/M$ respectively. Let $A=\{a_r\}$ be the
one-parameter subgroup such that the right translation action by $a_r$ on $G/M$ corresponds to $\gr$.
Let $\mBR$ denote the measure on $\G\ba G$ which is
the $M$-invariant extension of $m^{\BR}$ via the natural projection map
 $\G\ba G\to \G\ba G/M=\T^1(\G\ba \bH^n)$. Let $H=G_{\tE}$, and let $dh$ denote the invariant
 measure on $\G_H\ba H$ whose projection to $E$ coincides with $\mu_E^{\Leb}$.

\begin{theorem} \label{thmi:mainerg-group}
Let $\G$ be a Zariski dense discrete subgroup of $G$ such that $\abs{m^{\BMS}}<\infty$ and
 $\abs{\muPS_E}<\infty$. 
Then for any $\psi\in C_c(\G\bs G)$,
\[
\lim_{r\to\infty} e^{(n-1-\delta)r} \int_{h\in \G_H\backslash H} 
\psi(\G h a_r)\,d h=\frac{\abs{\muPS_E}}{\abs{m^{\BMS}}}\mBR(\psi).
\]
\end{theorem}

\medskip

When $\Gamma$ is a lattice in $G$ and $E$ is of finite Riemannian volume, Theorem
\ref{thmi:mainerg-group} is due to Sarnak~\cite{Sarnak1981} for horocycles
in $\bH^2$, Randol~\cite{Randol1984} for unit normal vectors based at a point in the
cocompact lattice case in $\bH^{2}$, Duke-Rudnick-Sarnak~\cite{DukeRudnickSarnak1993}
and Eskin-McMullen~\cite{EskinMcMullen1993} in general (also see \cite[Appendix]
{KleinbockMargulis1996}).

In Section~\ref{cs}, we deduce Theorem \ref{m11} from Theorem \ref{mainergint}. The standard techniques of orbital counting 
via equidistribution results
require significant modifications  due to the fact that $m^{\BR}$ is not $G$-invariant.

\subsection{On finiteness of $\mu_E^{\PS}$ for geometrically finite $\Gamma$} \label{tg}
An important condition for
the application of Theorem~1.8 is
to determine when $\mu_E^{\PS}$ is finite.
In this subsection we assume that $\G$ is geometrically finite.
Letting $\tE$ and $\tilde S=\pi(\tilde E)$ be as in section \ref{subsec:Expanding},
{\em suppose further that the natural imbedding $\G_{\tS}\bs \tS\to \G\bs \bH^n$ is proper\/};
in particular, $\p(\tS)$ is
closed in $\G\bs \bH^n$, where
$\G_{\tS}=\{\gamma\in \Gamma:\gamma\tS=\tS\}$.

When $\tS$ is a point or a horosphere, $\mu_{E}^{\PS}$ is compactly supported (Theorem \ref{hstar}).

\begin{theorem}[Theorem~\ref{gff}] \label{Thm:closed:GeomFinite}
If $\tS$ is totally geodesic, then $\G_{\tS}$ is geometrically finite.
\end{theorem}

\begin{definition}[Parabolic-Corank] \label{def:par-corank}
Let $\Lambda_{\rmp}(\G)$ denote the set of parabolic fixed points of $\G$ in $\partial{\bH^n}$.
For any $\xi\in\Lambda_{\rmp}(\G)$, $\G_{\xi}$ is a virtually free abelian group of rank at least one.
Define
\begin{equation*}
 \parcorank(\G_{\tS})=\max_{\xi\in\Lambda_{\rmp}(\G)\cap \partial ( \tS)} \left(\rank(\G_\xi)-\rank(\G_
\xi
\cap \G_{\tS})\right).
 \end{equation*}
\end{definition}

If $\Lambda_{\rmp}(\G)\cap \partial ( \tS)=\emptyset$, we set $\parcorank(\G_{\tS})=0$.
In particular, the parabolic co-rank of $\G_{\tS}$ is always zero when
 $\G$ is convex cocompact.

\begin{lemma}[Lemma~\ref{corank}]
\label{lemma:par-corank} If $\tS$ is totally geodesic, then
 \[
 \parcorank(\G_{\tS})\leq \codim(\tS).
 \]
\end{lemma}

 \begin{theorem}[Theorems~\ref{cstar-sb} and  \ref{stfinite}]
\label{thm:parcorank}  We have:
\begin{enumerate}
 \item $\supp(\mu^{\PS}_E)$ is compact  if and only if $\parcorank(\G_{\tS})=0$.

 \item $\abs{\mu^{\PS}_E}<\infty$ if and only if $\parcorank(\G_{\tS})<\delta$.
\end{enumerate}
 \end{theorem}

Note that by \cite[Prop. 2]{DalboOtalPeign},
$\delta> \tfrac 12 \max_{\xi\in\Lambda_{\rmp}(\G)} \rank(\G_\xi)$.  As a consequence of Theorem~\ref{thm:parcorank}, 
we get:

\begin{corollary}[Corollary~\ref{infvol}]
\label{Cor:Leb:PS}
Suppose that $\dim(\tS)\ge (n+1)/2$.
If $\abs{\mu_E^{\Leb}}<\infty$, then
$\abs{\mu_E^{\PS}}<\infty$.
\end{corollary}

\subsection{Finiteness of $\muPS_E$ or $\mu^{\Leb}_E$ and closedness of $E$}

Let $\tE$ and $E$ be as in the subsection \ref{subsec:Expanding}. In \cite{Raghunathanbook}, it is shown that $
\abs{\mu_E^{\Leb}}<\infty$ implies
that $E$ is a closed subset of $\T^1(\Gamma\ba \bH^n)$.  We prove an analogous statement
for $\mu_E^{\PS}$.

\begin{theorem}[Theorem~\ref{thm:finite-closed}]
Let $\G$ be a discrete Zariski dense subgroup of $G$. If $\abs{\mu_{E}^{\PS}}<\infty$, then the natural embedding
$\G_{\tS}\bs\tS\to \G\bs \bH^n$ is proper.
 \end{theorem}

\subsection{Integrability of $\phi_0$ and a characterization of a lattice}
 Define $\phi_0\in C(\Gamma\ba \bH^n)$ by
$$\phi_0(x):=|\nu_x|\quad\text{ for $x\in \Gamma\ba \bH^n$}.$$
The function $\phi_0$ is an eigenfunction of the hyperbolic Laplace operator
with eigenvalue $-\delta (n-1-\delta)$ \cite{Sullivan1979}. Sullivan \cite{Sullivan1984}
 showed that if $\delta >\tfrac{n-1}2$, then
 $\phi_0\in L^2(\Gamma\ba \bH^n,d\Vol_{{\Riem}})$ if and only if $\abs{m^{\BMS}}<\infty$.
The following theorem, which is
 a new application of Ratner's theorem \cite{Ratner1991},
relates the integrability of $\phi_0$ with the finiteness of $\op{Vol}_{\rm Riem}(\G\ba \bH^n)$:

\begin{theorem}[\S\ref{subsec:eigen}] \label{thm:mbf}
For any discrete subgroup $\G$,
the following statements are equivalent:
\begin{enumerate}
\item \label{mbf3} $\phi_0\in L^1(\Gamma\ba\bH^n,d\Vol_{\Riem})$;
\item  \label{mbf2} $\abs{m^{\BR}}<\infty$; 
\item \label{mbf1} $\Gamma$ is a lattice in $G$.
\end{enumerate}
\end{theorem}

Although $m^{\BR}$ depends on the choice of the base point $o$, its finiteness is independent of the choice. If  $\G$ is a lattice, then $\delta=n-1$ and hence $\phi_0$ is a constant function by the uniqueness of the harmonic function~\cite{Yau1975}.

\subsection*{Acknowledgements}
We thank Thomas Roblin for useful comments on an earlier version of this paper. We thank
the referee for carefully reading the paper and asking many pertinent questions, which led us to proving more general results 
and improving the overall presentation of the paper. 

\section{Transverse measures}

\subsection{}
Let $(\bH^n, d)$ denote the hyperbolic $n$-space and $\partial{\bH^n}$ its geometric
boundary.
Let $G$ denote the identity component of the isometry group of $\bH^n$.
We denote by $\op{T}^1(\bH^n)$ the unit tangent bundle of $\bH^n$ and by $\pi$ the natural
projection from $\op{T}^1(\bH^n)\to \bH^n$. By abuse of notation, we use $d$ to denote a left
$G$-invariant  metric on
$\T^1(\bH^n)$ such that $d(\pi(u),\pi(v))=\min\{ d(u_1, v_1): \pi(u_1)=\pi(u), \pi(v_1)=\pi(v)\}$.
For a subset $A$ of $\T^1(\bH^n)$ or $\bH^n$ or $\partial{\bH^n}$ and a subgroup $H$ of
$G$, we denote by $H_A$ the stabilizer subgroup $\{g\in H: g(A)=A\}$ of $A$ in $H$.

Denote by $\{\gr: r\in \br\}$ the geodesic flow on $\T^1(\bH)$. For $u\in \op{T}^1(\bH^n)$,
 we set
 \begin{equation} \label{eq:u+-}
 u^+:=\lim_{r\to\infty}\gr(u)\quad\text{and}\quad
 u^-:=\lim_{r\to -\infty}\gr(u),
 \end{equation}
 which are the endpoints in $\partial{\bH^n}$ of the geodesic defined by $u$.
Note that
$(g(u))^\pm=g(u^\pm)$ for $g\in G$.
 The map $\V: \op{T}^1(\bH^n)\to \partial{\bH^n}$ given by $\V(u) =u^+$
 is called the {\it visual\/} map.

\subsection{}
\label{subsec:busemann}
The {\em Busemann function\/} $\beta:\bh\times \bH^n\times \bH^n\to \br$
 is defined as follows: for $\xi\in \bh$ and $x, y\in \bH^n$,
\begin{equation} \label{eq:busemann}
\beta_\xi(x,y)=\lim_{r\to \infty}d(x,\xi_r)-d(y,\xi_r)
\end{equation}
where $\xi_r$ is any geodesic ray tending to $\xi$ as $r\to \infty$; and the limiting value is
independent of the choice of the ray $\xi_r$. 

Note that $\beta$ is differentiable and invariant under isometries; that is, for $g\in G$ and $x, y\in \bH^n$, 
$\beta_\xi(x,y)=\beta_{g(\xi)}(g(x),g(y))$. 

For $u\in \op{T}^1(\bH^n)$,
 the {\em unstable horosphere\/} based at $u^-$ is the set
 $${\mathcal H}^+_u=\{v\in \op{T}^1(\bH^n): v^-=u^-, \beta_{u^-}(\pi(u),\pi(v))=0\},$$
and the {\em stable horosphere\/} based at $u^+$ is the set
 $${\mathcal H}_u^-=\{v\in \op{T}^1(\bH^n): v^+=u^+, \beta_{u^+}(\pi(u),\pi(v))=0\}.$$

The {\em weak stable manifold\/} corresponding to $u$ is
 \[
 \tilde W_u^s=\V^{-1}(u^+)=\{v\in \T^1(\bH^n): v^+=u^+\}.
 \]
 \begin{align} 
 \label{eq:expand}
\text{$v_{1},v_{2}\in \cH^{+}_{u}$, $r\in\R$} &\Rightarrow d(\gr(v_{1}),\gr(v_{2}))=e^{r}d(v_{1},v_{2}). \\
\label{eq:stable}
\text{$v_{1},v_{2}\in \tilde W_{u}^{s}$, $r\geq 0$} &\Rightarrow d(\gr(v_{1}),\gr(v_{2}))\leq d(v_{1},v_{2}). 
\end{align}

The image under $\pi$ of a stable or an unstable horosphere $\cH$ in $\Th$ based at $\xi$ is
called a {\em horosphere in $\bH^n$ based at $\xi$}. Hence
$\pi(\cH)=\{y\in \bH^n: \beta_\xi(x,y)=0\}$ for $x\in\pi(\cH)$.

\subsection{}\label{name}
Let $\tilde S$ be  one of the following: a horosphere or a complete connected totally geodesic submanifold of $\bH^n$ of 
dimension $k$ for $0\leq k\leq n-1$. Let $\tE\subset \T^1(\Hn)$ denote the unstable horosphere with $\pi(\tE)=\tS$ if $\tS$ is a 
horosphere, and the unit normal bundle over $\tS$ if $\tS$ is totally geodesic.

\begin{lemma} \label{lemma:diffeo}
The visual map $\V$ restricted to $\tE$ is a diffeomorphism onto $\partial (\Hn) \setminus
\partial (\tS)$.
\end{lemma}

\begin{proof}
The conclusion is obvious if $\tS$ is a point or a horosphere.

Now suppose that $\tS$ is a totally geodesic subspace of dimension $1\leq k\leq n-1$.
 Consider the upper-half space model for $\bH^n$:
\begin{equation} \label{eq:upper}
\bH^n=\{x+jy:x\in\R^{n-1},\,y>0,\,j=(0,\dots,0,1)\},
\end{equation}
and $\partial{\bH^n}\cong \R^{n-1}\cup\{\infty\}$.
Without loss of generality, we may assume that $\infty\in \partial (\tilde S)$ and hence
 $\partial\tS\setminus\{\infty\}$ is a $(k-1)$-dimensional affine subspace, say $F$, of $
\R^{n-1}$. For any $x\in \R^{n-1}\setminus L$, let $x_1$ be the orthogonal projection of $x
$ on $L$. Let $x_2=x_{1}+\norm{x-x_{1}}\cdot j\in \bH^{n}$. Let $v\in\T^{1}(\bH^{n})$ be the unit vector based at $x_2$ in the 
same direction as $x-x_1$. Then $v\in \tE$ and $v^+=x$. Now the conclusion of the lemma is straightforward to deduce.
\end{proof}

\subsubsection{Maps between $\tE$ and $\cH_{v}^{+}$}
\label{subsubsec:xiv}
For $v\in \T^1(\bH^n)$,
$-v$ is the vector with the same base point as $v$ but in the opposite direction.
For $v\in \tE$, let $\xi_v:\cH_v^+\setminus \V\inv(\partial \tS) \to \tE \setminus \{-v\}$ be the
map given by
\begin{equation} \label{eq:2.3.1}
\xi_v(u)=\V\inv(u^+)\cap \tE.
\end{equation}
Then $\xi_v$ is a diffeomorphism.
Its inverse,  $q_v: \tE\setminus \{-v\} \to \cH_v^+ \setminus \V\inv(\partial \tS)$ is the map given
by
\begin{equation} \label{eq:2.3.2}
q_v(w)=\V\inv(w^+)\cap \cH^+_v.
\end{equation}

\begin{proposition}  \label{prop:xiv}
There exist $C_1>0$ and $\epsilon_0>0$ such that:
\begin{enumerate}
\item \label{itm:xiv1} if $v, w\in \tE$ and $d(v,w)<\epsilon_0$, then
\[
 \abs{\beta_{w^+}(\pi(q_v(w)),\pi(w))}\leq d(q_v(w),w)  <C_1d(w,v);
 \]
\item \label{itm:xiv2} if $v\in \tE$ and $w\in\cH_v^+$ with $d(v,w)<\epsilon_0$, then
\[\abs{\beta_{w^+}(\pi(\xi_v(w)),\pi(w))}\leq d(\xi_v(w),w)<C_1d(v,w).
\]
\end{enumerate}
\end{proposition}

\begin{proof} In each of the two statements, the first inequality follows directly from the
definition of Busemann function, so  we only need to prove the second inequality.

Consider the upper half space model of $\bH^n$ given by \eqref{eq:upper}. By applying an
isometry $g\in G$,
since $q_{g(v)}(gw)=g (q_v(w))$, we may assume that $v$ is the unit vector based at $j$ so that $v^+=\{\infty\}$.

Since $f(u):=d(q_v(u),u)$ is a differentiable function of $u\in \tE$, there exist $\epsilon_0>0$
and $C_1'>0$ such that  $\|{Df(u)}\|\leq C_1'$ for any $u$ with $d(v,u)<\epsilon_0$.
Therefore, since $f(v)=0$, there exists $C_1>0$ such that
$\abs{f(u)}=\abs{f(u)-f(v)}\leq C_1\cdot d(v,u)$ for all  $u\in\tE$ with $d(u,v)<\epsilon_0$.
This proves \eqref{itm:xiv1}. And \eqref{itm:xiv2} can be proved similarly.
\end{proof}

\begin{remark}
The following stronger form of statements in Proposition~\ref{prop:xiv} holds:
{\em There exist
$\epsilon_0>0$ and $C_1>0$ such that
\[
\abs{\beta_{w^+}(\pi(q_v(w)),\pi(w))}\leq C_1d(v,w)^2, \text{ for all $w\in \tE$ with $d(w,v)<
\epsilon_0$;}
\]
\[
\abs{\beta_{w^+}(\pi(\xi_v(w),\pi(w)))}\leq C_1d(v,w)^2, \text{ for all
$w\in \cH_v^+$ with $d(w,v)<\epsilon_0$.}
\]
}
We omit a proof as the stronger version will not be used in this article.
\end{remark}

\begin{notation}
Let $\G$ be a non-elementary torsion-free discrete subgroup of $G$ and set $X:=\G\ba \bH^n$.
Both the natural projection maps $\bH^n\to X$ and $\op{T}^1(\bH^n)\to \op{T}^1(X)$
will be denoted by $\mathbf p$.
\end{notation}

\subsection{Boxes, Plaques and Transversals} \label{subsec:box} Let $u\in \T^1(\bH^n)$.
Consider a relatively compact open set $P$ containing $u$ in $\cH^+_u$, and a relatively
compact open neighborhood $T$ of $u$ in $\V^{-1}(u^+)$. For each $t\in T$ and $p\in P$,
the horosphere $\cH^+_t$ intersects $\V^{-1}(p^+)$ at a unique
vector: we define
\[
tp:=\cH^{+}_{t}\cap \V\inv(p^{+})\in T^1(\bH^n).
\]
The map $(t,p)\to tp$ provides a local chart of a neighborhood of $u$ in $\T^1(\bH^n)$. Since $u\in P$, in this notation $tu=t$. 
We call the set
\[
B(u)=\{tp\in \Th: t\in T, p\in P\}
\]
a {\em box around $u$}  if some neighborhood of $B(u)$ {\em
injects\/} into $\T^1(X)$ under $\p$. We write $B=B(u)=TP$.

Note that  $P$ (resp. $T$) may be disconnected and of `large' diameter, and then the corresponding $T$ (resp. $P$) will be 
chosen to be of small diameter in order to achieve the required injectivity of $\p$ on a neighborhood of $B(u)$.

For any $t\in T$, the set
\[
tP:=\{tp:p\in P\}\subset \cH^+_t
\]
 is called a {\em plaque at $t$}; and for any $p\in P$, the set
 \[
 Tp:=\{tp:t\in T\}\subset  \V^{-1}(p^+)
 \]
 is called a {\em transversal at $p$}. The {\em holonomy map\/} between the transversals
 $Tp$ and $T{p'}$ is given by $tp\mapsto t p'$ for all $t\in T$.

\begin{remark} \label{rem:recenter}
If $v=tp\in B$, then $tP\subset \cH^{+}_{v}$, $Tp\subset \V\inv (v^{+})$ and
$B(v)=(Tp)(tP)$ is a box about $v$ and $TP=(Tp)(tP)$.  Also $B(u)$ and $B(v)$ have the same collections of plaques and 
transversals.
\end{remark}

For small $\epsilon>0$, let 
\begin{gather*}
T_{\epsilon+}=\{s\in \V \inv (u^+):d(s,T)<\epsilon\},\\
T_{\epsilon-}=\{t\in  T: d(t,\partial T)>\epsilon\}, \text{ and } 
B_{\epsilon\pm}=T_{\epsilon\pm}P.
\end{gather*}

Note that for any $\g \in G$, $\g P\subset \cH^+_{\g u}$, $\g T \subset \V\inv((\g u)^+)$, $\g (tp)=(\g t)(\g p)$ for any 
$(t,p)\in T\times P$, $\g(TP)=\g(B(u))=B(\g u)=(\g T)(\g P)$, $\g(tP)$ is a plaque at $\g t$ and $\g(Tp)$ is a transversal at 
$\g p$. Also $\g B_{\e\pm}=(\g T)_{\pm}(\g P)$.

For $r\in\R$, $\gr(B(u))=B(\gr (u))=(\gr(T))(\gr(P))$.

\subsection{} \label{subsec:2.6} For the rest of this section, let $B=TP\subset \T^1(\bH^n)$ denote a box such that $B_{\e_0+}$ 
injects into $\T^1(X)$ for some $\e_0>0$.  By choosing a smaller $\e_0$ if necessary, let  $C_1>0$ be such that 
Proposition~\ref{prop:xiv} holds. Let
\begin{equation} \label{eq:C2}
C_2=\max\{d(tp_{1}, tp_{2}):t\in T_{\epsilon_0+},\, p_{1},p_{2}\in P\}.
\end{equation}

In this section we will develop auxiliary results to understand the intersection of $\gr(E)$ with $\p(B)$ for $r\gg1$. First we will 
show that for any $\g\in \G$ if $\gr(\g \tE)\cap B$ is nonempty, then there exists a unique $t\in T_{\e_{0}+}\cap \gr(\g \tE)$ and 
the sets $\gr(\g \tE)$ and $\gr(tP)$ are contained in $C_1C_2e^{-r}$-tubular neighborhoods of each other.

%

\begin{lemma}\label{lemma:wer}
Let $r\in\R$ and $\g\in\G$. Suppose that $\gr(\g\tE)\ni tp$ for some $t\in T$, $p\in P$. Let $v=\grinv(\g\inv tp)\in \tE$. Let 
$p_{1}\in P$, $y=\grinv(\gamma\inv tp_{1})\in \cH^+_v$, and $w=\xi_v(y)\in \tE$. Then $w^{+}=y^{+}$,
\[
d(v,y) \le C_2 e^{-r} \text{ and } d(y,w)\leq C_1C_2 e^{-r}.
\]
\end{lemma}

\begin{proof}
 By \eqref{eq:2.3.1}, $w^+=y^+$. Since  $tp,tp_{1}\in \cH_{t}^+$, by \eqref{eq:expand} and \eqref{eq:C2},
 \begin{align*}
 d(v,y)=d(\grinv(\g\inv tp), \grinv(\g\inv tp_{1})) &=d(\grinv(tp),\grinv(tp_{1})) \\
 & \leq d(tp,tp_{1})e^{-r}\leq C_2e^{-r}.
 \end{align*}
 By Proposition~\ref{prop:xiv},
$d(y,w)=d(y,\xi_v(y))\leq C_1d(v,y)\leq C_1C_2e^{-r}$.
 \end{proof}

\begin{lemma}\label{atmost}
For any $r\in\R$ and $\g\in \G$,
\[
\#\bigl(T\cap \gr(\g \tE)\bigr) = \#(\grinv(\g\inv T)\cap \tE) \le 1.
\]
\end{lemma}

\begin{proof}
Since $\V ({\grinv}(\g\inv T))=\g\inv \V(T)$ is a singleton set and $\V$ restricted to $\tE$ is
injective, the conclusion follows.
 \end{proof}

\begin{notation} \label{note:Erg} For $r\in\R$ and $\g\in\G$, in view of Lemma~\ref{atmost}, define
\begin{equation} \label{eq:Ert}
\tE_{r,\g}=\begin{cases} \xi_{{\grinv}(\g\inv t)} ({\grinv}(\g\inv tP))\subset \tilde E & \text{if $T\cap \gr(\g \tE)=\{t\}$}\\
\emptyset & \text{if $T\cap \gr(\g \tE)=\emptyset$.}
\end{cases}
\end{equation}
\end{notation}

\begin{proposition} \label{prop:eq9eq10} 
For any $0<\e\le \e_0$, $r>r_{\e}:=\log(C_1(C_{1}+1)C_2/\e)$ and $\g\in \G$, we have
 \begin{equation}
\label{eq:910}
{\grinv} (\g\inv B_{\epsilon-}) \cap \tE\subset \  \tE_{r,\g}
\subset \
{\grinv}(\g\inv B_{\epsilon+})\cap  \tE.
\end{equation}
\end{proposition}

\begin{proof}[Proof of first inclusion in \eqref{eq:910}] Let $\g\in \G$ and $t\in T_{\e-}$ and $p\in P$ be such that 
$v:=\grinv(\g\inv tp)\in \tE$. 
Let $y=\grinv(\g\inv t)$ and $w=\xi_v(y)\in \tE$. 
By Lemma~\ref{lemma:wer},
\[
d(y,w)\leq C_1C_2e^{-r}<\e/(C_{1}+1)<\e.
\]
Let $t_1=\gr(\g w)$. Since $t=\gr(\g y)$ and $w^{+}=y^{+}$,  $t_{1}^{+}=t^{+}$. By \eqref{eq:stable}, 
\[
d(t,t_{1})=d(\gr(\g y)), \gr(\g w))\leq d(\g y,\g w)=d(y,w)<\e.
\]
Therefore $t_1\in T$, for $t\in T_{\e-}$. Since $(t_1p)^+=(tp)^+$, we have
\[
\grinv(\g\inv t_{1}p)^{+}=\grinv(\g\inv tp)^{+}=v^{+}.
\]
Since $w=\grinv(\g\inv t_{1})$, $\grinv(\g\inv t_{1}p)\in \cH^{+}_{w}$. Also $w, v\in\tE$. Therefore by \eqref{eq:2.3.1},
\[
v=\xi_w({\grinv}(\g\inv t_1p))\in \tE_{r,\g}.
\]
\end{proof}

\begin{proof}[Proof of  second inclusion in \eqref{eq:910}] By Lemma~\ref{atmost}, let
$\{t\}= T\cap \gr(\g \tE)$ for some $\g\in\G$. Let $v=\grinv(\g\inv t)\in\tE$, $p\in P$,
$y=\grinv(\g\inv tp)$, and $w=\xi_v(y)\in \tilde E_{r,\g}$. By Lemma~\ref{lemma:wer},
\[
d(v,w)\leq d(v,y)+d(y,w)\leq C_2e^{-r}+C_1C_2e^{-r}\leq \e/C_{1}.
\]
Put $v_{1}=q_{w}(v)\in \cH^{+}_{w}$. By \eqref{eq:2.3.2}, $v_{1}^{+}=v^{+}$, and by
Proposition~\ref{prop:xiv}(\ref{itm:xiv1}),
\[
d(v,v_{1})\leq C_{1}d(v,w)\leq\e.
\]
Put $t_{1}=\gr(\g v_{1})\in \cH^{+}_{\gr (\g w)} $. Since  $t=\gr(\g v)$, $t_{1}^{+}=t^{+}$. By \eqref{eq:stable},
\[
d(t,t_{1})=d(\gr(\g v), \gr(\g v_{1})) \leq d(\g v, \g v_{1})\leq \e.
\]
Hence $t_{1}\in T_{\e+}$. Now $\gr(\g w), t_{1}p\in \cH^{+}_{t_{1}}$. Since $w^{+}=y^{+}$, 
\[
(\gr(\g w))^{+}=(\gr(\g y))^{+}=(tp)^{+}=(t_{1}p)^{+}.
\]
Since $\V$ is injective on $\cH^+_{t_1}$, $\gr (\g w)=t_{1}p$. Hence $w\in \grinv(\g \inv B_{\e+})$.
\end{proof}

\subsection{Measure on $E$ corresponding to a conformal density on $\partial{\bH^n}$}
\label{subsec:muE}

Let $\{\mu_x:x\in \bH^n\}$ be a {\em $\G$-invariant conformal density\/}
 of dimension $\delta_\mu > 0$  on $\partial{\bH^n}$. That is, for each $x\in\bH^n$,
$\mu_x$ is a positive finite Borel measure on $\partial{\bH^n}$
such that for all $y\in \bH^n$, $\xi\in \partial{\bH^n}$ and $\gamma\in
\G$,
\begin{equation}\label{eq:conformal}
\gamma_*\mu_x=\mu_{\gamma x}\quad\text{and}\quad
 \frac{d\mu_x}{d\mu_y}(\xi)=e^{\delta_\mu \beta_{\xi} (y,x)},
\end{equation}
where $\gamma_*\mu_x(F):=\mu_x(\gamma^{-1}(F))$ for any Borel subset $F$ of $
\partial{\bH^n}$.

Fix $o\in\bH^n$. We consider the measure on $\tE$ given by
\begin{equation} \label{eq:mutE}
d \mu_{\tE}(v)=e^{\delta_\mu \beta_{v^+}(o, \pi(v))}\,d\mu_o(v^+).
\end{equation}
By \eqref{eq:conformal}, $\mu_{\tE}$ is independent of the choice of $o\in \bH^n$ and
$\gamma_*\mu_{\tE}=\mu_{\gamma\tE}$ for any $\gamma\in \G$. Let $\mu_{\G_{\tE}\bs\tE}$ be the locally finite 
Borel measure on $\G_{\tE}\bs\tE$ induced by $\mu_{\tE}$ as follows: For any
$f\in C_c(\tE)$,  let $\bar f(\G_{\tE} v)=\sum_{\gamma\in\G} f(\gamma v)$, for all $v\in \tE$.
Then $f\mapsto\bar f$ is a surjective map from $C_c(\tE)$ from to $C_c(\G_{\tE}\bs\tE)$, and
\begin{equation} \label{eq:muEbar}
\int_{\G_{\tE}\bs \tE} \bar f \,d\mu_{\G_{\tE}\bs\tE}:=\int_{\tE} f \,d\mu_{\tE}
\end{equation}
is well defined; see~\cite[Chapter 1]{Raghunathanbook} for a similar construction.

Now let $\mu_E$ be the measure on $E=\p(\tE)$ defined as the pushforward of
$\mu_{\G_{\tE}\bs\tE}$ from $\G_{\tE}\bs\tE$ to $\T^1(\G\bs\bH^n)$ under the map
$\G_{\tE} v\mapsto \G v$.  Thus
for any set $B\subset \T^{1}(\bH^{n})$ such that $\p$ is injective on $B$, and any measurable
nonnegative function $f$ on $E\cap\p(B)$,
\begin{equation} \label{eq:muPSE}
\begin{array}{ll}
\int_{E\cap\p(B)} f \, d\mu_{E}
&=\sum_{[\g]\in \G/\G_{\tE}} \int_{u\in \g\tE\cap B}f(\p(u))\,d\mu_{\g\tE}(u)\\
&=\sum_{[\g]\in \G/\G_{\tE}} \int_{u\in \tE\cap \g\inv B}f(\p(u))\,d\mu_{\tE}(u),
\end{array}
\end{equation}
where the integration over an empty set is defined to be $0$. Therefore by
Proposition~\ref{prop:eq9eq10} we obtain the following:

\begin{proposition} \label{prop:integral}
Let $0<\e\leq \e_{0}$ and $r>r_\e$. Then for all borel
measurable functions $\Psi\geq 0$ on $\T^{1}(X)$ with $\supp(\Psi)\subset \p(B_{\e-})$
and $f\geq 0$ on $E$, we have
\begin{equation*}
\label{eq:prop:muE}
\begin{array}{r}
\int_{u\in E} \Psi(\gr(u))f(u)\; d\mu_{E}(u)
=\int_{E\cap \p(\grinv(B_{\e\pm}))} \Psi(\gr(u))f(u)\,d\mu_{E}(u)\\
=\sum_{[\g]\in\G/\G_{\tE}} \int_{\grinv(\g\inv B_{\e\pm})\cap\tE} \Psi(\gr(u))f(u)\,d\mu_{\tE}(u)
\\
=\sum_{[\g]\in \G/\G_{\tE}} \int_{\tE_{r,\g}} \Psi(\gr(\p(u))) f(\p(u))
\,d\mu_{\tE}(u).
\end{array}
\end{equation*}
\end{proposition}

\begin{remark}
(1) For the counting application in section \ref{cs}, we will use the results of this section only
for the case when the
map $\G_{\tE}\bs \tE\to \T^1(\G\bs \bH^n)$ is proper, in which case $\mu_E$ is a locally finite
Borel measure.

(2) In the general case, $\mu_E$ may not be $\sigma$-finite, but it is an
{\em $s$-finite measure\/}; namely, a countable sum of finite measures (with possibly non-disjoint supports).

(3) If the dimension of $\tS=\pi(\tE)$ in $\bH^n$ is $0$ or $n-1$, the
map $\G_{\tE}\bs\tE$ to $\T^1(\G\bs \bH^n)$ is {\em injective\/}, and hence $\mu_E$ is $
\sigma$-finite on $\T^1(\G\bs \bH^n)$.
\end{remark}

\subsubsection{Measures on horospherical foliation and semi-invariance under geodesic flow}
\label{subsec:horosphere}
The conformal density $\{\mu_x\}$ induces a $\G$-equivariant
family of measures $\{\mu_{\cH_u^+}:u\in\T^1(\bH^n)\}$ on the unstable horospherical
foliation on $\T^1(\bH^n)$:

\begin{equation} \label{eq:muH}
d\mu_{\cH^+_u}(v)=e^{\delta_\mu \beta_{v^+}(o,\pi(v))}\,d\mu_o(v^+).
\end{equation}

For any $r\in \br$, since $\gr(v)^+=v^+$ and
\[
\beta_{v^+}(o,\pi(\gr(v)))-\beta_{v^+}(o,\pi(v))=\beta_{v^+}(\pi(v),\pi(\gr(v)))=r,
\]
by \eqref{eq:conformal}, we get for all $\g\in \G$ and $r\in\R$,
\begin{equation} \label{eq:conf}
\g_{\ast}\mu_{\cH^{+}_{u}}=\mu_{\cH^{+}_{\g u}}  \quad \text{and} \quad
\gr_{\ast}\mu_{\cH^+_u} = e^{-\delta_\mu r} \mu_{\cH^+_{{\gr}(u)}}.
\end{equation}

\subsection{On transversal intersections of $\gr(\G\tE)$ with $B$}
\label{subsec:transversal}
Let a box $B$, $\e_0>0$, $C_1>0$ and $C_2>0$ be as described in the beginning of
\S\ref{subsec:2.6}.
For any $0<\e\leq\e_0$, we put
\begin{equation} \label{eq:re}
r_\e=\log((C_1+1)C_1C_2/\e).
\end{equation}

\begin{proposition}\label{prop:eq11}
Let  $0<\e\leq \e_{0}$, $r>r_\e$, and $\{t\}=T\cap \gr(\g \tE)$ for some $\g\in\G$. Then for all measurable functions 
$\Psi\geq 0$ on $B_{\e_{0}+}$ and $f\geq0$ on $\tE$,
\begin{align*}
&(e^{-\delta_{\mu}\epsilon})  f_\e^-({\grinv}(\g\inv t))  \int_{tP}\Psi_\e^-\,d\mu_{\cH^+_t}
\notag\\
\le & e^{\delta_\mu r} \int_{w\in \tE_{r,\gamma}} \Psi(\gr(\g w)) f(w) d\mu_{\tE}(w)\\
\le &(e^{\delta_{\mu}\epsilon})  f_\e^+({\grinv}(\g\inv t))
\int_{tP}\Psi_\e^+\,d\mu_{\cH^+_t},
\end{align*}
where $f_{\e}^{\pm}$ on $\tE$ and $\Psi_\e^{\pm}$ on $B_{\e+}$ are defined as
\begin{equation}
\label{fdef}
\begin{array}{lcl}
f_{\e}^+(u)&= &\sup_{\{u_1\in \tE:d(u_1,u)\le \e\}}f(u_1),\\
f_{\e}^-(u)&=&\inf_{\{u_1\in \tE:d(u_1,u)\le \e\}} f(u_1),\\
\Psi_{\e}^+(tp)&=&\sup_{\{t_{1}\in T_{\epsilon+}:d(t_1p,tp)\le \e\} }\Psi (t_1p),\\
\Psi_{\e}^-(tp)&=&\inf_{\{t_1\in T_{\epsilon+}:d(t_1p,tp)\le \e\}} \Psi(t_1p).
\end{array}
\end{equation}
\end{proposition}

\begin{proof}
Let $v={\grinv}(\g\inv t)\in \tE$. Let $\phi:tP\subset\cH^{+}_{t}\to\tE_{r,\g}\subset\tE$ be the map given by 
$\phi(tp)=w:=\xi_{v}(y)$, where $p\in P$ and $y=\grinv(\g\inv tp)$.
By Lemma~\ref{lemma:wer},
\begin{equation} \label{eq:dvw}
d(y,w)<C_{1}C_{2}e^{-r}<\e, \quad d(v,w)<(C_{1}+1)C_{2}e^{-r}<\e
\end{equation}
and since $w^{+}=y^{+}$,
\[
d(\gr(\g y),\gr(\g w))=d(\gr(y),\gr(w))\leq d(y,w) < \e,
\]
and by Proposition~\ref{prop:eq9eq10}, $\gr(\g w)\in T_{\e+}p$.  Therefore
 \begin{gather}
 \label{eq:f+-}
 f_{\e}^{-}(v) \le f(w)\le f_{\e}^+(v); \\
 \Psi^{-}_{\e}(\gr(\g y)) \leq \Psi(\gr(\g w))\leq \Psi_{\e}^+(\gr(\g y)) \label{eq:Psi+}.
 \end{gather}

For the map $tp\mapsto y:=\grinv(\g\inv tp)$, by \eqref{eq:conf},
\begin{equation} \label{eq:growth}
e^{\delta_\mu r} d\mu_{\cH^+_{v}}(y)=d\mu_{\cH^+_{t}}(tp).
\end{equation}

For the map $y\mapsto w=\xi_{v}(y)$, by \eqref{eq:mutE} and \eqref{eq:muH}, since $w^{+}=y^{+}$,
 \begin{equation} \label{eq:tE-H}
 d\mu_{\tE}(w) =\frac{e^{\delta_\mu \beta_{w^+}(o,
\pi(w))}}{e^{\delta_\mu \beta_{y^+}(o,\pi(y))}}{d\mu_{\cH^+_{v}}(y)}
=e^{\delta_\mu \beta_{w^+}(\pi(y),\pi(w))}{d\mu_{\cH^+_{v}}(y)}.
 \end{equation}
By  \eqref{eq:dvw}, $\abs{\beta_{w^+}(\pi(y), \pi(w))}\leq d(\pi(y),\pi(w))\leq \e$. Therefore
\begin{equation} \label{eq:ratio}
e^{-\delta_{\mu}\e}
<d\mu_{\tE}(w)/d\mu_{\cH^+_{v}}(y) < e^{\delta_{\mu}\e}.
\end{equation}

Combining \eqref{eq:growth} and \eqref{eq:ratio}, for the map $w=\phi(tp)$ we get
\begin{equation} \label{eq:jacobian}
e^{-\delta_{\mu}\e} \leq e^{\delta_\mu r}\,\frac{ d\mu_{\tE}(w)}{d\mu_{\cH^{+}_{t}}(tp)}
\leq e^{\delta_{\mu}\e}.
\end{equation}
By noting that $\grinv(\g\inv t)=v$ and $tp=\gr(\g y)$, the conclusion of the proposition follows from \eqref{eq:f+-}, 
\eqref{eq:Psi+} and \eqref{eq:jacobian}.
\end{proof}

\newcommand{\Drt}{\bar\G_{r,t}}
\newcommand{\Nrt}{\#(\Drt)}

\begin{notation} \label{notn:Drt}
For $r\geq 0$ and $t\in T\cap \gr(\G \tE)$, in view of Lemma~\ref{atmost} let
\begin{equation} \label{eq:Nrt}
\Drt=\{[\g]\in \G/\G_{\tE}:\{t\}=T\cap \gr(\g  \tE)\}.
\end{equation}
\end{notation}

Since $\p$ is injective on $B_{\e_0+}$, for notational convenience,
 we identify $t\in T_{\e_0+}$ with its image $\p(t)\in\p(T)\subset X$. Therefore we have
\begin{equation} \label{eq:Drt}
\{[\g]\in \G/\G_{\tE}:\tE_{r,\g}\neq\emptyset\}=\bigcup_{t\in T\cap\gr(\G\tE)} \Drt
= \bigcup_{t\in T\cap \gr(E)} \Drt.
\end{equation}

Combining  Proposition~\ref{prop:integral} and Proposition~\ref{prop:eq11}, in view of \eqref{eq:Drt} we deduce the following:

\begin{corollary}\label{psie} Let $0<\e\leq \e_0$ and $r>r_\e$. For all measurable functions
$\Psi\geq 0$ on $B_{\e_{0}+}$ with $\supp(\Psi)\subset B_{\e-}$ and $f\geq0 $ on $E$, we have
\begin{align*}
&(e^{-\delta_\mu \e})\sum_{t\in T\cap \gr(E)}  \Nrt f_\e^-({\grinv}(t))\cdot \int_{tP}
\Psi_\e^{-}  \;d\mu_{\cH^+_t}\\
\le &e^{\delta_\mu r} \int_{E} \Psi(\gr(u))f(u) \;d\mu_{E}(u) \\
\le & (e^{\delta_{\mu}\e}) \sum_{t\in T\cap \gr(E)} \Nrt \cdot f_\e^+({\grinv}(t)) \cdot \int_{tP}
\Psi_\e^{+}  \;d\mu_{\cH^+_t},
\end{align*}
where $f_\e^\pm$ on $B_{\e+}$ and $\Psi_\e^\pm$ on $E$ are defined as in \eqref{fdef}.
\end{corollary}

\subsection{Haar system and admissible boxes}

\begin{lemma}[\cite{Roblin2003}]
 \label{haar}
For a uniformly continuous $\Psi\in C(B)$, the map
$$t\in T\mapsto \int_{tP}\Psi \, d\mu_{\cH^+_t}$$
is uniformly continuous. In particular the map $t\mapsto \mu_{\cH^{+}_{t}}(tP)$ is uniformly continuous.
\end{lemma}

\begin{proof}
Note that $(tp)^+=p^+$. Therefore by \eqref{eq:muH}
\[
\int_{tP} \Psi \, d\mu_{\cH^+_t}=\int_P \Psi(tp)
e^{\delta_\mu\beta_{p^+}(o,\pi(tp))}\,d\mu_o(p^+).
\]
Put $\phi(tp)=\Psi(tp) e^{\delta_\mu\beta_{p^+}(o,\pi(tp))}$. Since $\phi$ is uniformly
continuous on $B$,
\begin{align*}
\abs{\int_{t_1P} \Psi \, d\mu_{\cH^+_{t_1}}-\int_{t_2P} \Psi \, d\mu_{\cH^+_{t_2}}}
\leq \mu_o(\V(P))\cdot\sup_{p\in P}\abs{\phi(t_1p)-\phi(t_2p)}\to 0
\end{align*}
as $d(t_1,t_2)\to 0$.
\end{proof}

\begin{definition}  \label{def:admissible}
A box $B=TP$ as defined in subsection~\ref{subsec:box} is called
{\em admissible with respect to the conformal density $\{\mu_x\}$}, if every plaque of $B$ has a positive measure 
with respect to $\{\mu_{\cH^+}\}$; that is, $\mu_{\cH^+_{t}}(tP)>0$ for all $t\in T$, or equivalently,
\[
\mu_{x}(\V(tP))=\mu_{x}(P^+)>0 \text{ for some (and hence all) $x\in\bH^n$}.
\]
\end{definition}

\begin{lemma} \label{lemma:adm} Fix a conformal density $\{\mu_{x}\}_{x\in\bH^{n}}$ on $\partial\bH^{n}$. Then for any 
$u\in \op{T}^1(\bH^n)$, there exists an admissible box  around $u$ with respect to $\{\mu_{x}\}$.
\end{lemma}

\begin{proof} Fix any $x\in\bH^{n}$. Since $\G_{u^-}$ is virtually abelian, and since we assume that $\G$ is non-elementary, 
$\G$ does not fix $u^-$. Therefore by the $\G$-invariance and the conformality of the density $\{\mu_x\}$, we have
$\supp(\mu_{x})\neq \{u^{-}\}$.  Since $\V:\cH^+_u \to \partial{\bH^n}\setminus \{u^-\}$ is a diffeomorphism, there exists 
$u_{1}\in \cH^{+}_{u}$ such that $u_{1}^{+}=\V(u_{1})\in\supp(\mu_{x})$. If $\g u=u_1$ for any $\g\in \G$, then by the 
conformality, $u\in\supp(\mu_{\cH^+_u})$ and we replace $u_1$ by $u$. Since $\p$ is injective on $\{u,u_1\}$, there exists a 
relatively compact open subset $P$ of $\cH_u^+$ containing  $\{u,u_1\}$ such that $\p$ is injective on an open set 
$\Omega$ of $\T^1(\bH^n)$ containing $\cl{P}$. Then $\mu_x(\V(P))>0$. By Lemma \ref{haar},
we can choose $T$ a enough ball in $\V\inv(u^{+})$ so that some neighborhood of the closure of $B=TP$ is contained in 
$\Omega$. Now $B=TP$ is an admissible box.
\end{proof}

\subsubsection{}  \label{subsec:281} 
Let $B=TP$ be an admissible box with respect to a conformal density $\{\mu_x\}$ such that 
$\p$ is injective on a neighborhood of the closure of $B_{\e_0+}$ for some $\e_0>0$. Let $C_1$, $C_2$ be 
as described in the beginning of \S\ref{subsec:2.6}.  For notational convenience, we will {\em identify\/}  $T_{\e_{0}+}$ and $B_{\e_{0}+}$  with their respective images in $\T^1(X)$ under $\p$.

\begin{proposition} \label{prop:trans+integral}
Let $0<\epsilon \leq \e_0$ and $r>r_\e$ (see~\eqref{eq:re}). Then for all measurable functions $\psi\geq 0$ on $T_{\e_{0+}}$ 
with $\supp(\psi)\subset T_{\e-}$
and $f\geq 0$ on $E$, we have
\begin{align*}
\notag
&(e^{-\delta_\mu \epsilon}) \int_{E} \Psi^-_\epsilon(\gr(w)) f^-_\epsilon(w)\,d\mu_{E}(w)
\\
\leq &  e^{-\delta_\mu r} \sum_{t\in T\cap \gr(E )} \#(\Drt) \cdot \psi(t)f({\grinv}(t))
\\
\leq  & (e^{-\delta_\mu \epsilon}) \int_{E} \Psi^+_{\epsilon}(\gr(w))f^+_\epsilon(w)d \mu_{E}(w),
\end{align*}
where the function $\Psi$ on $B_{\e_{0}+}$ is defined by
\[
\Psi(\p(tp)):={\psi(t)}/{\mu_{\cH_t^+}(tP)}, \text{ for all $(t,p)\in T_{\e_{0}+}\times P$},
\]
and $\Psi^{\pm}_{\e}$ on $B_{\e+}$ and $f_{\e}^{\pm}$ on $E$ are defined as in \eqref{fdef}.
\end{proposition}

\begin{proof}
Since  $\int_{tP} \Psi\,d\mu_{\cH^+_t}=\psi(t)$,
the result is straightforward to deduce from Corollary~\ref{psie}.
\end{proof}

In the section \ref{sec:mix}, Proposition~\ref{prop:trans+integral} will enable us to describe the limiting distribution of the 
transversal intersections $T\cap\gr(E)$ using
the mixing of the geodesic flow with respect to $m^{\BMS}$ (cf.\
Theorem~\ref{thm:transversal}).

\subsection{Some direct consequences}

The results proved in this subsection are also of independent interest. Let the notation be as in \S\ref{subsec:281}.

\begin{corollary}\label{finite}
Let $0<\e\leq\e_0$ and $f$ be a measurable function on $E$ such that $f_\e^+\in L^1(E,\mu_{E})$. Then  for any $r>r_\e$ 
and any measurable function $\psi$ on $T$:
\begin{equation*} \label{eq:sum-f-Nrt}
\sum_{t\in T\cap \gr( E )} \Nrt\cdot\abs{\psi(t) f({\grinv}(t))}<\infty.
\end{equation*}

In particular, if there exists a $\G$-invariant conformal density $\{\mu_x\}$ and
$\abs{\mu_E}<\infty$, then
\begin{equation*} \label{eq:sumNrt}
\sum_{t\in T\cap \gr E} \Nrt <\infty .
\end{equation*}
\end{corollary}
\begin{proof} By Proposition~\ref{prop:trans+integral} with $T_{\e}$ in place of $T$ and declaring $\psi$ to be zero outside $T$, 
we obtain the first claim, because
\[
\sum_{t\in T\cap \gr(E )} \Nrt\cdot |\psi(t) f({\grinv}(t))|
\le (1+\epsilon)e^{\delta_\mu r}\|\Psi_\e^+\|_\infty \cdot \mu_{E}(|f^+_\e|).
\]
To deduce the second claim from the first one, we choose $f=1$ on $E$ and $\psi=1$ on $T$.
\end{proof}

\begin{definition}[Radial Limit points]

The limit set $\Lambda(\G)$ of $\G$ is the set of all
accumulation points of an orbit $\G (z)$ in $\overline{\bH}^n$ for $z\in \bH^n$. As
    $\G$ acts properly discontinuously on $\bH^n$, $\Lambda(\G)$
    is contained in $\partial{\bH^n}$.

A point $\xi\in \Lambda(\G)$ is called
 a {\em radial limit point\/} if for some (and hence every) geodesic ray
$\beta$ tending to $\xi$ and some (and hence every) point
$x\in \bH^n$, there is a sequence $\gamma_i\in \G$ with
$\gamma_i x\to \xi$ and $d(\gamma_i x, \beta)$ is bounded.

We denote by $\Lambda_{\rmr}(\G)$ the set of radial limit points
for $\G$.

If $\G$ is non-elementary, $\Lambda_{\rmr}(\G)$ is a nonempty $\G$-invariant subset of
$\Lambda(\G)$. Since $\Lambda(\G)$ is a $\G$-minimal closed subset of $\partial{\bH^n}$, we have that $
\overline{{\Lambda_{\rmr}}(\G)}=\Lambda(\G)$.
\end{definition}

\begin{theorem} \label{thm:finite-closed} Let $C$ denote the smallest subsphere of $\bH^{n}$ containing $\Lambda(\G)$. 
Suppose that $C=\partial \tS$ or $\dim(C)>\dim(\partial\tS)$.   If there exists a $\G$-invariant conformal density
$\{\mu_x:x\in \bH^n\}$  such that $\abs{\mu_E}<\infty$, then the natural map
$\bar\p:\Gamma_{\tE}\backslash \tE \to \Gamma\bs \T^1(\bH^n)$ is proper.
\end{theorem}

\begin{proof}
Note that $\G\subset G_{C}=\{g\in G:gC=C\}$, because if $\g\in \G$, then $\g C\cap C\supset\Lambda(\G)$, 
hence by minimality $\g C=C$.

Suppose $C=\partial\tS$
Then, since $G_{\tS}=G_{\partial\tS}$,
 $\G=\G\cap G_{C}=\G_{\tS}=\G_{\tE}$ and hence the properness of $\bar\p$ is obvious.

Now suppose that $\dim(C)>\dim(\tS)$ and that that $\bar\p$ is not proper.  Then there exist sequences
$\gamma_i\in \Gamma$ and $e_i\in \tE$
such that  $\gamma_i e_i$ converges to a vector $ v\in \T^1(\bH^n)$ as $i\to\infty$, and
\begin{equation} \label{eq:GE}
\gamma_i \G_{\tE}\neq \gamma_j \G_{\tE}, \text{ for all $i\neq j$.}
\end{equation}

Fix $e_0\in \tE$. Since $G_{\tE}$ acts transitively on $\tE$,  there exists $h_i\in G_{\tE}$ such
that $e_i=h_ie_0$. Then $\gamma_ih_ie_0$ converges to $v$. Therefore there exists $g\in G$ such that $
\gamma_ih_i\to g$ and $v=ge_0$.

Now $\V(g\tE)=\partial{\bH^n}-\partial (g\tS)$. Since $\dim(\partial(g\tS))=\dim(\tS)<\dim(C)$, we have that $\Lambda(\G)
\setminus \partial(g\tS)$ is a nonempty open subset of $\Lambda(\G)$. Since $\Lambda_{\rmr}(\G)$ is dense in $\Lambda(\G)$, 
it follows that
 \[
 \Lambda_{\rmr}(\G)\cap \V(g\tE)\neq\emptyset.
 \]

Therefore there exists $h_0\in G_{\tE}$ such that $\V(gh_0e_0)=(gh_0e_0)^+ \in
\Lambda_{\rmr}(\G)$. Hence there exist $r_i\to\infty$ such that $\bar\p(\gri(gh_0e_0))$ converges to a point in $\T^1(X)$. Then 
there exists a sequence $\{\g_i'\}\subset \G$ such that $\gri(\g_i'gh_0e_0)\to u$ for some $u\in \T^1(\bH^n)$.

Let $B=TP$ be an admissible box centered at $u$. Let $\e>0$ be such that
$u\in B_{3\e-}$. Fix $k\in \N$ such that $r_k>r_\e$ (see~\ref{eq:re}) such that for $\g'=\g_k'$, we have $\gr(\g' gh_0e_0)\in 
B_{2\e-}$.

Since $\gamma_ih_i\to g$, $\gr(\gamma' \gamma_i h_ih_0e_0)\in B_{\e-}$
for all $i\geq i_0$ for some $i_0$. Since $h_ih_0e_0\in \tE$,
by \eqref{eq:910} $t_i\in T\cap \gr(\gamma' \gamma_i \tE)$
for all $i\geq i_0$. Therefore
\begin{equation} \label{eq:GTE}
(\G T \cap \gr \tE)\supset \{(\gamma^{\prime}\gamma_{i})\inv t_i:i\ge i_0\}.
\end{equation}
We claim that for any $i\in \N$,
\begin{equation} \label{eq:Gtitj}
\G_{\tE}\g_{i}\inv (\g^{\prime})\inv t_i \neq \G_{\tE}\g_{j}\inv(\g^{\prime})\inv t_j, \text{ for all but  finitely many $j$}.
\end{equation}
To see this, since $\p$ is injective on $T$, if $t_i\neq t_j$, then $\G t_i\neq \G t_j$ and hence
\eqref{eq:Gtitj} holds. If $t_i=t_j$, then it follows from \eqref{eq:GE} as $\G\cap G_{(\g')\inv t_{i}}$ is finite.  Combining
\eqref{eq:GTE} and \eqref{eq:Gtitj}, we get that
\[
\#(\G_{\tE}\bs (\G T \cap \gr \tE))=\infty.
\]
We observe that if $t\in T\cap \gr(E)$, then $\G_{\tE}\bs(\G t\cap \gr\tE)=\Drt\inv t$.
If $\abs{\mu_E^{\PS}}<\infty$, then by \eqref{eq:sumNrt} of Corollary~\ref{finite}
\[
\#(\G_{\tE}\bs (\G T \cap \gr \tE))\leq \sum_{t\in T\cap\gr(E)} \Nrt<\infty,
\]
which is a contradiction.
\end{proof}

\begin{remark} \label{rem:proper}
(1) Theorem~\ref{thm:finite-closed} holds for $\G$ Zariski dense:
since $\G\subset G_C$ and $G_C$ is Zariski closed, we have $C=\partial\bH^{n}$ for $\G$ Zariski dense.

(2) Theorem~\ref{thm:finite-closed} holds in the case $\Lambda (\G_{\tS})=\partial\tS$; since $\tS \subset C$ in this case, and hence we have either 
$\tS=C$ or $\text{dim}(C)>\text{dim}(\tS)$. 
\end{remark}

\section{Equidistribution of $\gr_*\mu_E^{\Leb}$}

\label{sec:mix}

 \subsection{BMS-measure and BR-measure on $\T^1(X)$}\label{defbms}
As before, let $\G$ be a non-elementary torsion-free discrete subgroup of $G$ and set
$X:=\G\ba \bH^n$. Let $\{\mu_x\}$ and $\{\mu_x'\}$ be $\G$-invariant conformal densities on
$\partial{\bH^n}$ of dimension $\delta_\mu$ and $\delta_{\mu'}$ respectively.
 After Roblin~\cite{Roblin2003}, we define a measure $m^{\mu,\mu'}$ on
 $\T^1(X)$ associated to $\{\mu_x\}$ and $\{\mu_x'\}$ as follows. Fix $o\in \bH^n$. The
the map
\[
u \mapsto (u^+, u^-, \beta_{u^-} (o,\pi(u)))
\]
is a homeomorphism between $\op{T}^1(\bH^n)$ with
 \[
 (\partial{\bH^n}\times \partial{\bH^n} \setminus \{(\xi,\xi):\xi\in \partial{\bH^n}\})  \times \br.
 \]
Hence we can define a measure $\tilde m^{\mu,\mu'}$ on $\T^1(\bH^n)$ by
\begin{equation} \label{eq:m-mumu}
d \tilde m^{\mu,\mu'}(u) =
e^{\delta_\mu \beta_{u^+}(o, \pi(u))}\; e^{\delta_{\mu'} \beta_{u^-}(o,\pi(u)) }\;d\mu_o(u^+) d
\mu'_o(u^-) ds,
\end{equation}
where $s=\beta_{u^-}(o,\pi(u))$.
Note that $\tilde m^{\mu,\mu'}$ is $\G$-invariant. Hence
it induces a locally finite measure $m^{\mu,\mu'}$ on $\T^1(X)$ such that if $\p$ is
injective on $\Omega\subset \T^1(\bH^n)$, then
\[
m^{\mu,\mu'}(\p(\Omega))=\tilde m^{\mu,\mu'}(\Omega).
\]
This definition is independent of the choice of $o\in \bH^n$.

Two important conformal densities on $\bH^n$ we will consider are the Patterson-Sullivan
density and the $G$-invariant (Lebesgue) density.

\subsubsection{Critical exponent $\delta_{\G}$} \label{subsec:criticalexp}
We denote by $\delta_\G$ the critical exponent of $\G$ which is defined as the abscissa of convergence of a Poincare series 
$\sum_{\gamma\in \G} e^{-sd(o, \gamma (o))}$ for
some $o\in\bH^n$; that is, the series converges for $s>\delta_{\G}$ and diverges for
$s<\delta_{\G}$ and the convergence property is independent of the choice of $o\in\bH^n$.

As $\G$ is non-elementary, we have $\delta_\G>0$.
Generalizing the work of Patterson \cite{Patterson1976} for $n=2$,
Sullivan \cite{Sullivan1979} constructed a $\G$-invariant conformal density
$\{\nu_x: x\in \bH^n\}$ of dimension $\delta_\G$  supported on $\Lambda(\G)$, which is
unique up to homothety, and called the {\em Patterson-Sullivan density}. From now on, we will simply write $\delta$ instead of 
$\delta_{\G}$.

We denote by $\{m_x:x\in \bH^n\}$ a $G$-invariant conformal density on the boundary
$\partial{\bH^n}$ of dimension $(n-1)$, which is
unique up to homothety, and
each $m_x$ is invariant under the maximal compact subgroup $G_x$. It will be called the {\em
Lebesgue density}.

The measure $m^{\nu,\nu}$ on $\op{T}^1(X)$ is called
the {\em Bowen-Margulis-Sullivan measure\/} $m^{\BMS}$
associated with $\{\nu_x\}$  (\cite{Bowen1971}, \cite{Margulisthesis}, \cite{Sullivan1984}):
 \begin{equation} \label{eq:defBMS}
 dm^{\BMS}(u)= e^{\delta \beta_{u^+}(o,
\pi(u))}\;
 e^{\delta \beta_{u^-}(o, \pi(u)) }\; d\nu_o(u^+) d\nu_o(u^-) ds.
 \end{equation}

The measure $m^{\nu, m}$ is called the {\em Burger-Roblin measure\/} $m^{\BR}$
associated with $\{\nu_x\}$ and $\{m_x\}$
(\cite{Burger1990}, \cite{Roblin2003}):
 \begin{equation} \label{eq:defBR}
 dm^{\BR}(u)= e^{(n-1) \beta_{u^+}(o,
\pi(u))}\;
 e^{\delta \beta_{u^-}(o, \pi(u)) }\; dm_o(u^+) d\nu_o(u^-) ds.
 \end{equation}

We note that the support of $m^{\BMS}$ and  $m^{\BR}$ are given respectively
by $\{u\in \T^1(X): u^+, u^-\in
\Lambda(\G)\}$ and
$\{u\in \op{T}^1(X): u^-\in \LG\} .$

\subsection{Relation to classification of measures invariant under horocycles}
 Burger \cite{Burger1990} showed that for a convex cocompact hyperbolic surface $\Gamma\ba \bH^2$
with $\delta > 1/2$, $m^{\BR}$ is a unique ergodic horocycle invariant locally finite measure
 which is not supported on closed horocycles. Roblin extended
 Burger's result in much greater generality.
By identifying the space $\Omega_\cH$ of all unstable horospheres
with $\partial{\bH^n}\times \br$ by $\cH^+(u)\mapsto (u^-, \beta_{u^-}(o,\pi(u)))$,
one defines the measure $d\hat \mu(\cH)=d\nu_o(\xi)e^{\delta s}ds$
for $\cH=(\xi, s)$. Then Roblin's theorem  \cite[Thm. 6.6]{Roblin2003} says that
if $\abs{m^{\BMS}}<\infty$, then
$\hat \mu$ is the unique Radon $\G$-invariant measure
on $\Lambda_{\rmr}(\G)\times \br\subset \Omega_\cH$.
This important classification result is not used in this article, but
it suggests that the asymptotic distribution of expanding
horospheres should be described by $m^{\BR}$.

\subsection{Patterson-Sullivan and Lebesgue measures on $\tE$, $\cH^{+}_{u}$ and $E$}
Let $\tS$ and $\tE$ be as in the subsection \ref{name}.
The following measures are special cases of the measures defined in the subsection
\ref{subsec:muE}.

Fix $o\in\bH^n$. Define the Borel measure $\mu_{\tE}^{\Leb}$ on $\tE$ such that
\begin{equation} \label{eq:muELeb}
d \mu_{\tE}^{\Leb}(v)=
e^{(n-1)\beta_{v^+}(o, \pi(v))} dm_o(v^+) .
\end{equation}

Since $\{m_x\}$ is a $G$-invariant conformal density on $\partial\bH^n$, the measure
$\mu_{\tE}^{\Leb}$ is $G$-invariant; that is, $g_*\mu_{\tE}^{\Leb}=\mu_{g(\tE)}^{\Leb}$. In
particular, it is a $G_{\tE}$ invariant measure on $\tE$.

Define the Borel measure $\mu_{\tE}^{\PS}$ on $\tE$ such that
\begin{equation} \label{eq:muEPS}
d \mu_{\tE}^{\PS}(v) =
e^{\delta \beta_{v^+}(o, \pi(v))} d\nu_o(v^+) .
\end{equation}
We note that $\mu_{\tE}^{\PS}$ is a $\G$-invariant measure.

As described in the section \ref{subsec:muE}, we denote by
 $ \mu_{ E}^{\Leb}$ and $\mu_{E}^{\PS}$
 the measures on $E=\p(\tilde E)$ induced by $ \mu_{\tE}^{\Leb}$ and $\mu_{\tE}^{\PS}$
respectively. Each of them is a pushforward of the corresponding locally finite
measure on $\G_{\tE}\bs \tE$.

As in section \ref{subsec:horosphere}, we have families of measures
$\mu^{\PS}=\{\mu^{\PS}_{\cH^+}\}$
and $\mu^{\Leb}=\{\mu^{\Leb}_{\cH^+}\}$ on the unstable horospherical foliation satisfying
$$ \mu^{\PS}_{\gr(\cH^+)}(\gr(F)) = e^{\delta r}
\mu^{\PS}_{\cH^+}(F)\;\; \text{and}\;\; \quad  \mu^{\Leb}_{\gr(\cH^+)}(\gr(F)) = e^{(n-1)r}
\mu^{\Leb}_{\cH^+}(F)$$ for any Borel subset $F$ of $\p(\cH^+)$.

\subsection{Transverse measures for $m^{\BMS}$}
\label{subsec:Transverse}
For each measurable $T$ contained in a weak stable leaf of the geodesic flow on $\T^1(\bH^n)$, called a {\em transversal}, 
define a measure $\lambda_T$ on $T$ by
\begin{equation} \label{eq:LT}
d\lambda_T(t)=e^{-\delta s}\, d\nu_o(t^-) ds
\end{equation}
where $s=\beta_{t^-}(o,\pi(t))$. If $B=TP$ is any box and $p\in P$, then $(tp)^-=t^-$ and $\cH^
+_{tp}=\cH^+_{t}$,  and hence $\beta_{(tp)^-}(o,\pi(tp))=\beta_{t^-}(o,\pi(t))$. Hence
\[
d\lambda_{Tp}(tp)=d\lambda_T(t);
\]
that is, $\lambda_T$ is holonomy invariant, where the holonomy is given by $t\mapsto tp$.

Now for any $\Psi\in C(B)$, by \eqref{eq:defBMS}- \eqref{eq:LT}, we have
 \begin{align}
 \label{eq:BMS-product}
 \int_{B} f \, dm^{\BMS}& = \int_T\int_P \Psi(tp) \, d\mu^{\PS}_{\cH^+_{t}}(tp) d\lambda_T(t)\\
 \label{eq:Leb-product}
 \int_{B} f \, dm^{\BR}& = \int_{T}\int_{P}\Psi(tp) \, d\mu^{\Leb}_{\cH^+_{t}}(tp)d\lambda_{T}(t).
 \end{align}

\subsubsection{Backward admissible box}
\label{subsec:Backward admissible box}
\begin{lemma} \label{lemma:backadm}
For any $u\in \T^{1}(\bH^{n})$ and $\e>0$, there exists a box $B=TP$ about $u$ such that
\begin{enumerate}
\item $\abs{\lambda_{T}}>0$; or equivalently $\nu_{o}(\{t^{-}:t\in T\})>0$, and
\item $\limsup_{r\to\infty} d(\gr(tp),\gr(t^{\prime}p))<\epsilon$, for all $t,t^{\prime}\in T$ and $p
\in P$.
\end{enumerate}
\end{lemma}
Such a box $B$ as above will be called a {\it backward admissible box with asymptotically $\e$-thin transversals.}

\begin{proof} As in the proof of Lemma~\ref{lemma:adm}, there exists a relatively compact
open neighborhood $P^{-}$ of $u$ in $\cH^{-}_{u}$ such that $\nu_{o}(\{t^{-}:t\in P^{-}\})>0$,
and $\p$ is injective on a neighborhood of the closure of $P^{-}$. Let $r_0= -\log(\epsilon/
4\diam(P^{-}))$. Then $\diam(\gro(P^{-}))=\epsilon/4$ and $\p$ is injective on a neighborhood
of the closure of $\gro(P^{-})$. Let $T_{1}$ be an open relatively compact neighborhood of $
\gro(P^{-})$ in $\V\inv(u^{+})$ and $P_{1}$ be an open relatively compact neighborhood of $
\gro(u)$ in $\cH^{+}_{\gro(u)}$ such that $T_{1}P_{1}$ is a box about $\gro(u)$ contained in a
ball of
radius $\epsilon/2$ about $u$. Let $T={\mathcal G}^{-r_0}(T_{1})$ and $P={\mathcal G}^{-
r_0}(P_{1})$. Then $B=TP$ has
the required properties. The property (1) holds because
\[
\{t^{-}:t\in T\}=\{t^{-}:t\in T_{1}\}\supset \{t^{-}:t\in \gro(P^{-})\}=\{t^{-}:t\in P^{-}\}.
\]
For the property (2), let $t_{1}=\gro(t)$ and $t_{1}^{\prime}=\gro(t^{\prime})$ in $T_{1}$ and
$p_{1}=
\gro(p)\in P_{1}$. Since $(t_{1}p_{1})^{+}=(t_{1}'p_{1})^{+}$, for any $r>r_0$,
 \[
 d({\mathcal G}^{r}(tp),{\mathcal G}^{r}(t^{\prime}p))=d({\mathcal G}^{r-r_0}(t_{1}
p_{1}),{\mathcal G}^{r-r_0}(t_{1}^{\prime}p_{1}))\leq d(t_{1}p_{1},t_1' p_{1})\leq \epsilon.
 \]
\end{proof}

\subsection{Mixing of geodesic flow}
We assume that $\abs{m^{\BMS}}<\infty$ for the rest of this section. This implies that $\G$
is of divergent type,
that is, $\sum_{\gamma\in \G} e^{-\delta d(o, \gamma o)}=\infty$ and that
the $\G$-invariant conformal density of dimension $\delta$ is unique up to
homothety (see \cite[Coro.1.8]{Roblin2003}).

Hence, up to homothety, $\nu_x$ is the weak-limit as $s\to \delta^+$
of the family of measures
$$\nu_{x,o}(s):=\frac{1}{\sum_{\gamma\in \G} e^{-sd(o, \gamma o)}}
\sum_{\gamma\in\G} e^{-sd(x, \gamma o)} \delta_{\gamma o},$$ where $\delta_{\g o}$ denotes the unit mass at $\g o$ 
for some $o\in \bH^n$.

The most crucial ergodic theoretic result involved in this work is the mixing of geodesic flow
which was obtained by Rudolph for $\G$ geometrically finite and by Babillot in general:

\begin{theorem}[Rudolph~\cite{Rudolph1982}, Roblin~\cite{Roblin2003}, Babillot~\cite{Babillot2002}]
\label{thm:mixing}
For any $ \Psi_1\in L^2(\op{T}^1(X), m^{\BMS})$ and $\Psi_2 \in L^2(\op{T}^1(X), m^{\BMS})$,
\[
\lim_{r\to\infty}\int_{\op{T}^1(X)} \Psi_1(\gr(x)) \Psi_2(x)\; d m^{\BMS}(x) =
\frac{1}{|m^{\BMS}|}\,  m^{\BMS}(\Psi_1)\cdot
 m^{\BMS}(\Psi_2).
 \]
\end{theorem}

From this theorem, we derive the following result, which generalizes the corresponding result
for PS-measures on unstable horospheres due to
Roblin~\cite[Corollary~3.2]{Roblin2003}.

\begin{theorem} \label{thm:grmuE}
For any $\Psi\in C_c(\T^1(X))$ and $f\in L^{1}(E,\mu_{E}^{\PS})$,
\begin{equation} \label{eq:grmuE}
\lim_{r\to\infty} \int_{x\in E} \Psi(\gr(x))f(x)\,d\mu^{\PS}_{E}(x) =\frac{\mu^{\PS}_{E}(f)}
{\abs{m^{\BMS}}} \cdot m^{\BMS}(\Psi).
\end{equation}
\end{theorem}

We will deduce the above statement from its following version.

\begin{proposition} \label{prop:grmuE}
Let  $\Psi\in L^{1}(\T^1(X),m^{\BMS})$ and $f\in L^{1}(E,\muPS_{E})$, both nonnegative, bounded and vanish outside 
compact sets. Then for any $\e>0$,
\begin{align} \label{eq:grmuE1}
\limsup_{r\to\infty} \int_{x\in E} \Psi(\gr(x))f(x)\,d\mu^{\PS}_{E}(x)
&\leq \frac{\mu^{\PS}_{E}(f)}{\abs{m^{\BMS}}} \cdot m^{\BMS}(\Psi^+_\e) \\
\liminf_{r\to\infty} \int_{x\in E} \Psi(\gr(x))f(x)\,d\mu^{\PS}_{E}(x)
&\geq \frac{\mu^{\PS}_{E}(f)}{\abs{m^{\BMS}}} \cdot m^{\BMS}(\Psi^-_\e),
\label{eq:grmuE2}
\end{align}
where, for any $u\in \T^1(\bH^n)$,
\begin{equation} \label{eq:Psi-e}
\begin{array}{ll}
\Psi^+_\e(\p(u))&:=\sup\{\Psi(\p(v)):d(v,u)<\e,\; v\in\V\inv(u^+)\}, \\
\Psi^-_\e(\p(u))&:=\inf\{\Psi(\p(v)):d(v,u)<\e,\; v\in \V\inv(u^+)\}.
\end{array}
\end{equation}
\end{proposition}

\begin{proof}
By Lemma~\ref{lemma:backadm}, there exists a finite
open cover $\cB$ of $\supp(f)\subset E\subset \T^{1}(X)$ consisting of backward admissible
boxes $B$ with asymptotically $\e$-thin transversals; we identify $B\subset \T^{1}(\bH^{n})$ with $\p(B)$. By considering a 
partition of unity subordinate to this cover, $f=\sum_{B\in\cB} \phi_{B}$, where $ \phi_{B}\in L^{1}(E,\muPS_{E})$ is a 
non-negative function whose support is contained in $ \p(B)$. Therefore it is enough to prove
\eqref{eq:grmuE1} and \eqref{eq:grmuE2} for $\phi_B$ in place of $f$ for each $B\in \cB$.

Fix any $B\in\cB$. For each $[\g]\in \G/\G_{\tE}$, let $\phi_{\g}(w)=\phi_{B}(w)$
 for all $w\in\g\tE$. By
\eqref{eq:muPSE},
\begin{align*}
\muPS_{E}(\phi_{B})&=\sum_{[\gamma]\in\G/\G_{\tE}} \muPS_{\g\tE}(\phi_{\g}), \text{ and }
\\
\int_{x\in E} \Psi(\gr(x))\phi_{B}(x)\, d\muPS_{E}(x)
&=\sum_{[\g]\in \G/\G_{\tE}} \int_{w\in \g\tE\cap B}
\Psi(\gr(w))\phi_{\g}(w)\,d\muPS_{\g\tE}(w).
\end{align*}
Therefore to prove  \eqref{eq:grmuE1} and \eqref{eq:grmuE2} for $\phi_B$ in place of $f$, 
it is enough to prove the following: for any $\g\in\G$ and
$\phi:=\phi_\g\in L^{1}(\g\tE,\muPS_{\g\tE})$ vanishing outside $\g\tE\cap B$, we have
\begin{align}
\label{eq:grmuE+}
\limsup_{r\to\infty} \int_{w\in \g\tE\cap B} \Psi(\gr(w))\phi(w)\,d\mu^{\PS}_{\g\tE}(w)
\leq \frac{\mu^{\PS}_{\g\tE}(\phi)}{\abs{m^{\BMS}}} m^{\BMS}(\Psi^{+}_{\e});
\\
\label{eq:grmuE-}
\liminf_{r\to\infty} \int_{w\in \g\tE\cap B} \Psi(\gr(w))\phi(w)\,d\mu^{\PS}_{\g\tE}(w)
\geq \frac{\mu^{\PS}_{\g\tE}(\phi)}{\abs{m^{\BMS}}} m^{\BMS}(\Psi^{-}_{\e}).
\end{align}

Now we express $B=TP$. If $\g\tE\cap B=\emptyset$, then both sides of \eqref{eq:grmuE+} are zero and hence the claim
is true. Otherwise, there exists $(t_1,p_1)\in T\times P$ such that $v:=t_1p_1\in\g\tE$. We recall that as in 
\S\ref{subsubsec:xiv}, $\xi_{v}:\cH^{+}_{v}\setminus  (\g\cdot\V\inv(\partial\tS))
\to \g\tE\setminus\{-v\}$ and $q_{v}:\g\tE\setminus\{-v\}\to \cH^{+}_{v}\setminus (\g\cdot\V\inv
(\partial\tS))$ are differentiable inverses of each other.

 Letting
\begin{equation*} 
P_{1}=\{p\in P: \xi_{v}(t_{1}p)\in Tp\},
\end{equation*}
we claim that
\begin{equation} \label{eq:3.16}
\g\tE\cap B = \{\xi_{v}(t_{1}p):p\in P_{1}\}.
\end{equation}
To see this, if $tp\in \g\tE$ for some $(t,p)\in T\times P$, then
\[
q_v(tp)=\cH_v^+\cap \V\inv((tp)^+)=\cH_{t_1}^+\cap \V\inv(p^+)=t_{1}p.
\]
Hence $\xi_{v}(t_{1}p)=tp$, and so $p\in P_{1}$. The opposite inclusion is obvious.

We define a map $\rho:TP\to\g\tE$ as follows:
\[
\rho(tp)=\xi_{v}(t_{1}p), \text{ for all $(t,p)\in T\times P$}.
\]

For any $t\in T$, for the restricted map $\rho:tP\to \g\tE$, by \eqref{eq:mutE} and \eqref{eq:muH}, since 
$(tp)^{+}=p^{+}=\rho(tp)^{+}$, we have
\begin{equation}
\label{eq:3.11}
d\mu^{\PS}_{\g\tE}(\rho(tp))/d\mu^{\PS}_{\cH^{+}_{t}}(tp)
=e^{\beta_{p^{+}}(\pi(tp),\pi(\rho(tp)))}.
\end{equation}
In view of this, we define $\Phi\in L^2(\T^{1}(X),\mu^{\BMS})$ as follows: $\Phi(x)=0$ if $x\in X\setminus B$ and
\begin{equation} \label{eq:3.112}
\Phi(tp)=\phi(\rho(tp))e^{\beta_{p^{+}}(\pi(tp),\pi(\rho(tp)))}, \text{ if $x=tp\in B$.}
\end{equation}
We note that
\begin{equation} \label{eq:P1}
\Phi(tp)\neq 0\Rightarrow \rho(tp)\in B\Rightarrow p\in P_{1}.
\end{equation}

And for $t\in T$ and $p\in P_{1}$, we have
$\{\rho(tp),tp\}\subset Tp$. Since $\gr(B)$ has $\e$-thin transversals as $r\to\infty$ (see Lemma~\ref{lemma:backadm}(2)):
\begin{equation} \label{eq:313}
\limsup_{r\to\infty} d(\gr(\rho(tp)),\gr(tp))\leq \e \text{ for all $p\in P_{1}$.}
\end{equation}

By Theorem~\ref{thm:mixing},
\begin{align}
\label{eq:3.13}
&\frac{1}{|m^{\BMS}|}\,  m^{\BMS}(\Psi^{+}_{\e})\cdot m^{\BMS}(\Phi)
\\
\label{eq:3.14}
&=\lim_{r\to\infty}\int_{B} \Psi^{+}_{\e}(\gr(x))\Phi(x)\,dm^{\BMS}(x)
\\
\label{eq:3.12}
&=\lim_{r\to\infty}\int_{t\in T}\left(\int_{p\in P_{1}} \Psi^{+}_{\e}(\gr(tp))
\Phi(tp)\,d\mu^{\PS}_{\cH^{+}_{t}}(tp)\right) d\lambda_{T}(t)
\\
\label{eq:tp-p}
&= \lim_{r\to\infty}\int_{t\in T}\left(\int_{p\in P_{1}}
\Psi^{+}_{\e}(\gr(tp))\phi(\rho(tp))\,d\mu^{\PS}_{\g\tE}(\rho(tp))\right) d\lambda_{T}(t)
\\
\label{eq:tp-p2}
&\geq \abs{\lambda_{T}} \cdot \limsup_{r\to\infty}\int_{w\in \g\tE\cap B} \Psi(\gr(w))\phi(w)\,d\mu^{\PS}_{\g\tE}(w)\end{align}
where \eqref{eq:3.12} follows from \eqref{eq:BMS-product}  and \eqref{eq:P1},  \eqref{eq:tp-p} follows from \eqref{eq:3.16}, 
\eqref{eq:3.11} and \eqref{eq:3.112}, and to justify \eqref{eq:tp-p2} we put $w=\rho(tp)$ and use \eqref{eq:Psi-e} and 
\eqref{eq:313}.

By putting $\Psi(x)=1=\Psi^+_\e(x)$ in \eqref{eq:3.14}--\eqref{eq:tp-p2} with equality in \eqref{eq:tp-p2}, we
get
\begin{equation} \label{eq:Psi2}
m^{\BMS}(\Phi)=\abs{\lambda_{T}}\cdot \muPS_{\g\tE}(\phi)<\infty.
\end{equation}

Now \eqref{eq:grmuE+} is deduced by comparing \eqref{eq:3.13}, \eqref{eq:tp-p2} and
\eqref{eq:Psi2}, and  noting that $\abs{\lambda_{T}}\neq 0$ by the backward admissibility of
$B$.  Similarly we can deduce \eqref{eq:grmuE-}.
 \end{proof}

 \begin{proof}[Proof of Theorem~\ref{thm:grmuE}]
Since both the sides of \eqref{eq:grmuE} are linear in $\Psi$, it is enough to prove it for $\Psi\geq 0$. Since $\Psi$ is uniformly 
continuous and $\abs{m^{\BMS}}<\infty$,
\[
\lim_{\e\to0}m^{\BMS}(\Psi^{+}_{\e}-\Psi^{-}_{\e})=0.
\]
Therefore by Proposition~\ref{prop:grmuE}, we have that \eqref{eq:grmuE} holds for all nonnegative bounded measurable $f$ 
with compact support on $E$. Since the set of such $f$'s is dense in $L^1(E,\muPS_E)$ and both sides of \eqref{eq:grmuE} are 
linear and continuous in $f\in L^{1}(E,\muPS_{E})$, and  \eqref{eq:grmuE} holds for all $f\in L^{1}(E,\muPS_{E})$.
\end{proof}

The following result is one of the basic tools developed in this article.

\begin{theorem}[Transversal equidistribution] \label{thm:transversal}
Let $f \in L^1(E,\muPS_{E})$ such that $\muPS_{E}(f^+_\epsilon-f_\e^-)\to 0$ as $\epsilon\to 0$. Let $\psi\in C_{c}(T)$ for a 
transversal $T$ of a box $B$ (\S\ref{subsec:box}). Then
\begin{equation} \label{eq:transversal-equi}
\lim_{r\to \infty}
e^{-\delta r} \sum_{t\in T\cap \gr( E )} \Nrt\cdot\psi(t) \cdot f({\grinv}(t)) =\frac{\muPS_{E}(f) }{|
{m^{\BMS}}|} \cdot \lambda_T(\psi),
\end{equation}
where
\[
f^{+}_{\e}(x)=\sup_{\{y\in E:d(y,x)<\e\}} f(y) \text{ and }
f^{-}_{\e}(x)=\inf_{\{y\in E:d(y,x)<\e\}}f(y),
\]
the transverse measure $\lambda_{T}$ is defined by \eqref{eq:LT} and $\Drt$ is defined by \eqref{eq:Drt}.
\end{theorem}

\begin{proof}
Since both sides of \eqref{eq:transversal-equi} are linear in $f$ and in $\psi$, without loss of generality we may assume that $f
\geq 0$ and $\psi\geq 0$. By Lemma~\ref{lemma:adm},
$\supp(\psi)$ can be covered by finitely many admissible boxes. By a partition of unity argument, in view of Remark~
\ref{rem:recenter}, we may assume  without loss of generality that $T$ is a transversal of an admissible box $B$.

Let $\e_{0}>0$ be such that $\p$ is injective on $B_{\e_{0+}}$ and that $\psi$ vanishes outside $T_{\e_{0}-}$. We extend $\psi$ 
to a continuous function on  $T_{\e_{0}+}$ by putting $\psi=0$ on $T_{\e_{0}+}\setminus T$. Since $B$ is admissible, due to 
Lemma~\ref{haar}, if we define
\[
\Psi(tp)=\psi(t)/\muPS_{\cH^+_t}(tP), \text{ for all $(t,p)\in T_{\e_{0}+}\times P$},
\]
then $\Psi$ is a bounded continuous function on $B_{\e_{0}+}$ vanishing outside $B_{\e_{0}-}$. 
If $\Psi^{\pm}_{\e}\in C(B_{\e+})$ are defined as in \eqref{eq:Psi-e} for $0<\e\leq\e_{0}$, then
\begin{equation} \label{eq:Psi-cpt}
\lim_{\e\to 0}\, \norm{\Psi_{\e}^{+}-\Psi_{\e}^{-}}_{\infty}=0.
\end{equation}

By Proposition~\ref{prop:grmuE},
\begin{equation} \label{eq:3.29}
\begin{array}{ll}
\limsup_{r\to\infty} \int_{E} \Psi^+_\epsilon(\gr(v)) f^+_\epsilon(v)\,d\muPS_{E}(v)
&\leq \frac{ \muPS_{E}(f^+_\epsilon) m^{\BMS}(\Psi^+_{\epsilon})}{|m^{\BMS}|}, \\
\liminf_{r\to\infty} \int_{E} \Psi^-_{\epsilon}(\gr(v)) f^-_\epsilon(v)\,d\muPS_{E}(v)
&\geq \frac{ \muPS_{E}(f^-_\epsilon) m^{\BMS}(\Psi^-_{\epsilon})}{|m^{\BMS}|}.
\end{array}
\end{equation}

Since $m^{\BMS}(B_{\e_{0}+})<\infty$, by \eqref{eq:Psi-cpt}, we have that $m^{\BMS}(\Psi^{+}_{\e}-\Psi^{-}_{\e})\to 0$ 
as $\e\to 0$. By our assumption, $\muPS_{E}(\abs{f^\pm_\epsilon-f})\to 0$ as $\epsilon\to0$. Therefore by 
Proposition~\ref{prop:trans+integral} and \eqref{eq:3.29},
\begin{equation*}
\lim_{r\to\infty} e^{-\delta r} \sum_{t\in T\cap \gr(E)} \Nrt\cdot \psi(t) \cdot f({\grinv}(t)) =
\frac{\muPS_E(f) m^{\BMS}(\Psi)}{|m^{\BMS}|}.
\end{equation*}
And
\[
m^{\BMS}(\Psi)=\int_T d\mu_T(t)\left(\int_{tP} \Psi(tp) d\muPS_{\cH^+_t}\right)=
\lambda_T(\psi).
\]
\end{proof}

Now we state and prove the main equidistribution result of this article which is more general than Theorem~\ref{mainergint}.

\begin{theorem}  \label{m2}
Let $f\in L^1(E,\mu_E^{\PS})$ such that
$\muPS_{E}(f^+_\epsilon-f^-_\e)\to 0$ as $\epsilon\to 0$. Let $\Psi\in C_{c}(\T^{1}(X))$.
Then
\begin{equation*} \label{eq:m2}
\lim_{r\to\infty} e^{(n-1-\delta)r}
\int_{u\in E} \Psi(\gr(u))f(u)\; d\mu_{E}^{\op{Leb}}(u)
= \frac{\mu_{E}^{\PS}(f)}{ |m^{\BMS}|} m^{\BR}(\Psi).
\end{equation*}

In particular, the result applies to $f=\chi_{F}$ for a Borel measurable $F\subset E$ such that
$\muPS_{E}(F_{\e_{1}})<\infty$ for some $\e_{1}>0$ and $\muPS_{E}(\partial F)=0$.
\end{theorem}

\begin{proof} By Lemma~\ref{lemma:adm}, the boxes admissible with respect to
$\{\mu_{\cH^{+}}^{\PS}\}$ form a basis of open sets in $\T^1(X)$. By a partition of
unity argument, without loss of generality we may assume that $\supp(\Psi)\subset B$ for an
admissible box $B=TP$. Let $\e_{0}>0$ be such that $\Psi=0$ outside $B_{\e_{0}-}$. For $0<\e\leq \e_{0}$, let $\Psi^{\pm}_{\e}$ 
be defined as in \eqref{fdef}. Then
\begin{equation} \label{eq:Psi-e-0}
\lim_{\e\to 0} \, \norm{\Psi_{\e}^{+}-\Psi_{\e}^{-}}_{\infty}=0.
\end{equation}

For $t\in T_{\e_{0}}$, and $\e>0$, define
$\psi_\e^{\pm}(t)=\int_{tP} \Psi_\e^{\pm} \;d\mu^{\op{Leb}}_{\cH^+_t}$. By
Lemma~\ref{haar}, $\psi_\e^{\pm}\in C_{c}(T)$ for any $0<\e<\e_{0}/2$.

For the conformal density, $\{\mu_x\}=\{m_x\}$, we have $\delta_{\mu}=n-1$,
and by multiplying all the terms in the conclusion of Corollary~\ref{psie} by $e^{-\delta r}$,  
for $r>r_{\e}$ (see \eqref{eq:re}), we get
\begin{align*}
&(e^{-(n-1)\e})e^{-\delta r} \sum_{t\in T\cap \gr( E)} \Nrt \cdot \psi_\e^{-}(t) \cdot f_\e^-({\grinv}(t))
\\
\le & e^{(n-1-\delta) r}  \int_{E} \Psi(\gr(u))f(u) \;d\mu^{\op{Leb}}_{E}(u)
\\
\le &(e^{(n-1)\e})e^{-\delta r} \sum_{t\in T\cap \gr( E)} \Nrt \cdot \psi_\e^{+}(t)  \cdot f_\e^+({\grinv}(t)).
\end{align*}

Define $\psi(t):= \int_{tP} \Psi (tp) \,d\mu^{\op{Leb}}_{\cH^+_t}$ for all $t\in T$. Then
$\lambda_T(\psi)=m^{\BR}(\Psi)$ and $\lambda_T(\psi_\e^{\pm})=m^{\BR}(\Psi_\e^{\pm})$.
Since $m^{\BMS}(B_{\e_{0}+})<\infty$, by \eqref{eq:Psi-e-0},
\[
\lambda_{T}(\psi^{+})-\lambda_{T}(\psi^{-})=m^{\BR}(\Psi_\e^{+}-\Psi_\e^-)\to 0,\text{ as $\e\to
0$}.
\]
And since $\muPS_{E}(f^+_\epsilon-f^-_\e)\to 0$, by Theorem~\ref{thm:transversal}
\[
\lim_{r\to\infty} e^{(n-1-\delta) r} \int_{E} \Psi(\gr(u))f(u) \;d\mu^{\op{Leb}}_{E}(u) =
\frac{\muPS_E(f)\lambda_{T}(\psi)}{|m^{\BMS}|}.
\]
Since $\lambda_T(\psi)=m^{\BR}(\Psi)$, we prove the claim.

In the particular case of $f=\chi_{F}$, we have
\[
\inf_{\e>0} f_{\e}^{+}=\chi_{\cl{F}}\quad\text{and}\quad \sup_{\e>0}f_{\e}^{-}=\chi_{\Int(F)}, \text { and }
\]
if $\muPS_{E}(f_{\e_{1}}^{+})=\muPS_{E}(F_{\e_{1}})<\infty$, then
$\lim_{\e\to 0}\muPS_{E}(f_{\e}^{+}-f_{\e}^{-})=\muPS_{E}(\partial F)$.
\end{proof}

The idea of the above proof was influenced by the work of Schapira~\cite{Schapira2005}.

Our proof also yields the following variation of  Theorem~\ref{m2}.

\begin{theorem} \label{thm:subset}
Let $\tilde F\subset \tilde E$ be a Borel subset such that
 $\mu_{\tE}^{\PS}(\tilde F_{\e}) <\infty$
for some $\e>0$ and $\mu_{\tilde E}^{\PS}(\partial \tilde F)=0$.
Then for any $\psi\in C_c(\T^1(\G\ba \bH^n))$,
\begin{equation*}
\lim_{t\to +\infty} e^{(n-1-\delta)t}\cdot
\int_{\tilde F}\psi({\gt}(v))\; d\mu^{\Leb}_{\tilde E}(v)
=\frac{ \mu_{\tilde E}^{\PS}(\tilde F)}{\abs{m^{\BMS}}}
 \cdot m^{\BR} (\psi).
\end{equation*}
\end{theorem}


\subsection{Integrability of the base eigenfunction $\phi_0$}
 \label{subsec:eigen}
\begin{proof}[Proof of theorem~\ref{thm:mbf}]
We want to prove equivalence of the following:
\begin{enumerate}
\item[\eqref{mbf3}] $\phi_0\in L^1(\Gamma\ba\bH^n,d\Vol_{\Riem})$;
\item[\eqref{mbf2}] $\abs{m^{\BR}}<\infty$;
\item[\eqref{mbf1}] $\Gamma$ is a lattice in $G$.
\end{enumerate}

The pushforward of $m^{\BR}$ from $\G\bs\T^1(\bH^1)$ to $\Gamma \ba \bH^n$ is the measure corresponding to
$\phi_0\,d\Vol_{\Riem}$ (see~\cite[Lemma~6.7]{KontorovichOh}). Therefore \eqref{mbf3} and \eqref{mbf2} are equivalent.

To prove that \eqref{mbf2} implies \eqref{mbf1}, suppose that $\abs{m^{\BR}}<\infty$. Since the  left $G$-action on $\T^1(\bH^n)
$ is transitive, we may identify $\T^1(\G\bs \bH^n)$ with
$\G\bs G/M$ for a compact subgroup $M$. We lift the measure $m^{\BR}$ to a measure $m$ on $\Gamma\ba G$ as follows: for 
any $f\in C_{c}(\G\bs G)$, we define
$m(f)=m^{\BR}(\bar f)$, where $\bar f(\G gM)=\int_{x\in M} f(\G gx)\,dx$, where $dx$ is the probability Haar measure on $M$.  
Denote by $U$ the horospherical subgroup of $G$ whose orbits in $G$ projects to the unstable horospheres in $\T^1(\bH^n)$. 
Then $M$ normalizes $U$ and any unimodular proper closed subgroup of $G$ containing $U$ is contained in the subgroup 
$MU$. As $m$ is invariant under $G_{\cH^+}$ for any unstable
horosphere $\cH^+$, it follows that $m$ is a
$U$-invariant {\em finite\/} measure on $\Gamma\ba G$. By Ratner's theorem \cite{Ratner1991}, any ergodic component, say, $
\lambda$, of $m$ is a homogeneous measure in the sense that $\lambda$ is a $H$-invariant finite measure supported on a 
closed orbit $x_0 H$ for some $x_0\in \G\ba G$ and a unimodular closed subgroup $H$ of $G$ containing $U$. If $H\ne G$, 
then $H\subset MU$ and $\G\cap H$ is co-compact in $H$. It follows by a theorem of Bieberbach (\cite[Theorem~2.25]
{Bowditch1993}) that $\G \cap U$ is co-compact in $U$. Hence $H=U$. Thus we can write $m=m_1+ m_2$, where $m_1$ is $G
$-invariant and $m_2$ is supported on a union of compact $U$-orbits.

If $m_{1}=0$, then $m=m_{2}$, and hence the projection of the support of $m^{\BR}$ in
$\T^1(\bH^n)$ is a union of compact unstable horospheres. It follows that the Patterson-Sullivan density is concentrated on the 
set of parabolic fixed points of $\G$, which is a contradiction.

If $m_{1}\neq 0$, then $m_{1}$ is a finite $G$-invariant measure on $\G\bs G$; that is,
$\G$ is a lattice in $G$. Hence \eqref{mbf2} implies \eqref{mbf1}.

If $\G$ is a lattice, then $\{\nu_x\}=\{m_x\}$ up to a constant multiple.
Hence $m^{\BR}$ is the projection of a finite $G$-invariant measure of $\Gamma\ba G$ to
$\T^1(\Gamma\ba \bH^n)$. Hence \eqref{mbf1} implies \eqref{mbf2}.
\end{proof}

\section{Geometric finiteness of closed totally geodesic immersions}
\label{sec:4}

for obtaining the criterion for the finiteness of $\muPS_{E}$ in \S\ref{sec:5}.

\subsection{Parabolic fixed points and minimal subspaces} \label{subsec:bieberbach}

\newcommand{\Fix}{\operatorname{Fix}}
Let $\G$ be a torsion free discrete
subgroup of $G$.
\begin{definition}
An element $g\in G$ is called
{\em parabolic\/} if $\Fix(g):=\{\xi\in\partial{\bH^n}:g\xi=\xi\}$ is a singleton set. An element $
\xi\in\partial{\bH^n}$ is called a {\em parabolic fixed point of $\G$} if there exists a parabolic
element $\g\in\G$ such that $\Fix(\g)=\{\xi\}$. Note that if $\xi$ is a parabolic fixed point for 
$\G$, then $\xi\in \Lambda(\G)$. Let $\Lambda_{\rmp}(\G)$ denote the set of parabolic fixed points of $\G$.
\end{definition}

Let $\xi\in \Lambda_{\rmp}(\G)$.
In order to analyze the action of $\G_{\xi}$
on $\partial{\bH^n}\setminus \{\xi\}$, it is convenient to use the upper half space model
$\br^n_+=\{(x,y):x\in \br^{n-1}, y>0\}$ for $\bH^n$, where $\xi$ corresponds to $\infty$ and
$\partial{\bH^n}\setminus\{\xi\}$ corresponds to $\partial \br^n_+=\{(x,0):x\in \br^{n-1}\}$.
The subgroup $\G_\infty$ acts
properly discontinuously via affine isometries on $\partial{\bH^n}\setminus\{\infty\}\cong \R^{n-1}$; at
this stage we will treat $\R^{n-1}$
only as an affine space, and we will choose its origin $0$ later. Moreover the action of
$\G_{\infty}$ preserves every horosphere $\R^{n-1}\times \{y\}$, where $y>0$, based at
$\infty$.

By a theorem of Bieberbach (\cite[2.2.5]{Bowditch1993}), $\G_\infty$ contains a normal abelian subgroup of finite index, 
say $\G'_{\infty}$. By \cite[2.1.5]{Bowditch1993}, any (nonempty) $\G'_\infty$-invariant affine subspace of $\R^{n-1}$ contains a (nonempty) minimal $\G'_\infty$-invariant affine subspace, we call such an affine subspace a {\em $\G'_{\infty}$-minimal subspace}. By \cite[2.2.6]{Bowditch1993}, $\G'_{\infty}$ acts cocompactly via translations on any $\G'_{\infty}$-minimal subspace. Moreover, any two $\G'_{\infty}$-minimal subspaces are parallel, and if $v_{1}$ and $v_{2}$ belong to any two $\G'_{\infty}$-minimal subspaces, then $\g v_{1} - \g v_{2}=v_{1}-v_{2}$ for all $\g\in \G'_{\infty}$.  Let $\rank(\G_\infty)$ denote the rank of the (torsion free) $\z$-module $\G'_{\infty}$; it is independent of the choice of $\G_{\infty}'$, and it equals the dimension of a $\G'_{\infty}$-minimal subspace.

\begin{definition}\label{defbb} A parabolic fixed point $\xi\in \Lambda_{\rmp}(\G)$ is said to be
 {\em bounded\/} if $\G_{\xi}\ba (\Lambda(\G)\setminus\{\xi\})$ is compact. Denote by
 $\Lambda_{\bdp}(\G)$ the set of all bounded
parabolic fixed points for $\G$. Therefore $\infty\in \Lambda_{\bdp}(\G)$ {\em if and only if\/} $\infty\in\Lambda_{\rmp}(\G)$ and\begin{equation} \label{eq:bddpar}
\Lambda(\G)\setminus\{\infty\} \subset \{x\in \br^{n-1}: d_{\Euc}(x, L)\le r_0\},
\end{equation}
for some $r_0>0$, where $L$ is a $\G'_{\infty}$-minimal subspace.
\end{definition}

\subsection{On geometric finiteness of $\G_{\tS}$} \label{subsec:geomfinite}
For the rest of this section, let $\tS$ be a proper connected totally geodesic subspace
of $\bH^n$ such that the natural projection map $\G_{\tS} \ba  \tS \to X=\G\ba \bH^n$ is
proper, or equivalently, the map $\G_{\tS}\bs G_{\tS}\to \G \bs G$ is proper, or equivalently 
$
\G G_{\tS}$ is closed in $G$. Since $\tS$ is totally geodesic,  the geometric boundary
$\partial \tS$ is the intersection of $\partial{\bH^n}$ with the closure
of $\tS$ in $\overline{\bH^n}$.

\begin{proposition} \label{prop:parallel}
Let $\infty\in \Lambda_{\rmp}(\G)\cap \partial \tS$. Let $L$ be a $\G'_\infty$-minimal subspace of  
$\partial\bH^n \setminus \{\infty\}\cong\R^{n-1}$ and choose the origin $0\in L$. Then the intersection of $L$ with the (parallel) translate of the affine subspace $\partial\tS\setminus\{\infty\}$ through $0$ is a $(\G'_{\infty}\cap G_{\tS})$-minimal subspace.
\end{proposition}

\begin{proof}
Let $\G'=\G'_{\infty}$, $\Delta=\G'\cap G_{\tS}$, and the affine subspace $F=\partial\tS\setminus\{\infty\}$. Since $\Delta F=F$, 
let $v$ belong to a $\Delta$-minimal subspace of $F$. Since $v$ and $0$ belong to two $\Delta$-minimal subspaces,  $\g v - \g 
0 = v-0=v$. 
Since $\g v\in F$, we have $\g 0 =\g v - v\in F-v$. Since $0\in F-v$, we have $\g 0\in \g(F-v)\cap (F-v)$. Now $\g(F-v)$ and $\g 
F=F$ are parallel. Therefore $F-v$ and $\g (F-v)$ are parallel, and since they intersect, $\g(F-v)=F-v$. Thus $\Delta(F-v)=F-v$. 
Therefore $\Delta$-action preserves $L_{0}:=L\cap (F-v)$. We want to prove that $\G' \cap G_{\tS}$ acts cocompactly on 
$L_{0}$.

Since $\infty\in\Lambda_{\rmp}(\G)$,  by \cite[Lemma~3.2.1]{Bowditch1993} $\G_\infty$ consists of parabolic elements of 
$G_{\infty}$; that is $\G_{\infty}\subset MN$, where $N$ is the maximal unipotent subgroup of $G$ which acts transitively on $
\R^{n-1}=\partial\bH^{n}\setminus\{\infty\}$ via translations and $M$ is a compact subgroup of $G$ normalizing $N$ and acts on 
$\R^{n-1}$ by Euclidean isometries fixing $0$.  Let $U=\{g\in N: gL=L\}$. Then $U$ acts transitively on $L$ by translations.
Since $0\in L$ and $\G'$ acts cocompactly on $L$ via translations, 
the connected component of the Zariski closure of $\Gamma'$ in $G$ is a connected abelian
subgroup of the form $M_LU$, where $M_L\subset M$ and $M_L$ acts trivially on $L$. 

Since $\G'\ba L$ is a compact Euclidean torus, the closure of the image of $L_0$ in $\G'\bs L$ equals the image of an affine  
subspace, say
$L_1$, of $L$. Thus $\cl{\Gamma' L_0}=\Gamma' L_1$.
 For $i=0,1$, let $U_i=\{u\in U:uL_i=L_i\}$. Then $U_i$ acts transitively on $L_i$, and  
 $\cl{\Gamma' M_L U_0}=\Gamma'M_LU_1$. Therefore the identity component of 
 $\cl{\Gamma' U_0}$ is of the form $M_1U_1$, where $M_1\subset M_L$ and 
 $(\G'\cap M_1U_1)\bs M_1U_1$ is compact. In particular, $\G'\cap M_1U_1$ acts cocompactly on $L_1$.

By our assumption $\Gamma G_{\tS}$
is a closed subset of $G$. Therefore $\cl{\G' U_0}\subset \G G_{\tS}$.
 Since $G_{\tS}$ is the identity component in $\Gamma G_{\tS}$, we have
$M_1U_1 \subset G_{\tS}$. It follows that $U_1$ preserves $L_0$.
Since $U_1$ acts transitively on $L_1$, $L_1\subset L_0$; hence $L_1=L_0$.
 In particular, $\G'\cap M_1U_1$ acts cocompactly on $L_0$.
Therefore $\Delta=\G'\cap G_{\tS}$ acts cocompactly on $L_0$.
\end{proof}

\begin{proposition}\label{sstar}
Let $\infty\in  \Lambda_{\bdp}(\Gamma) \cap \partial \tS$ and 
$\G_{\tS}:=\G\cap G_{\tS}$. Then 
\begin{equation*}
\begin{cases}
\infty \in\Lambda_{\bdp}(\G_{\tS}) & \text{if
$\G_{\infty}\cap \G_{\tS}$ is infinite}; \\
\infty\notin \Lambda(\G_{\tS}) &\text{if $\G_{\infty}\cap \G_{\tS}$ is finite, hence trivial.}
 \end{cases}\end{equation*}
 \end{proposition}
\begin{proof}
Let the notation be as in Proposition~\ref{prop:parallel}. 
Since $\infty\in  \Lambda_{\bdp}(\Gamma)$, by \eqref{eq:bddpar},
$\Lambda(\G)\setminus\{\infty\}$ is contained in a bounded neighborhood of $L$, 
and hence in a bounded neighborhood of $L+v$. Therefore $(\Lambda(\G)\setminus\{\infty\})\cap \partial \tS$ is
contained in a bounded neighborhood of $L+v$ intersected with $F=\partial \tS\setminus\{\infty\}$, and hence in a bounded neighborhood of $L_0=L\cap (F-v)$ as well. By Proposition~\ref{prop:parallel},  $L_0$ is a $(\G_{\tS}\cap \G')$-minimal subspace. Now if $\G_\infty\cap \G_{\tS}$ is infinite, or equivalently $\infty\in \Lambda_{\rmp}(\G_{\tS})$, then $\infty\in \Lambda_{\bdp}(\G_{\tS})$.

Suppose that $\G_\infty\cap \G_{\tS}$ is finite. Then $L_0$ is a singleton set. Therefore $\Lambda({\G_{\tS}})\setminus\{\infty\}$ 
is contained in a bounded subset of $\partial \tS\setminus\{\infty\}$. Then $\infty\in\partial \tS$ is isolated from $\Lambda 
(\G_{\tS})$. Since the limit set of a non-elementary hyperbolic group is perfect, it follows that $\G_{\tS}$ is 
elementary, and hence $\G_{\tS}$ is either parabolic or loxodromic. Now suppose that $\infty\in \Lambda(\G_{\tS})$. In the 
parabolic case $\Lambda(\G_{\tS})=\{\infty\}=\Lambda_{\rmp}(\G_{\tS})$, contradicting the assumption that 
$\G_{\infty}\cap \G_{\tS}=\{e\}$. In the loxodromic case, $\infty\in\Lambda_{\rmr}(\G_{\tS})\subset \Lambda_{\rmr}(\G)$, contradicting the assumption that $\infty\in \Lambda_{\rmp}(\G)$.
\end{proof}

\begin{lemma}\label{gor} We have
$$\Lambda_{\rmr} (\G) \cap \partial \tS=\Lambda_{\rmr}(\G_{\tilde S}) .$$
\end{lemma}
\begin{proof}
  Let $\xi \in \Lambda_{\rmr} (\G) \cap \partial{\tS}$. As $\tS$ is
totally geodesic, there exists a geodesic ray, say, $\beta$, lying
in $ \tS$ pointing toward  $\xi$. Since $\xi$ is a radial limit
point, $\G \beta$ accumulates on a compact subset of $\bH^n$. By
the assumption that the natural projection map
$\G_{\tS} \ba  \tS \to X$ is  proper,
$\G_{\tS}\beta$ accumulates on a compact subset of $\tS$. This
implies $\xi\in \Lambda_{\rmr}(\G_{\tS})$. The other direction for the inclusion
is clear.
\end{proof}

In \cite{Bowditch1993}, Bowditch proved the equivalence of
several definitions of geometrically finite hyperbolic groups. In particular, we have:
\begin{theorem}[\cite{Beardon1983}, \cite{Bowditch1993}, \cite{Maskit1988}]
\label{bo}
$\G$ is geometrically finite if and only if
$\Lambda(\G)=\Lambda_{\rmr}(\G)\cup \Lambda_{\bdp}(\G)$. 
\end{theorem}

Hence, for geometrically finite $\G$, we have 
$\Lambda_{\rmp}(\G)=\Lambda_{\bdp}(\G)$.

\begin{theorem}\label{gff}
If $\G$ is geometrically finite, then $\G_{\tS}$ is geometrically finite.
\end{theorem}

\begin{proof}
 Since $\Lambda(\G)=\Lambda_{\rmr}(\G)\cup \Lambda_{\bdp}(\G)$,
it follows from Proposition \ref{sstar} and Lemma \ref{gor}
that $\Lambda(\G_{\tS})=\Lambda_{\bdp}(\G_{\tS})\cup \Lambda_{\rmr}(\G_{\tS})$,
proving the claim by Theorem~\ref{bo}.
\end{proof}

\subsection{Compactness of $\supp(\mu_E^{\PS}$) for Horospherical $E$}
\begin{theorem}[Dal'bo \cite{Dalbo2000}]\label{do} Let $\G$ be geometrically finite.
For a horosphere $\cH$ in $\T^1(\bH^n)$ based at
 $\xi\in \partial{\bH^n}$,  $E:=\p(\cH)$ is
closed in $\op{T}^1(X)$ if and only if either $\xi\notin \Lambda(\G)$ or  $\xi\in
\Lambda_{\rmp}(\G)$.
\end{theorem}

\begin{theorem}\label{hstar} Let $\G$ be geometrically finite.
If $E:=\p(\cH)$ is a closed  horosphere in $\T^1(X)$, then
$\op{supp}(\mu_E^{\PS})$ is compact.
\end{theorem}

\begin{proof} Let $\xi\in \bh$ be the base point for $\cH$.
The restriction of the visual map $\text{Vis}: v\mapsto v^+$
induces a homeomorphism $\psi:\cH\to \bh \setminus\{\xi\}$. As
$E$ is closed, by Theorem \ref{do}, either $\xi\notin
\Lambda(\G)$ or $\xi$ is a bounded parabolic fixed point.
If $\xi\notin \Lambda(\G)$,
 then $\Lambda(\G)$ is a
compact subset of $\bh \setminus \{\xi\}$. Since
$\op{supp}(\mu_E^{\PS})=\p(\psi^{-1}(\Lambda(\G)))$, it follows that $\op{supp}(\mu_E^{\PS})
$ is compact.

Suppose now that $\xi$
is a bounded parabolic fixed point.  By Definition~\ref{defbb},
 $\G_{\xi} \ba (\Lambda(\G) \setminus \{\xi\})$ is
compact.
Since $\G_{\xi}$ is discrete, it preserves the horosphere $\cH$ based at $\xi$, and  $\G_{\xi}=\G_{\cH}$. Therefore
 $\psi$ induces a homeomorphism between $\G_{\cH}\ba \cH$ and
 $\G_{\cH}\ba (\bh \setminus\{\xi\})$. It follows that
 $\G_{\cH}\ba \psi^{-1}( \Lambda(\G)\setminus \{\xi\})$
is compact and is equal to $\op{supp}(\mu_E^{\PS})$.
\end{proof}

\newcommand{\PrM}{P_{W^{\perp}}}

\section{On the cuspidal neighborhoods of $\Lambda_{\bdp}(\G)\cap \partial\tS$}\label{not:model}

\subsection{} 
Throughout this section, let $\G$ be a torsion-free discrete subgroup of $G$ and $\tS$
 a connected complete totally geodesic subspace
of $\bH^n$ such that the natural projection $\G_{\tS} \ba  \tS \to \G\ba \bH^n$ is a
proper map. 

\newcommand{\cDa}{\cD(a,\G_{\tS})}

The {\em Dirichlet domain} for $\G_{\tS}$ attached to some $a\in\tS$ is defined by
 \begin{equation} \label{eq:cDa}
 \cDa:=\{s\in \tS: d(s,a)\le d(s, \g a) \text{ for all }\g\in\G_{\tS}\}.
 \end{equation}

\begin{proposition} \label{prop:Dir-radial}
$\Lambda_{\rmr}(\G)\cap \partial\cDa=\emptyset$. 
\end{proposition}

\begin{proof}
Let $\xi\in  \Lambda_r(\Gamma)\cap \partial{\mathcal D}(a, \Gamma_{\tS})$.
 As 
 \[
 \overline{\cDa}=\cDa\cup(\partial\cDa\cap\partial\bH^{n})
 \]
is convex in $\mathbb H^n$, there exists a geodesic $\{\xi_t\}\subset \cDa$ such that $\xi_0=a$ and $\xi_\infty=\xi$. As 
$\xi\in\Lambda_r(\Gamma_{\tS})$ by Lemma~\ref{gor},
there exist sequences $t_i\to \infty$ and $\gamma_i\in \Gamma_{\tS}$ such that
$ d(\gamma_i \xi_{t_i},a)$ is uniformly bounded for all i.
Since $d(\xi_{t_i}, a)\to \infty$, it follows that for all large $i$,
$d(\xi_{t_i}, \gamma_i^{-1}a) <d(\xi_{t_i}, a)$,
yielding that $\xi_{t_i}\notin \cDa$, a contradiction.
\end{proof}

Let $\tE\subset \Th$ denote the set of all normal vectors to $\tS$.
Given $U\subset\partial{\bH^n}$, we define
\begin{equation}
\label{eq:cEU}
\cE_U =\{v\in \tilde E:\pi(v)\in \cDa,\ v^+\in U\cap\Lambda(\G)\}.
\end{equation}

\begin{remark} \label{rem:cEr}  If $\xi\in \Lambda_{r}(\G)\cap \partial\tS$, then there exists a neighborhood $U$ of $\xi$ in 
$\partial\bH^{n}$ such that $\cE_{U}=\emptyset$; to see this, note that if there exists a sequence $\{v_{i}\}\subset \tE$ such that 
$v_{i}^{+}\to \xi$, then $\pi(v_{i})\to \xi$, and hence by Proposition~\ref{prop:Dir-radial} $\pi(v_{i})\not\in \cDa$ for all large $i$. 
\end{remark}

In view of Theorem~\ref{bo} and Remark~\ref{rem:cEr}, the main goal of this section is to describe the structure of $\cE_U$ for a 
neighborhood $U$ of a point in $\Lambda_{\bdp}(\G)\cap \partial\tS$ and to compute the measure $\mu_{\tE}^{\PS}(\cE_U)$.

In this section, we will use the upper half space model $\bH^n=\R^{n-1}\times \R_{>0}$ and first we assume that
\[
\infty\in\partial{\tS}\cap \Lambda_{\rmp}(\G).
\]
Here $\R^{n-1}$ is to be treated as an affine space till we make a choice of the origin. Hence $\tilde S$ is a vertical plane over the 
affine subspace $\partial{\tS}\setminus\{\infty\}$ of $\R^{n-1}$.  For any affine subspace $F$ of $\R^{n-1}$, let 
$P_F:\R^{n-1}\to F$ denote the orthogonal projection. Let
\begin{equation} \label{eq:bh}
\bfb:\R^{n-1}\times \R_{>0}\to \R^{n-1} \text{ and } \bfh:\R^{n-1}\times \R_{>0}\to \R_{>0}
\end{equation}
denote the natural projections. 

Let $\G'=\G'_\infty$ be a normal abelian subgroup of $\G_\infty$ with finite index, as in section \ref{subsec:bieberbach} and fix a 
$\G'$-minimal subspace $L$ of $\R^{n-1}$.
Noting that $\bfb(a)\in \partial{\tS}\setminus \{\infty\}$, we choose $0:=P_{L}(b(a))$, the origin of $\R^{n-1}$. This choice of $0$ 
makes $L$ a linear subspace. Set
$W:=\{v-\bfb(a): v\in \partial{\tS}\setminus\{\infty\}\}$, a linear subspace of $\R^{n-1}$, and $\Delta:=\G_{\tS}\cap \G'$. By 
Proposition~\ref{prop:parallel}, $L_0:=L\cap W$ is a $\Delta$-minimal (linear) subspace.

Let $V$ be the largest affine subspace of $\R^{n-1}$ such that $\Delta$ acts by translations on $V$. Then $0\in L\subset V$ 
and $V$ is the union of all (parallel) $\Delta$-minimal subspaces of $\R^{n-1}$. There exist group homomorphisms 
$\tau:\Delta \to L_0\subset \R^{n-1}$ and $\theta:\Delta\to O(n-1)$ such that for any $\g\in \Delta$,
\begin{equation} \label{eq:lambda}
\g(x)=\theta(\g)(x)+\tau(\g),  \text{ for all $x\in\R^{n-1}$}.
\end{equation}

We note that $V=\{x\in\R^{n-1}:\theta(\Delta)x=x\}$, and $V^\perp$ is the sum of all nontrivial (two-dimensional) 
$\theta(\Delta)$-irreducible subspaces of $\R^{n-1}$. 

 \begin{lemma}
\label{eq:WV} 
\begin{enumerate}
\item $W=(W\cap V)+(W\cap V^{\perp})$;
 \item $W^\perp=(W^\perp\cap V)+(W^\perp\cap V^\perp)$.
\end{enumerate}
\end{lemma}

\begin{proof}
Put $F=\partial\tS\setminus\{\infty\}$. Then $\Delta F=F$, and there exists a $\Delta$-minimal affine subspace 
$L_{\tS}\subset F$. Choose $0'\in L_{\tS}\subset F\cap V$. Since $W$ is a parallel translate of $F$ through $0$, we have 
$W=F-0'$. As in the proof of Proposition~\ref{prop:parallel}, $\Delta(W)=W$. Since $0\in W$, we have $\theta(\Delta)(W)=W$, and hence $\theta(\Delta)(W^\perp)=W^\perp$. Thus $W\cap V$ is the set of fixed points of $\theta(\Delta)$ in $W$, and
its ortho\-complement in $W$ is the sum of all nontrivial $\theta(\Delta)$-irreducible subspaces of $W$ which is same as 
$W\cap V^\perp$. Therefore (1) follows. And (2) is proved similarly.
\end{proof}

For any $v\in \tE$, $\pi(v)\in\tS$. By abuse of notation,
we write $\bfb(v):=\bfb(\pi(v))\in \partial{\tS}\setminus\{\infty\}$ and $\bfh(v):=\bfh(\pi(v))\in \R_{>0}$.
We denote by $\sigma(v)\in W^{\perp}$ the unique element in $W^\perp$ of norm one satisfying
\begin{equation} \label{eq:v+}
v^+:=\V(v)=\bfb(v)+\bfh(v)\sigma(v).
\end{equation}


\subsubsection*{Bounded parabolic assumption} For the rest of this section we will further assume that $\infty\in\partial(\tS)\cap
\Lambda_{\bdp}(\G)$. Hence there exists $R_0>0$ such that
 for all $x\in \Lambda(\G)\cap\R^{n-1}$,
\begin{equation} \label{eq:R0}
\norm{P_{L^\perp}(x)}\leq R_0,
\end{equation}
where $\norm{\cdot}$ denotes the Euclidean norm.

\begin{lemma} \label{eq:Vperp} For any $v\in \tilde E$ with
$v^+\in \Lambda(\G)$,
\begin{equation*}
\norm{P_{V^{\perp}}(\bfb(v))}\leq R_0.
\end{equation*}
\end{lemma}

\begin{proof}
Let $0'\in V$ be as in the proof of Lemma~\ref{eq:WV}. Since $\bfb(v)-0'\in W$ and $0'\in V$, we have $P_{V^{\perp}}(0')=0$ and 
by Lemma \ref{eq:WV},
\[
P_{V^{\perp}}(\bfb(v))=P_{V^{\perp}}(\bfb(v)-0')\in W \text{ and }
P_{V^\perp}(\sigma(v))\in W^\perp.
\]
Therefore by \eqref{eq:v+}, 
$\norm{P_{V^\perp}(\bfb(v))}\leq \norm{P_{V^\perp}(v^+)}$. Since $L\subset V$, we have $V^\perp\subset L^\perp$, and hence 
by \eqref{eq:R0}, $\norm{P_{V^\perp}(v^+)}\leq\norm{P_{L^\perp}(v^+)}\leq R_0$.
\end{proof}

\newcommand{\hyp}{{\rm hyp}}

\begin{proposition}  \label{prop:P_LS}\label{eq:R1}
There exists $R_1>0$ such that for all $v\in \cE_{\partial{\bH^n}}$,
\begin{equation*}
 \norm{P_{L_0}(v^+)}\leq R_1.
\end{equation*}
\end{proposition}

\begin{proof}
Let $v\in \cE_{\partial{\bH^n}}$.
 Then for all $\g\in \Delta\subset \G_{\tS}$,
\begin{multline}
d_\hyp(\pi(v),a)\leq d_{\hyp}(\g \pi(v), a) \\
\Rightarrow\  d_{\eucl}(\bfb(v),\bfb(a))\leq d_{\eucl}(\g \bfb(v),\bfb(a)). \label{eq:dirichlet}
\end{multline}
Now $\bfb(v)-\bfb(a)\in W$, $L_0\subset W\cap V$, and $P_{L_0}(\bfb(a))=0$. As
\[
W=(V^\perp\cap W)+ L_0 + (W\cap V\cap L_0^\perp),
\]
which is a sum of $\theta(\Delta)$-invariant orthogonal subspaces of $W$, we get
\begin{multline*}
\g\bfb(v)-\bfb(a)=\left[\theta(\g)P_{V^\perp}(\bfb(v))-P_{V^\perp}(\bfb(a))\right]
+ \\  \left[P_{L_0}(\bfb(v))+\tau(\g)\right]
+ P_{V\cap W\cap L_0^\perp}(\bfb(v)-\bfb(a)).
\end{multline*}
Comparing this with \eqref{eq:dirichlet}, for any $\g\in \Delta$  we get
\begin{multline} \label{eq:PLSR1}
\norm{P_{L_0}(\bfb(v))}^2\leq \norm{\theta(\g)P_{V^\perp}(\bfb(v))-P_{V^\perp}(\bfb(a))}^2 + \\
\norm{P_{L_0}(\bfb(v))+\tau(\g)}^2.
\end{multline}
Since $\tau(\Delta)$ is a lattice in  $L_0$,
the radius of the smallest ball containing a fundamental domain of $\tau(\Delta)$ in $L_0$
 is finite, which we denote by $R_2$. Then by \eqref{eq:Vperp} and \eqref{eq:PLSR1}, we conclude that
\[
\norm{P_{L_0}(v^+)}^2=\norm{P_{L_0}(\bfb(v))}^2 \leq (R_0+\norm{\bfb(a)})^2+R_2^2.
\]
By setting $R_1=( (R_0+\norm{\bfb(a)})^2+R_2^2 )^{1/2}$, we finish the proof.
\end{proof}

\subsection{Co-rank at $\infty$ and the structure of $\mathcal E_U$} Set
\[
r_\infty:=\rank(\G_{\infty})-\rank(\G_{\infty}\cap \G_{\tS}).
\]
More precisely,
$r_\infty=\rank(\G')-\rank(\Delta)=\dim(L)-\dim(L_{0})$.
\begin{proposition} \label{prop:internal}
If $r_\infty=0$, then there exists a neighborhood $U$ of $\infty$ in $\partial{\bH^n}$ such that $\cE_U=\emptyset$, where $
\cE_{U}$ is defined in \eqref{eq:cEU}.
\end{proposition}
\begin{proof}
As $r_\infty=0$, we have $L=L_0$. Therefore, for all $x\in \Lambda(\G)\cap\R^{n-1}$,
\[
\norm{P_{L_0^\perp}(x)}\leq R_0.
\]
Hence for any $v\in \cE_{\partial\bH^n}$, by Proposition~\ref{eq:R1},
\[
\norm{v^+}^2=\norm{P_{L_0}(v^+)}^2 + \norm{P_{L_0^\perp}(v^+)}^2
\leq R_1^2+R_0^2.
\]
Let $U=\{x\in\R^{n-1}:\norm{x}^2>R_0^2+R_1^2\}\cup\{\infty\}$. Then $\cE_U=\emptyset$.
\end{proof}

In the rest of this section, we now consider the case when
$$r:=r_\infty\ge 1.$$

\begin{notation}\rm
 For any $\bss=(s_{1},\ldots,s_{r})\in \br^{r}$ and an ordered $r$-tuple $(w_1,\ldots, w_r)$
 of vectors in $\R^{n-1}$,
  we set $\bss
\cdot \bw:=s_{1}w_{1}+\cdots+s_{r}w_{r}\in\R^{n-1}$, $\R^{r}\bw:=\{\bss\cdot \bw:\bss\in\br^{r}\}$
and $\abs{\bss}=\max(\abs{s_{1}},\ldots,\abs{s_{r}})$.
   For $\bk\in\Z^{r}$ and an ordered $r$-tuple $\bg=(\g_{1},\ldots,\g_{r})$ of elements of $G$,
   we write $\bg^{\bk}=\g_{1}^{k_{1}}\cdots\g_{r}^{k_{r}}\in G$.
   \end{notation}

 Fix an ordered $r$-tuple $\bg=(\g_{1},\ldots,\g_{r})$ of elements of $\G'=\G'_{\infty}$ such that the subgroup generated by $\bg
\cup \Delta$ is of finite index in $\G'$. For each $\g_i$, there exists $w_{i}\in L$ and $\sigma_{i}\in O(n-1)$ such that
 for all $x\in\R^{n-1}$, $$ \gamma_i( x)=\sigma_{i}(x)+w_{i}.$$
Moreover $\sigma_i$ and the translation by $w_i$ commutes, and hence
for any $k\in \z$, $\gamma_i^k(x)=\sigma_i^k(x) + kw_i$.

Setting $\bw=(w_1, \ldots, w_r)$ and ${\bf{\sigma}}=(\sigma_1, \cdots, \sigma_r)$,
we have that for any $x=y+z\in \R^{n-1}$ with $y\in L^\perp$ and $z\in L$ and $\bk=(k_1, \cdots, k_r) \in\Z^{r}$,
\begin{equation}  \label{eq:bgk}
\bg^{\bk} (x)= {\bf{\sigma}}^{\bk}(y) + z +\bk \cdot \bw.
\end{equation}

Let $R_{0}$ and $R_1$ be as in \eqref{eq:R0} and in Proposition~\ref{prop:P_LS} respectively. Set 
\[
B_{0}:=\{x\in L^\perp: \|x\|\le R_0\} \quad \text{and}\quad
   B_{1}:=\{x\in L_0: \|x\|\le R_1\}.
\]
   
Let $M_{1}:=L\cap L_0^{\perp}$. Then $\Z^{r}\cdot P_{M_{1}}(\bw)$ is a lattice in $M_{1}=\R^{r}P_{M_{1}}(\bw)$, which admits a 
relatively compact fundamental domain, say $F_{1}$.
Let $F_{2}$ be a relatively compact fundamental domain for the lattice $\tau(\Delta)$ in $L_0$. We define the following relatively 
compact subset of $\R^{n-1}$:
\begin{equation}\label{cf} 
\cF:=B_{0}+(B_{1}+F_{2})+F_{1}\subset L^\perp +L_0 +(L\cap L_0^\perp).
\end{equation}
By \eqref{eq:bgk}, 
\[
\bg^{\bk}\cF=\cF+\bk\cdot\bw.
\] 

For related variable quantities $x\geq 0$ and $y\geq 0$,
the symbol $x\gg y$ means that there exists a constant
$C>0$ such that  for all related $x$ and $y$,
$x \geq C y $, and the symbol $x\asymp y$ means that
$x\gg y$ and $y\gg x$.

\begin{proposition}  \label{eq:v+=kw} \label{eq:E_U} \label{eq:UT}
There exists  $c_0\ge 1$ such that for all sufficiently large $N\ge 1$,
\begin{equation}
\V(\cE_{U_{c_0N}})\subset \cup_{\abs{\bk}\geq N} \Delta\bg^{\bk}(\cF)
\end{equation} where $U_{c_{0}N}=\{x\in\R^{n-1}:\norm{x}\ge c_0N\}.$
\end{proposition}

\begin{proof}
 Since $\R^{n-1}=L^{\perp}+L_0+M_{1}$ for $M_{1}=L_0^{\perp}\cap L$, we have
for any $v\in \R^{n-1}$, \begin{equation} \label{eq:LSv+}
v^{+}=P_{L^{\perp}}(v^{+})+P_{L_{0}}(v^{+})+P_{M_{1}}(v^{+}).
\end{equation}

Let $v\in \cE_{\partial\bH^{n}}$. By \eqref{eq:R0} and Proposition \ref{eq:R1},
\begin{equation} \label{eq:v+B0B1}
P_{L^{\perp}}(v^{+})\in B_{0} \text { and } P_{L_{0}}(v^{+})\in B_{1}.
\end{equation}

In order to control $P_{M_{1}}(v^{+})$,  let $\bk=\bk(v^+)\in\Z^{r}$ be
 such that $P_{M_{1}}(v^{+})\in \bk\cdot P_{M_{1}}(\bw)+F_{1}$, where $\bk$ is uniquely determined.
Let $\lambda_{\bk}\in \Delta$  be such that
\begin{equation*} \label{eq:M1}
P_{L_0}(\bk \cdot \bw)\in  \tau(\lambda_{\bk})+F_{2}. \end{equation*}

Since $ \bk\cdot P_{M_{1}}(\bw)-\bk\cdot\bw =P_{L_0}(\bk \cdot \bw)$,
 $$P_{M_1}(v^+) \in (\bk\cdot P_{M_{1}}(\bw)-\bk\cdot\bw) +(F_1 +\bk\cdot\bw)
 \in  \tau(\lambda_{\bk})+F_{2} + (F_1 +\bk\cdot\bw).$$

Therefore by \eqref{eq:LSv+}, for $\bk=\bk(v^+)$, we have
\begin{equation}\label{c1}
v^{+}\in \cF+\bk\cdot\bw+\tau(\lambda_{\bk})=\lambda_{\bk}\bg^{\bk}(\cF).
\end{equation}

Since $P_{M_{1}}:\R^{r}\bw\to M_{1}$ is a linear isomorphism,
there exists $N_1\ge 1$ such that for all $\bk\in \Z^r$ with $|\bk|>N_1$,
\begin{equation} \label{eq:PM1}
\norm{P_{M_{1}}(\bk\cdot\bw)}\asymp \abs{\bk}.
\end{equation}

 By \eqref{eq:LSv+} and \eqref{eq:v+B0B1}, $\norm{P_{M_{1}}(v^{+})-v^{+}}\leq R_{0}+R_{1}$ and
$P_{M_{1}}(v^{+})-P_{M_{1}}(\bk\cdot\bw)\in F_{1}$ for $\bk=\bk(v^+)$. It follows that there exists a constant $B>0$ such that for 
all  $v\in \cE_{\partial\bH^{n}}$,
\[
\norm{P_{M_{1}}(\bk\cdot\bw)}- B  \le \norm{v^{+}}\le  \norm{P_{M_{1}}(\bk\cdot\bw)}+B
\]
where $\bk=\bk(v^+)$.
Hence by \eqref{eq:PM1}, there exists $N_2\ge 1$ such that for all $v\in \cE_{\partial\bH^{n}}$
 with $|\bk(v^+)|>N_2$,
$$\norm{v^{+}}\asymp  \abs{\bk (v^+)}.$$
In view of \eqref{c1}, this finishes the proof.
\end{proof}

\begin{lemma} \label{claim3}\label{eq:hv-v+} There exists $N_0\ge 1$ such that
for all $\bk\in \Z^r$ with $|\bk|>N_0$, the following hold:
\begin{enumerate}
\item For $\xi\in \bg^{\bk}(\cF)$,
$ \norm{\xi}\asymp \abs{\bk}$.
\item For $v\in \tE$ with $v^+\in \bg^{\bk}(\cF)$,
$\bfh(v)\asymp \abs{\bk}$.
\end{enumerate}
\end{lemma}
\begin{proof}
If $\xi\in \bg^{\bk}\cF$, then
$\norm{\xi-\bk\cdot\bw}\le \diam_{\eucl}(\cF)$. Hence
$\norm{\xi}\asymp \norm{\bk\cdot\bw}\asymp \abs{\bk}$,
proving (1).

For $v\in \tE$ such that $v^+\in \bg^{\bk}(\cF)$,
by \eqref{eq:v+},
\begin{equation*}
 \bfh(v)\asymp\norm{P_{W^{\perp}}(v^{+})} \asymp P_{W^{\perp}}(\bk\cdot\bw).
\end{equation*}

Since $W\cap L=L_{0}$ and $L=L_{0}\oplus\R^{r}\bw$, the map
$P_{W^{\perp}}:\R^{r}\bw\to W^{\perp}$ is injective. Therefore
\begin{equation*} \label{eq:PWperp}
\norm{P_{W^{\perp}}(\bk\cdot\bw)}\asymp \norm{\bk\cdot\bw}\asymp \abs{\bk},
\end{equation*}
from which (2) follows.
\end{proof}

Let $o=(0,1)\in \R^{n-1}\times \R_{>0}$. For $T\ge 1$, put
\begin{equation}\label{eq:bt3}
\cB_{T}=\{v\in\tE: \beta_{\infty}(o,\pi(v)) \ge \log T\},
\end{equation}
that is, $\cB_T$ is the intersection with $\tE$ of a horoball based at $\infty$.
We note that for $v\in \tE$, $\beta_{\infty}(o,\pi(v)) =\log h(v)$.
Hence in the vertical plane model of $\tS$, $\cB_T$ consists of vectors $v\in \tE$
whose base points have the Euclidean height at least $T$.

\begin{proposition}\label{prop:cB} Let $\cF_0:=B_0+F_2+F_1$.
Then  $\nu_o(\cF_0)> 0$, and
for all sufficiently large $T$, there exists $N\asymp T$ such that
\begin{equation} \label{eq:cBT}
\V (\cB_T) \supset \cup_{\abs{\bk}\geq N} \bg^{\bk}(\cF_0).
\end{equation}
\end{proposition}

\begin{proof}
Since $\Delta F_{2}=L_{0}$ and $\bg^{\Z^{r}}F_{1}=M_{1}$,
\[
\Delta(\cup_{\bk\in\Z^r} \bg^\bk(\cF_{0}))=B_{0}+L_0+M_1=B_{0}+L\supset \Lambda(\G)\setminus\{\infty\}.
\]
Therefore if $\nu_{o}(\cF_0)=0$, then by the conformality, it follows that $\nu_{o}(\Lambda(\G)\setminus\{\infty\})=0$. Since $\G$
does not fix $\infty$, by the $\Gamma$-invariance of $\{\nu_x\}$, we get
$\nu_{o}(\Lambda(\G))=0$, which is a contradiction, proving the first claim.

If $v\in \tE$ and $v^+\in \bg^{\bk}(\cF_0)$, then by Lemma \ref{claim3},
 $h(v)\asymp \abs{\bk}$. If $h(v)>T$, then $v\in\cB_{T}$. Therefore \eqref{eq:cBT} holds for suitable $N\asymp T$.
\end{proof}

\subsection{Estimation of $\mu_{\tE}^{\PS}(\mathcal E_U)$}
Let $\Vinv:\R^{n-1}\setminus\partial\tS\to \tE$ be the inverse of the restriction of the visual map $\V:\tE\to \partial\bH^{n}\setminus
\partial\tS=\R^{n-1}\setminus \partial\tS$.
\begin{lemma}\label{eq:e13} There exists $N_1\ge 1$ such that
for all $\bk\in \Z^r$ with $\abs{\bk}>N_1$,
\begin{equation*}
\int_{\xi\in \bg^{\bk}\cF }e^{\delta\beta_{\xi}(o,\pi(\Vinv(\xi)))}\;d  \nu_o(\xi)\asymp
 \abs{\bk}^{-\delta} .
 \end{equation*}
\end{lemma}

\begin{proof}
We have $\norm{P_{W^{\perp}}(\bk\cdot\bw)}\asymp \abs{\bk}$. Hence for sufficiently large $\abs{\bk}$, we have that $\bg^{\bk}
\cF\cap \partial\tS=\emptyset$. Note that the Euclidean diameter of the horosphere based at $\xi$ and passing through
$o=(0,1)\in\R^{n-1}\times\R_{>0}$ is $1+\norm{\xi}^2$. And the diameter of the horosphere
based at $\xi$ and passing through  $\pi(\Vinv(\xi))$ is $\bfh(\pi(\Vinv(\xi)))$. Therefore the
signed hyperbolic distance of the segment cut by these two horospheres on the vertical
geodesic ending in $\xi$ is 
\begin{equation*} \label{eq:beta-xi}
\beta_{\xi}(o,\pi(\Vinv(\xi)))=\log(1+\norm{\xi}^{2})-\log(\bfh(\pi(\Vinv(\xi)))).
\end{equation*}
Hence by Lemma \ref{claim3},
\begin{equation} \label{eq:e111}
e^{\delta\beta_{\xi}(o,\pi(\Vinv(\xi)))}=\Bigl(\frac{1+\norm{\xi}^{2}}{\bfh(\pi(\Vinv(\xi)))}
\Bigr)^{\delta}\asymp \abs{\bk}^{\delta}.
\end{equation}

By conformality and $\G$-invariance of Patterson-Sullivan density $\{\nu_{x}\}$,
\begin{align}
 \nu_o(\bg^{\bk}\cF)=(\bg^{-\bk}\nu_o)(\cF)=\nu_{\bg^{-\bk}\cdot o}(\cF)
& =  \int_{\xi\in \cF} \frac{d\nu_{\bg^{-\bk}\cdot o}}{d\nu_{o}}(\xi)\, d\nu_o(\xi)
\notag \\
&= \int_{\xi\in \cF} e^{-\delta\beta_{\xi}(\bg^{-\bk} o,o)} d\nu_o(\xi).
\label{eq:nuvol}
\end{align}

We note that the horosphere based at $\xi$ passing through $\bg^{-\bk}o=(-\bk\cdot\bw,1)\in
 \R^{n-1}\times \R_{>0}$ has diameter $1+\norm{\xi+\bk\cdot\bw}^{2}$. Therefore
\[
\beta_{\xi}(\bg^{-\bk}o,o)=\log(1+\norm{\xi+\bk\cdot\bw}^{2})-\log(1+\norm{\xi}^2),
\]
and hence, since $\norm{\bk\cdot\bw}\asymp \abs{\bk}$ for all large $|\bk|$,
  we have, for any $\xi\in \cF$,
\[
e^{-\delta\beta_{\xi}(\bg^{\bk}o,o)}=\Bigl(\frac{1+\norm{\bk\cdot\bw-\xi}^{2}}{1+\norm{\xi}^{2}}
\Bigr)^{-\delta}\asymp \abs{\bk}^{-2\delta},
\]
Since $\nu_o(\cF)>0$ by Proposition~\ref{prop:cB},
we deduce from \eqref{eq:nuvol} that 
\begin{equation*} \label{e12}
\nu_{o}(\bg^{\bk}\cF)\asymp \abs{\bk}^{-2\delta}\nu_{o}(\cF)\asymp \abs{\bk}^{-2\delta}.
\end{equation*}
Together with \eqref{eq:e111}, this proves the claim.
\end{proof}

Let $\p:\tE\to \G_{\tE}\bs \tE$ be the natural quotient map. We note that $\G_{\tE}=\G_{\tS}$.
From \S\ref{subsec:muE}, we recall that the measure $\muPS_{\tE}$, which is $\G_{\tE}$ invariant, naturally induces a measure 
on $\G_{\tE}\bs \tE$. The pushforward of this measure from $\G_{\tE}\bs \tE$ to $E=\p(\tE)$ is $\muPS_{E}$.

Recall the definition of $c_0>0$ and $U_{c_0N}$
from Proposition \ref{eq:UT}.
\begin{proposition}\label{prop:mupsEbounds}
\begin{enumerate}
 \item  For all sufficiently large $N\ge 1$, we have
\begin{equation*} \label{eq:mupsE1}
\muPS_{E}(\p({\cE}_{U_{c_0N}} ))\ll \sum_{\bk\in \z^r\setminus\{0\}} \abs{\bk}^{-\delta}.
\end{equation*}

\item  For all sufficiently large $T\ge 1$, we have
\begin{equation*}\label{eq:mupsE2}
\muPS_E(\p(\cB_T))\gg \sum_{\bk\in\Z^r\setminus\{0\}} \abs{\bk}^{-\delta}
\end{equation*}
\end{enumerate}
 \end{proposition}

\begin{proof}
By  Proposition \ref{eq:v+=kw} and by Lemma~\ref{eq:e13},
 for all large $N\geq 1$,
\begin{align*}
\muPS_E(\p(\cE_{U_{c_0N}}))& \leq \sum_{\abs{\bk}\geq N} \muPS_{\tE}(\Vinv(\bg^{\bk}(\cF)))
\notag \\
&=\sum_{\abs{\bk}\ge N}\int_{\xi\in \bg^{\bk}\cF} e^{\delta\beta_{\xi}(o,\pi(\Vinv(\xi)))}\;d  \nu_o(\xi) \\&
\ll \sum_{\bk\geq N} \abs{\bk}^{-\delta},
\end{align*}
proving (1).

Consider the natural quotient map
\begin{equation} \label{eq:pinfty}
\p_{\infty}:(\G_{\tE}\cap\G_{\infty})\bs \tE \to \G_{\tE} \bs \tE.
\end{equation}

Since $\infty\in \Lambda_{\bdp}(\G)$, there exists $T_{0}>0$ such that $\p_\infty$ restricted to $(\G_{\tE}\cap\G_{\infty})\bs 
\cB_{T}$ is proper and injective for all $T\geq T_{0}$.

Now since $F_2$ is a fundamental domain for $\Delta$ action on $L_{\tS}$ and $F_1$ is a fundamental domain for the action of 
$\{\bg^{\bk}:\bk\in\Z^r\}$ on $M_1$, the quotient map $\tE\to \Delta\bs \tE$ is injective on $\cup_{\abs{\bk}\geq N} \Vinv(\bg^{\bk}
(\cF_0))$.
Since $[\G_{\tS}\cap\G_{\infty}:\Delta]<\infty$ and
 $\p_{\infty}$ is injective on $(\G_{\tS}\cap \G_{\infty})\bs \cB_{T}$, for all sufficiently large $T\gg1$,
\begin{equation*}
\begin{array}{ll}
\muPS_{E}(\p(\cB_{T}))& =\muPS_{\G_{\tE}\bs\tE}(\p_{\infty}(\cB_{T})); \quad \text{see \eqref{eq:muEbar}}\\
& \gg \sum_{\abs{\bk}\geq N} \muPS_{\tE}(\Vinv(\bg^{\bk}(\cF_0))); \quad \text{by Proposition~\ref{prop:cB}}\\
& \gg \sum_{\bk\geq N} \abs{\bk}^{-\delta};  \quad \text{by Lemma \eqref{eq:e13}.}
\end{array}
\end{equation*}
This proves (2).
\end{proof}



\section{Parabolic co-rank and Criterion for finiteness of $\mu^{\PS}_E$} \label{sec:5}\label{subsec:corank}
Let $\G$ be non-elementary torsion free discrete subgroup of $G$. Let $\tS$, $\tE$ and $E$ be as in section  \ref{not:model}. In 
particular, $\tS$ is totally geodesic and the map $\G_{\tS}\bs \tS\to \G\bs \bH^{n}$ is proper.

\begin{definition}[Parabolic corank]
Define
\begin{equation*}
 \parcorank(\G_{\tS})=\max_{\xi\in\Lambda_{\rmp}(\G)\cap \partial (\tS)} \left(\rank(\G_\xi)-\rank(\G_\xi \cap \G_{\tS})\right).
 \end{equation*} 
 When $\Lambda_{\rmp}(\G)\cap \partial ( \tS)=\emptyset$, we set $\parcorank(\G_{\tS})=0$.
\end{definition}

\begin{lemma}[Corank Lemma]\label{corank} 
$\parcorank(\G_{\tS}) \leq \codim(\tS)$.
\end{lemma}

\begin{proof}
Suppose $\infty\in \Lambda_{\rmp}(\G)\cap\partial\tS$. Let $L$ be a $\G'_{\infty}$-minimal subspace of $\partial\bH^n\setminus
\infty$ and let $W$ be the intersection of a translate of $\partial\tS \setminus\{\infty\}$ through a point in $L$ . Then by 
Proposition~\ref{prop:parallel},
$\rank(\G'_{\infty})-\rank(\G_{\tS}\cap\G'_{\infty})=\dim(L)-\dim(W)\leq (n-1)-\dim(\partial\tS)
=n-\dim\tS$.
\end{proof}

\subsection{Finiteness criterion for geometrically finite $\G$}  For the rest of this section we further assume that $\G$ is 
geometrically finite.

\begin{theorem}\label{cstar-sb} \label{sb}
$\parcorank(\G_{\tS})=0 \Leftrightarrow \op{supp}(\mu_E^{\PS})$ is
compact.
\end{theorem}
\begin{proof}
Suppose that $\supp(\muPS_{E})$ is not compact. Fix a Dirichlet domain $\cDa$ for the $\G_{\tS}$ action on $\tS$. Since the 
projection of $\G_{\tE}\bs\tE$ into $\G\bs\T^{1}(\bH^{n})$ is proper, there exists an unbounded sequence $v_m\in \tilde E$ with 
$\pi(v_m)\in \cDa$ and $v_m^+\in \Lambda(\G)$. Since $\Lambda(\G)$ is compact, by passing to a subsequence, we assume 
that $v_m^+\to \xi$ for some $\xi\in \Lambda(\G)$. Thus for any neighborhood $U$ of $\xi$ in $\partial\bH^{n}$, we have 
$v_{m}\in \cE_{U}$ for all large $m$. 

Consider the upper half space model $\bH^n=\R^{n-1}\times \R_{>0}$ with $\xi$ identified with $\infty$ as in \S\ref{not:model}. 
As $v_{m}^{+}\to\xi=\infty$, by \eqref{eq:v+} we have $\norm{\bfb(v_{m})}\to \infty$ or
 $\bfh(v_{m})\to\infty$ (see \eqref{eq:bh} for notation) and hence $\pi(v_m)\to \infty=\xi$.
Therefore $\xi\in \partial(\cDa)$. By Proposition~\ref{prop:Dir-radial}, $\xi\not\in \Lambda_{\rmr}(\G)$. Since $\G$ is 
geometrically finite, by Theorem~\ref{bo}, $\xi \in \Lambda_{\bdp}(\G)\cap \partial\cDa$. Now by Proposition~\ref{prop:internal}, 
 $\parcorank(\G_{\tS})\ne 0$.

To prove the converse, suppose that there exists $\xi\in \Lambda_{\bdp}(\G)\cap\partial{\tS}$ such that $r=\rank(\G_{\xi})-
\rank(\G_{\xi}\cap \G_{\tS})\geq 1$. Without loss of generality, we may assume $\xi=\infty$.
 Fix $T_{0}>1$. The map $\p_\infty$ as in \eqref{eq:pinfty} restricted to $(\G_{\tE}\cap \G_\infty)\bs \cB_{T_{0}}$ is proper 
 (see \eqref{eq:bt3} for notation).
 Therefore for any compact subset $\Omega$ of $\G_{\tE}\backslash \tE$, we have
 $\p_\infty(\cB_{T})\cap \Omega=\emptyset$ for all sufficiently large $T>T_0$.
  By Proposition \ref{prop:mupsEbounds}(2),
\[
\muPS_E(\p(\cB_{T}))\gg \sum_{\bk\in\Z^r\setminus\{0\}} \abs{\bk}^{-\delta}>0.
\]
Therefore $\supp(\muPS_E)$ intersects $\p(\cB_{T})$ for all large $T\gg 1$. Since the projection of $\G_{\tE}\bs\tE$ into 
$\G\bs \T^{1}(\bH^{n})$ is proper, $\supp(\muPS_E)$ is noncompact.
\end{proof}

\begin{theorem}  \label{stfinite}
$\parcorank(\G_{\tS})<\delta \Leftrightarrow \abs{\muPS_E}<\infty$.
\end{theorem}

\begin{proof}
Suppose that $\parcorank(\G_{\tS})\ge \delta>0$. Then there exists $\xi\in \Lambda_{\bdp}(\G)\cap\tS$ such that 
$r:=\rank(\G_{\xi})-\rank(\G_{\xi}\cap \G_{\tS})\geq \max\{\delta,1\}$. Without loss of generality, we may assume $\xi=\infty$. By the 
second part of the proof of Theorem \ref{sb}, for all sufficiently large $T\gg 1$, since $r\geq \delta$,
\[
\abs{\muPS_{E}}\geq \muPS_E(\p(\cB_{T}))\gg \sum_{\bk\in\Z^r\setminus\{0\}} \abs{\bk}^{-\delta}=\infty.
\]

Now suppose that $\parcorank(\G_{\tS})<\delta$.
By the compactness of $\Lambda(\G)\cap\partial(\cDa)$,
where $\cDa$ is a fixed Dirichlet domain for $\G_{\tS}$,
to prove finiteness of $\muPS_E$, it suffices to show that
for every $\xi\in \Lambda(\G)\cap \partial(\cDa)$,
 there exists a neighborhood $U$ of $\xi$ in $\partial{\bH^n}$ such that
$\muPS_E(\p(\cE_U))<\infty$ with $\cE_U$ defined as in \eqref{eq:cEU}. By Proposition~\ref{prop:Dir-radial} and 
Theorem~\ref{gff}, $\xi\in\Lambda_{\bdp}(\G)$. Let $r:=\rank(\G_{\xi})-\rank(\G_{\xi}\cap \G_{\tS})$.
If $r=0$ then by Proposition~\ref{prop:internal}, there exists a neighborhood 
$U$ of $\xi$ such that $\cE_U=\emptyset$. Therefore we assume that $\delta>r\geq 1$. 
 By Proposition~\ref{prop:mupsEbounds}(1),  there exists a neighborhood $U$ of $\xi$ such that
\[
\muPS_E(\p(\cE_U))\ll \sum_{\bk\in\Z^r\setminus\{0\}} \abs{\bk}^{-\delta}<\infty.
\]
\end{proof}

%
\subsection{Finiteness of $\abs{\mu_E^{\Leb}}$ and $\abs{\mu_E^{\PS}}$.} 
\begin{theorem} \label{infvol}
Let $\tS$ be any totally geodesic immersion in $\bH^{n}$.
Suppose that $\dim(\tS)\geq (n+1)/2$ and $\abs{\mu_E^{\Leb}}<\infty$. Then $ \abs{\mu_E^{\PS}}<\infty$.
\end{theorem}

\begin{proof} Since $\G_{\tS}$ is a lattice in $G_{\tS}$,  $\Lambda(\G_{\tS})=\partial{\tS}$.  Hence by 
Theorem~\ref{thm:finite-closed}, the natural map $\p:\G_{\tS}\bs \tS\to \G\bs\bH^{n}$ is proper.

Let $k:=\op{dim}(\tS)\geq \lceil (n+1)/2\rceil\geq 2$. By a property of a lattice in rank one Lie group $G_{\tS}$,  
$\rank(\G_{\tS}\cap \G_{\xi})=k-1$ (cf.~\cite[\S13.8]{Raghunathanbook}). Therefore by Lemma~\ref{corank},
\[
r:=\parcorank(\G_{\tS})\leq n-k\leq n-(n+1)/2\leq (n-1)/2.
\]
Let $\xi\in \partial (\tilde{S})\cap\Lambda_{\bdp}(\G_{\tS})$ be such that $\rank(\G_{\tS}\cap\G_{\xi})=r$. Then 
$\rank(\G_{\xi})\geq (k-1)+r$. By  a result of Dalbo, Otal and Peign~\cite[Proposition~2]{DalboOtalPeign},
\[
\delta> \rank(\G_{\xi})/2\geq ((k-1)+r)/2\geq (k-1+(n-k))/2=(n-1)/2\geq r.
\]
Hence by Theorem~\ref{stfinite}(2), $|\muPS_{E}|$ is finite.
\end{proof}

As an immediate corollary, we state:
\begin{corollary}
Let $n=2, 3$. Then $\abs{\mu_E^{\Leb}}<\infty$ implies that $\abs{\mu_E^{\PS}}<\infty$.
\end{corollary}

To deduce that $\op{sk_\G}(w_0)>0$, when $w_0\G$ is
infinite in Theorem \ref{m11} we need the following. Here $\G$ need not be geometrically finite. 

\begin{proposition}\label{nempty}
 If $[\G:\G_{\tS}]=\infty$, then
$\Lambda(\G)\not\subset  \partial _\infty (\tS)$, and  $\abs{\mu_E^{\PS}}>0$.
\end{proposition}

\begin{proof}
Suppose on the contrary that $\Lambda(\G)\subset \partial _\infty( \tS)$.
Let $L$ be a geodesic joining two distinct points say $\xi_1,\xi_2\in \Lambda(\G)$. Then $L
\subset \tS$. For any $\g\in\G$, we have $\g L$ is the geodesic joining $\g\xi_1$ and $\g
\xi_2$, and hence $\g L\subset \tS$. Now fix $x_0\in L$. Then $\G x_0\subset \tS$. Since $
\G_{\tS}\bs\tS\to \G\bs\bH^n$ is a proper map, we get that $\G_{\tS}\bs \G$ is finite, a
contradiction to our assumption.
\end{proof}


\section{Orbital counting for discrete hyperbolic groups} \label{cs}
As before, let $G=\SO(n,1)^\circ$ for $n\ge 2$ and $\G$ a torsion-free, non-elementary, discrete subgroup of $G$.

\subsection{Computation with $\tilde m^{\BR}$}\label{cmbr}
Let $K$ be a maximal compact subgroup of $G$. Let $o\in\bH^{n}$ be such that $K=G_{o}$.
Then $G/K\cong \bH^{n}$. Let $X_{0}\in \T^{1}_{o}(\bH^{n})$ and $M=G_{X_{0}}$. Then $G/M\cong \T^{1}(\bH^{n})$, where 
$g[M]=gX_{0}$. Let $A=\{a_{r}:r\in\R\}\subset Z_{G}(M)$ be a
one-parameter subgroup of $G$ consisting of diagonalizable elements such that $
\gr(X_{0})=a_{r}[M]$. Via the map $k\mapsto kX_{0}^{+}$, we have $K/M\cong \partial{\bH^n}$.

Let $N<G$ be the expanding horospherical subgroup with respect to the right $a_r$-action;
that is,
\begin{equation} \label{eq:Nhoro}
N:=\{g\in G: a_r g a_r^{-1} \to e\quad\text{as $r\to \infty$}\}.
\end{equation}
The $N$-leaves $gNM/M$ correspond to unstable horospheres $\cH^{+}_{gX_{0}}$
in $\T^1(G/K)=G/M$ based at $gX_{0}^{-}$.  The map $N\ni z\mapsto zX_0^+\in \partial{\bH^n}\setminus\{X_0^-\}$ is a 
diffeomorphism.

As before let $m_{o}$ denote the $G$-invariant (Lebesgue) conformal density $\{m_x\}_{x\in\bH^n}$ on $\partial\bH^n$. 
We normalize it so that ${m_{o}}$ (and hence every $m_{x}$) is a probability measure. 
Here $m_{o}$ is $K$-invariant.

\begin{lemma} \label{lemma:N-Leb}
 For any $g\in G$, consider the measure $\lambda_g$ on $N$ given by
\[
d\lambda_g(z)=e^{(n-1)\beta_{gzX_0^+} (o, gz(o))} dm_o(gzX_0^+), \text{ where $z\in N$};
\]
Then $\lambda_{g}=\lambda_{e}$. In particular $\lambda_{e}$ is a Haar measure on $N$ which we shall denote by the integral 
$dn=d\lambda_{e}(n)$ on $N$
\end{lemma}

\begin{proof}
Since $\{m_{x}\}$ is a $G$-invariant conformal density,
\[
dm_o(gzX_0^+)=dm_{g^{-1}(o)} (zX_0^+)=
e^{(n-1)\beta_{zX_0^+}(o, g^{-1}(o))} dm_o(zX_0^+).
\]
Since $\beta_{gzX_0^+} (o, gz(o))= \beta_{zX_0^+} (g^{-1}(o), z(o))$,
\begin{equation} \label{eq:dlambda}
d\lambda_g(z)=e^{(n-1)\beta_{zX_0^+}(o, z(o))} dm_o(zX_0^+)=d\lambda_e(z).
\end{equation}
For any $g\in N$, $d\lambda_{e}(gz)=d\lambda_{g}(z)=d\lambda_{e}(z)$. Therefore $\lambda_{e}$ is $N$-invariant.
\end{proof}

\begin{notation} \label{note:kX0-}
Note that $G_{X_{0}^{-}}=ANM$ and $K\cap G_{X_{0}^{-}}=M$. For $\psi\in C(K)$ and a measure $\lambda$ on 
$\partial\bH^{n}=KX_{0}^{-}\cong K/M$, we define
\begin{equation} \label{eq:kX0-}
\int_{k\in K}\psi(k)\,d\lambda(kX_{0}^{-}):=\int_{K/M} \bigl(\int_{m\in M}\psi(km)\,dm\bigr)
d\lambda(kM).
\end{equation}
\end{notation}

We also fix a Patterson-Sullivan density $\{\nu_{x}\}$ on $\partial\bH^{n}$ and consider $\tilde m^{\BR}$ defined as
in subsection \ref{defbms} with respect to $\{m_x\}$ and $\{\nu_x\}$.

\begin{proposition} \label{prop:BR-Iwasawa}\label{mbr}
For any $\phi\in C_c( \T^{1}(\bH^{n}))=C_c(G)^M$, 
\begin{equation*}
 \tilde m^{\BR}(\phi)=\int_{k\in K}\int_{r\in\R} \int_{n\in N} \phi(ka_rn) e^{-\delta r} \, dn\, dr\,
 d\nu_o(kX_0^-).
 \end{equation*}
\end{proposition}

\begin{proof}
 By definition,
\[
\tilde m^{\BR}(\phi) =\int \phi(u) e^{(n-1)\beta_{u^+} (o, \pi(u))}e^{\delta\beta_{u^-} (o, \pi(u))} dm_o(u^+)d\nu_o(u^-)dt,
\]
where $t= \beta_{u^-} (o, \pi(u))$. Let $u=ka_rnX_0$.
Then, since $G_{X_0^-}=MAN$, we have $u^-=ka_rnX_0^-=kX_0^-$ and
\begin{equation*}
\begin{array}{ll}
 t=&\beta_{u^-}(o, \pi(u))=\beta_{kX_0^-}(o, ka_r no)=\beta_{X_0^-}(o, a_rn o)
 \\ & =\lim_{t\to \infty}d(o, a_{-t}o)-d(a_rno, a_{-t}o)
 \\ & =\lim_{t\to \infty} t-d(a_{t+r}n a_{-t-r} (a_{t+r}o), o)
 \\ &=\lim_{t\to\infty}t -d(a_{t+r}o,o)=-r .
\end{array}
 \end{equation*}
 Therefore
$e^{\delta \beta_{u^-} (o, \pi(u))} d\nu_o(u^-) = e^{-\delta r} d\nu_o(kX_0^-).$
And by Lemma~\ref{lemma:N-Leb} for  fixed $g=ka_r$ and variable $z=n\in N$,
\[
e^{(n-1)\beta_{ka_rnX_0^+} (o, ka_rn\pi(o))} dm_o(ka_rnX_0^+)=d\lambda_{ka_r}(n)=d\lambda_e(n)=dn.
\]
Putting together, this proves the claim. \end{proof}

\newcommand{\cR}{\mathcal R}

\begin{notation} \label{not:approxid}
(1) Let $dk$ denote the probability Haar measure on $K$. Since $m_{o}$ is a $K$-invariant probability measure on $\partial
\bH^{n}=K/M$, 
we have that $dk=dm_{o}(kX_{0}^{-})$ (and similarly  $dk=dm_{o}(kX^{+})$).
 We fix the Haar measure $dg$ on $G$ given as follows: for $g=ka_{r}n\in KAN$, 
 $dg=e^{-(n-1)r}\,dn\,dr\,dk$.
Since $G$ is unimodular, $dg=dg\inv$. Therefore if we express $g=na_{r}k$, then
 $dg=e^{(n-1)r}dn\,dr\,dk$.
And if we express $g=a_{r}nk$, then $dg=dr\,dn\,dk$.

\medskip
(2) For $\e>0$, let $U_\e$ denote the $\e$-neighborhood of $e$ in $G$. By an {\em approximate identity}
on $G$, we mean a family of nonnegative continuous functions
$\{\psi_\e\}_{\e>0}$ on $G$ with $\supp(\psi_\e)\subset U_\e$ and $\int_G\psi_\e(g) dg =1$.

\medskip
(3) For $\xi\in C(M\ba K)$ and $\psi\in C_c(G)$ and
a measurable $\Omega\subset K$ with $M\Omega=\Omega$, we define 
 a function $\xi\ast_\Omega \psi\in C_c(G/M)$ by
 \begin{equation} \label{eq:astOmega}
 \xi\ast_\Omega \psi(g):=\int_{k\in \Omega}\xi(k)\psi(gk)\, dk.
\end{equation}
For $\psi\in C_c(\G\ba G)$,  we define $\xi\ast_\Omega \psi\in C_c(\G\ba G/M)$ similarly.
\end{notation}

\begin{proposition}\label{corpt}\label{eq:BRformula}
Let $\{\psi_\e\}_{\e>0}$ be an approximate identity on $G$. Let $f\in C(M\ba K)$ and $\Omega\subset K$ be such that
$M\Omega=\Omega$ and $\nu_o(\partial(\Omega^{-1})X_{0}^{-})=0$. Then
\begin{equation*} 
\lim_{\e\to 0}
 \tilde m^{\BR}(f*_\Omega \psi_\e)=
 \int_{k\in \Omega^{-1}} f(k^{-1})\, d\nu_o(kX_0^-).
 \end{equation*}
\end{proposition}
\begin{proof}
Note that for some uniform constants $\ell_1, \ell_2>0$, we have for all $k\in K$ and for all
small $\e>0$,
\begin{equation} \label{eq:small}
k\inv U_\e \subset U_{\ell_1 \e}k\inv \subset
(A\cap U_{\ell_2 \e})(N\cap U_{\ell_2 \e})k\inv (K\cap U_{\ell_2 \e}).
\end{equation}

Set $K_\e:= (K\cap U_{\ell_2 \e})$,
 $\Omega_{\e +}=\Omega K_\e$ and $\Omega_{\e-}=\cap_{k\in K_\e} \Omega k$.

In view of the decomposition $G=ANK$, for a function $\phi$ on $K$,
we define a function $\cR_{\phi}$ on $G$ by $\cR_{\phi}(g)=\phi(k)$ for $g=ank\in ANK$.
 For any $\eta>0$, there exists $\e>0$ such that for all $k\in K $ and $g\in U_\e$,
\begin{equation} \label{eq:unifcont}
{\cR}_{f\cdot \chi_{\Omega_{\e{-}}}}(k^{-1})- \eta\leq {\cR}_{f\cdot \chi_{\Omega}}(k^{-1} g) \leq {\cR}_{f\cdot \chi_{\Omega_{\e
+}}}(k^{-1})+ \eta.
\end{equation}

Now by Proposition \ref{mbr},
\begin{equation*}
\begin{array}{ll}
&\ \tilde m^{\BR} (f*_\Omega\psi_\e)
\\
&= \int_{g\in G }\int_{k'\in  \Omega} \psi_\e(gk') f(k')\, dk'
d\tilde m^{\BR}(g)
\\
&= \int_{(k,a_r,n)\in K\times A\times N}
\int_{k'\in  \Omega} \psi_\e(ka_r n k') f(k') e^{-\delta r} \,dk'
dn dr d\nu_o(kX_0^{-})
\\
&= \int_{k\in K} \int_{(a_r,n,k')\in A\times N\times K}\psi_\e(k a_r n k') f(k')\chi_{\Omega}(k') e^{-\delta r} \, drdndk' d
\nu_o(kX_0^{-})
\\
&= \int_{k\in K} \int_{g\in G}\psi_\e(k g)  {\cR}_{f\cdot \chi_\Omega}(g)e^{-\delta r_{g}}\, dg d\nu_o(kX_0^{-}), \quad \text{if 
$g=a_{r_{g}}nk'$}
\\
&\leq e^{\delta \ell_{2} \e}\int_{k\in K} \int_{g\in G}\psi_\e(g) {\cR}_{f\cdot \chi_\Omega}(k^{-1} g)\, dg
d\nu_o(kX_0^{-}), \quad \text{by \eqref{eq:small}}
\\
&\leq  e^{\delta\ell_{2}\e}\int_{k\in K} \int_{g\in G}\psi_\e(g) ({\cR}_{ f\cdot \chi_{\Omega_{\e+ }}}
(k^{-1} )+\eta) \, dg d\nu_o(kX_0^{-})
\\
&= e^{\delta\ell_{2}\e}\bigl(\int_{k\in K}  {\cR}_{ f\cdot \chi_{\Omega_{\e+ }}}(k^{-1} )  \,  d\nu_o(k
X_0^{-}) + \eta\abs{\nu_{o}}\bigr), \text{ as $\int_G\psi_\e(g) \, dg=1$.}
\\
&= e^{\delta\ell_{2}\e}\bigl(\int_{k\in \Omega_{\e+}^{-1}} f(k^{-1}) \, d\nu_o(kX_0^{-}) + \eta\abs{\nu_{o}}\bigr).
\end{array}
\end{equation*}
Since $\cap_{\e>0} \Omega_{\e+}=\cl{\Omega}$, and since $\eta>0$ was arbitrarily chosen,
\[
\limsup_{\e\to 0}\tilde m^{\BR} (f*_\Omega\psi_\e)\leq \int_{k\in \cl{\Omega\inv}} f(k^{-1}) \, d\nu_o(kX_0^{-}).
\]
Similarly,
$\liminf_{\e\to 0}\tilde m^{\BR} (f*_\Omega\psi_\e)\geq \int_{k\in \Int{\Omega\inv}} f(k^{-1}) \, d\nu_o(kX_0^{-})$. Since $\nu_{o}
(\partial(\Omega\inv)X_{0}^{-})=0$, we obtain \eqref{eq:BRformula}.
\end{proof}

\subsection{Setup for counting results}\label{setup}
 Till the end of this section, let $V$ be a finite
dimensional vector space on which $G$ acts linearly from the
right and let $w_0\in V$. We set
$H:=G_{w_{0}}$.

\subsubsection{When $H$ is a symmetric subgroup of $G$}  \label{subsec:symmetric} Let $H<G=\SO(n,1)^
\circ$ be a  symmetric subgroup, i.e., there is a
non-trivial involution $\sigma$ of $G$ such that $H^\circ=(G^{\sigma})^\circ$
 where $G^{\sigma}=\{g\in G: \sigma(g)=g\}$. There exists a Cartan involution $\theta$ of $G$
such that $\theta\circ\sigma=\sigma\circ\theta$. Let $K=G^{\theta}$. It turns out that $H^\circ$
is a subgroup of finite index in its normalizer $N_{G}(H^\circ)$,
 and up to a conjugation of $G$, $H^\circ=(\SO(k,1)\times \SO(n-k))^\circ$ for some $0\leq k
\leq n-1$ and $K=\SO(n)$. Choose $o\in\bH^{n}$ such that $G_{o}=K$. Then $\tS=H\cdot o$
is an isometric imbedding of $\bH^{k}$ in $\bH^{n}$. Let $\tE$ be the unit normal bundle over
$\tS$.

 \subsubsection{When $G_{\br w_0}$ is a parabolic subgroup of $G$} \label{subsec:parabolic}
Suppose that $G_{\R w_0}$ is a parabolic subgroup of $G$. Let $\theta$ be any Cartan involution
of $G$ and let $K=G^{\theta}$. Then $G=G_{\R w_0}K$. Let $N$ be the unipotent radical of
$G_{\br w_0}$. Let $o\in\bH^n$ be such that $G_o=K$. Then $\tS:=N\cdot o$ is a horosphere. Let $\tE
\subset\T^{1}(\bH^{n})$ be the unstable horosphere such that $\pi(\tE)=\tS$. Let $H=(G_{\br w_0}\cap
K)N$.

\subsubsection{Common structure in both cases}
Let the notation be as in any of the above section \ref{subsec:symmetric} or
\ref{subsec:parabolic}. Let $X_0\in \T^1_{o}(\bH^{n})\cap \tE$. Let $\tE^\ast=H\cdot X_0$.
If $H$ is symmetric and $\codim(\tS)>1$, or in the parabolic case, then $\tE$ is
connected and $\tE^\ast=\tE$. If $H$ is symmetric and $\codim(\tS)=1$, then $\tE$ has two
connected components: $\tE^+$ containing $X_0$ and $\tE^-$ containing $-X_0$; and then
either $\tE^\ast=\tE$ or $\tE^\ast=\tE^+$. There exists a one-parameter subgroup $A=\{a_r\}
\subset G$ consisting of $\R$-diagonalizable elements, such that $\gr(X_0)=a_rX_0$ for all $r
\in\R$.
Let $M=G_{X_{0}}$, which coincides with $Z_{K}(A)$, i.e., the centralizer of $A$ in $K$, and $A^{\pm}=\{a_{\pm r}: r
\ge 0\}$. Let $N$ be the expanding horospherical subgroup with respect to $\{a_r\}$.

When $G_{\R w_{0}}$ is parabolic, then  $G_{w_{0}}=MN=H$ where $M=G_{\br w_0} \cap K$;
hence $N$ is the unipotent radical of $G_{\br w_0}$ so there is no conflict of notation.
 In the case when $H$ is symmetric, then $\tE^\ast=\tE$ if and only if $G=HA^{+}K$.  In all
cases, we have $G=HAK$. Put $E=\p(\tE)$, $E^\ast=\p(\tE^\ast)$, and in the special cases when $\tE$ is not connected, we set 
$E^
\pm=\p(\tE^\pm)$.

\subsubsection{$HAK$ decomposition of Haar measure on $G$}\label{hakm}

Note that $\tE^{\ast}=HX_0\cong H/(M\cap H)$ and recall that
$$d\mu_{\tilde E}^{\Leb}(v)=e^{(n-1)\beta_{v^+}(o,\pi(v))}dm_o(v^+).$$
There is a Haar measure $dh$ on $H$ such that for any $\psi\in C_c(H)$ if we put $
\bar\psi(h)=\int_{m\in M\cap H} \psi(hm)\,dm$, where $dm$ denotes the probability Haar
integral on $M\cap H$, then $\bar\psi\in C_{c}(H)^{M\cap H}=C_{c}(\tE)$, and
\begin{equation} \label{eq:dh}
\int_{H}\psi\,dh=\int_{\tE} \bar\psi\,d\mu^{\Leb}_{\tE}.
\end{equation}

In view of the decompositions $G=HA^{+}K$ or $G=HAK$,  there exists a function $\rho:\R
\to (0,\infty)$, such that we get the following  Haar measure $dg$ on $G$:
For any $\psi\in C_{c}(G)$, by \cite[Theorem~8.1.1]{sch}
\begin{align}  \label{eq:haar:g=hak}
\int_{G}\psi\, dg&=\int_{k\in K}\int_{r\in R}\int_{h\in H} \rho(r)\psi(ha_{r}k)\,dh dr dk, \text{ and}
\\
\rho(r)&\sim
\begin{cases}
e^{(n-1)\abs{r}} & \text{if $r\to\pm\infty$ and $H$ is symmetric}, \\
e^{(n-1)r} & \text{if $r\to\pm\infty$ and $G_{\R w_{0}}$ is parabolic}.
\end{cases}
\label{eq:Xi}
\end{align}
where  $R=\{r\geq 0\}$ if $G=HA^{+}K$, otherwise $R=\R$. In fact the Haar measure $dg$ described in 
Notation~\ref{not:approxid}(1) and the Haar measure $dg$ defined in \eqref{eq:haar:g=hak} are identical, see \S\ref{a2}.

\subsection{Extension of Theorem \ref{mainergint} to $\G\ba G$ for Zariski dense $\G$}\label{A1}
The result in this subsection will enable us to state our counting theorems
for general norms, provided $\Gamma$ is Zariski dense. 

Let $\mBR$ be the measure on $\G\bs G$ which is the $M$-invariant extension of $m^{\BR}$, that is,
  for $\psi\in C_c(\GmG)$,
\[  \mBR(\psi): =m^{\BR}(\bar\psi)
\]
where $\bar\psi(\p(gX_0))=\int_{m\in M} \psi(\G gm)\,dm$ and $dm$ denotes the Haar probability measure on $M$.

As $M$ normalizes $N$, $\mBR$ is
 invariant for the right-translation action of $N$ on $\GmG$.

\begin{theorem}[Flaminio-Spatzier~{\cite[Cor.~1.6]{FlaminioSpatzier1990}}] \label{thm:mBRergodic}
If $\G$ is Zariski dense and $\abs{m^{\BMS}}<\infty$, then $\mBR$ is
$N$-ergodic.
\end{theorem}

Let $H$ and $\tE$ be as in subsection \ref{subsec:symmetric} or \ref{subsec:parabolic} so that
 $H=G_{\tE}$. Let $dh$ be the Haar measure on $H$ defined as in \eqref{eq:dh}; by abuse of notation,
we also denote by $dh$ the measure on $\G_H\backslash H$ induced by $dh$.

We recall that for $\G$ Zariski dense, $\abs{\muPS_E}<\infty$ implies that
the canonical map $\G_H\ba H\to \G\backslash G$ is proper by Theorem \ref{thm:finite-closed}.
\begin{theorem} \label{thm:mainerg-group}\label{Zdapp}
Let $\G$ be a Zariski dense discrete subgroup of $G$ such that $\abs{m^{\BMS}}<\infty$ and
 $\abs{\muPS_E}<\infty$. Then for any $\psi\in C_c(\G\bs G)$,
\[
\lim_{r\to\infty} e^{(n-1-\delta)r} \int_{h\in \G_H\backslash H} \psi(\G h a_r)\,d h=\frac{\abs{\muPS_E}}{\abs{m^{\BMS}}}
\mBR(\psi).
\]
\end{theorem}

\begin{proof}
Define a measure $\lambda_r$ on $\GmG$ as follows: for any $\psi\in C_c(\GmG)$,
\[
\lambda_r(\psi)=e^{(n-1-\delta)r} \int_{h\in \GmH} \psi(\G h a_r)\,d h.
\]
Let $\quo:\GmG\to \T^1(\G\bs \bH^n)\cong \GmG/M$ be the natural quotient map. Then for any
$\bar\psi\in C_c(\T^1(\G\bs\bH^n))$, we have $\bar\psi(\quo(xma_r))=\bar\psi(\quo(xa_r))$ for any $m\in M$ and $x\in \GmG$,
as $M$ and $A$ commute with each other, and hence \[
\quo_\ast(\lambda_r)(\psi)=\lambda_r(\bar\psi\circ\quo)
=e^{(n-1-\delta)r} \int_{E}\bar\psi(va_r)\,d\mu^{\Leb}_E(v).
\]
Therefore by Theorem~\ref{mainergint}, $\quo_\ast(\lambda_r)\to C\cdot m^{\BR}$, where $C=\frac{\abs{\muPS_E}}
{\abs{m^{\BMS}}}$.

In order to show that $\lambda_r$ weakly converges to $C\mBR$, it suffices to show that every sequence $\lambda_{r_k}$ has
a subsequence converging to $C\mBR$.

 For any sequence $r_k\to\infty$, since $\quo$ is a proper map, after passing to a subsequence of $\{r_{k}\}$ there exists a 
measure $\lambda$ on $\GmG$ such that
$\lambda_{r_k}(\phi)\to\lambda(\phi)$ for every $\phi\in C_c(\GmG)$. Therefore
\[
\quo_\ast(\lambda)=C m^{\BR}.
\]

For any $g\in G$, define a measure $g\lambda$ on $\GmG$ by $g\lambda(A)=\lambda(Ag)$ for any measurable $A\subset
\GmG$.  Now for any $\psi\in C_c(\GmG)$,
\begin{equation} \label{eq:LmBR}
\begin{array}{ll}
\int_{m\in M} (m\lambda)(\psi)\,dm & = \int_{m\in M} \int_{\GmG} \psi(xm)\,d\lambda(x)\,dm\\
& =\bar\quo_\ast(\lambda)(\bar\psi)=Cm^{\BR}(\bar\psi)=C\mBR(\psi).
\end{array}
\end{equation}

\subsubsection*{Claim 1} $\lambda$ is $N$-invariant.

\begin{proof}[Proof of Claim~1] Due to Lemma~\ref{lemma:diffeo}, the map $h\mapsto hX_0^+$ is a submersion and hence 
there exists a neighborhood $\Omega$ of $e$ in $N$ and a continuous injective map $\sigma:\Omega\to H$ such that $
\sigma(e)=e$ and
$\sigma(z)X_0^+=zX_0^+$ for all $z\in\Omega$.

Fix $z\in \Omega$, let $z_k:=a_{r_k}za_{-r_k}$, and $h_k=\sigma(z_k)$ for all large $k$. Then $b_k=z_k\inv h_k\in G_{X_0^+}
=MAN^-$. Therefore $b_k\to e$ and $a_{-r_k}b_ka_{r_k}\to e$ as $k\to\infty$.

Let $\psi\in C_c(\GmG)$. Given $\e>0$ and $x\in \G\ba G$, set
\[
\psi_{\e+}(x)=\sup_{g\in U_\e} \psi(xg)  \quad \text{ and }\quad  \psi_{\e-}=\inf_{g\in U_\e} \psi(xg).
\]
Since $\psi$ is uniformly continuous and
$a_{r_k}z = h_kb_k^{-1} a_{r_k} = h_k a_{r_k} (a_{-r_k} b_k^{-1} a_{r_k})$,
 we have for all large $k$ and for all $x\in \GmG$,
\[
\psi_{\e-}(xhh_ka_{r_k}) \le \psi(xa_{r_k}z) \le \psi_{\e+}(xh_ka_{r_k}).
\]
Since the measure $d h$ is $H$-invariant,
\[
\int_{h\in\GmH} \psi(\G h a_{r_k}z)\,d h \leq \int_{\GmH} \psi_{\e+}(\G hh_k a_{r_k})\,d h = \int_{\GmH}
\psi_{\e+}(\G h a_{r_k})\,d  h.
\]
Similarly we get a lower bound in terms of $\psi_{\e-}$. Since $\lambda_{r_k}\to \lambda$ as $k\to\infty$,
\[
\lambda(\psi_{\e-})\leq \int_{\GmG} \psi(xz)\,d\lambda(x) \leq  \lambda(\psi_{\e+}).
\]
Since $\psi\in C_c(\GmG)$, we have that $\lambda(\psi_{\e\pm})\to \lambda(\psi)$ as $\e\to 0$. Therefore the $z$-action 
preserves $\lambda$.
\end{proof}

\subsubsection*{Claim 2} $\lambda=C\mBR$.

\begin{proof}[Proof of Claim~2]
By \eqref{eq:LmBR}, it is enough to show that $\lambda$ is $M$-invariant. 
For any $\e>0$, define a measure $\eta_\e$ on $\G\ba G$ by
 $$\eta_\e:=\frac{1}{|M_{\e}|}\int_{m\in M_\e }m\lambda\,dm,$$
 where $\abs{M_\e}=\int_{M_\e} dm $. Then  since $M$ normalizes $N$, $\eta_{\e}$ 
is $N$-invariant. By \eqref{eq:LmBR}
 $$\eta_{\e}\ll \mBR.$$

almost all $m\in M$, and hence

   Therefore, since $\mBR$ is $N$-ergodic by Theorem~\ref{thm:mBRergodic},
    there exists $c_{\e}> 0$ such that $\eta_{\e}=c_{\e} \mBR$. Thus $\eta_\e$ is $M$-invariant, as
    $\mBR$ is $M$-invariant.

If $\lambda$ is not $M$-invariant, there exist $\psi\in C_c(\G\ba G)$, $m_0\in M$ and $\beta>0$ such that    $\lambda(m_0.\psi)
\ge \lambda(\psi)+\beta$. There exists $\e>0$ such that
 for all $m\in M_{\e}$,
     $\lambda((m m_0) \psi)\ge \lambda(m. \psi)+\beta/2$. This implies that $\eta_\e (m_0\psi) \ge \eta_\e(\psi) +\beta/2$, which is 
a contradiction to the $M$-invariance of $\eta_\e$.
\end{proof}
As noted before this completes the proof of Theorem~\ref{Zdapp}.
\end{proof}

\subsection{Statements of Counting theorems}\label{oc}

 Now we describe the main counting results of this section.
 In the next two theorems \ref{thm:count-sector}, and
\ref{thm:count-cone}, we suppose
that the following conditions hold for $w_0\in V$ and $\G$ a non-elementary discrete torsion-free
subgroup of $G$:
\begin{enumerate}
\item $w_0\G$ is discrete.
\item $H$ is a symmetric subgroup of $G$, or  $G_{\br w_0}$ is a parabolic subgroup of
$G$.
\item  $\abs{m^{\BMS}}<\infty$ and $|\muPS_{E}|<\infty$.
\end{enumerate}
Let $\lambda\in\N$ be the $\log$ of the largest eigenvalue of $a_1$ on $\R$-$\Span(w_0G)$
and set
\begin{equation} \label{eq:wlambda}
w_{0}^{\lambda}:=\lim_{r\to\infty} \frac{w_{0}a_{r}}{e^{\lambda r}} \quad\text{and}\quad w_{0}^{-
\lambda}:=\lim_{r\to\infty}\frac{w_{0}a_{-r}}{e^{\lambda r}}.
\end{equation}

\begin{theorem}[Counting in sectors] \label{thm:count-sector}\label{thm:counting1}
Let $\norm{\cdot}$ be a norm on $V$ satisfying
\begin{equation} \label{eq:norm-cond}
\norm{w_0^{\pm\lambda}mk}=\norm{w_0^{\pm\lambda} k},
\text{ for all $m\in M$ and $k\in K$},
\end{equation}
 and set $B_T:=\{v\in V: \|v\| < T\}$.

\begin{enumerate}
\item For any Borel measurable  $\Omega\subset K$ such that $M\Omega=\Omega$ and $
\nu_o(\partial(\Omega \inv X_{0}^{-}))=0$,
\begin{multline*}
\lim_{T\to\infty} \frac{\#(w_0\G\cap B_T\cap (w_0 A^+\Omega))}{T^{\delta/\lambda}} \\ =
\frac{\muPS_{E}(E^{\ast})}{\delta\cdot \abs{m^{\BMS}}}
\int_{k\in\Omega \inv} \norm{w_0^\lambda k\inv}^{-\delta/\lambda}\, d\nu_o(kX_0^-).
\end{multline*}

\item For the full count in a ball,  we get 
\begin{align} \label{eq:main:counting}
&\lim_{T\to\infty} \frac{\# (w_0\G\cap B_T) }{T^{\delta/\lambda}}
\\
&=\begin{cases}
\frac{\muPS_{E}(E)}{\delta \cdot \abs{m^{\BMS}}}
\int_{k\in K}\norm{w_0^{\lambda}k\inv}^{-\delta/\lambda}\,d\nu_{o}(kX_{0}^{-})>0,
\text{ if $\tE=G_{w_{0}}\cdot X_{0}$} \\[5pt]
\sum_{\pm} \frac{\muPS_{E}(E^{\pm})}{\delta \cdot \abs{m^{\BMS}}}
\int_{k\in K}\norm{w_0^{\pm\lambda}k\inv}^{-\delta/\lambda}\,d\nu_{o}(kX_{0}^{-})>0,
\text{ otherwise}.
\end{cases}
\notag
\end{align}
\end{enumerate}
\end{theorem}

\begin{remark} \label{rem:main}
\begin{enumerate}
\item By \cite[Lemma~4.2]{GorodnikOhShah2009}, we have $w_{0}^\lambda\neq 0$. And if $H$
is symmetric, then $w_{0}^{-\lambda}\neq 0$.

\item Since $w_0\G$ is discrete,  $H\G$ is closed in $G$, and hence $\G H$ is closed in $G$. It follows that  the canonical 
imbedding $(\G\cap H)\bs H\to \G\bs G$ is a proper injective map; the properness follows from a suitable open mapping theorem 
in the category of locally compact Hausdorff second countable topological group actions. Therefore the map $(\G\cap G_{\tS})
\bs \tS\to \G\bs \bH^n$ is a proper map. In particular, $E$ and $E^{\pm}$ are closed subsets of $\T^{1}(\G\bs\bH^{n})$.

\item \label{itm:71}The condition \eqref{eq:norm-cond} holds if $\norm{\cdot}$ is $K$-invariant as in
Theorem~\ref{m11}. There exists a Weyl group element $k_0\in K$ such that $k_{0}\inv a_{r}
k_{0}=a_{-r}$ for all $r\in\R$. Then $w_0^{-\lambda}=w_0^{\lambda}k_{0}$. Therefore if
$\norm{\cdot}$ is $K$-invariant, then $\norm{w_0^{\pm\lambda}k}=\norm{w_0^\lambda}$ for
all $k\in K$. Then the limit \eqref{eq:main:counting} becomes \eqref{eq:m11}. Thus Theorem~
\ref{thm:counting1} implies Theorem~\ref{m11}.

\item \label{itm:72} When $\G$ is Zariski dense in $G$, Theorem \ref{thm:counting1}
 holds for any norm on $V$ without the condition \eqref{eq:norm-cond} and for the $\Omega$ 
without the $M$-invariance condition. 
See \S\ref{zde} for details.

\item \label{itm:73}Since $w_{0}^{\pm \lambda}$ is fixed by $H\cap Z_{K}(A)$, if $M=Z_{K}(A)\subset H$,
then the condition~\eqref{eq:norm-cond} holds for any norm on $V$.
We have $M\subset H$ in the parabolic case. In the case when $H$ is symmetric, if
$\tS$ is a single point or $\tS$ is of codimension one, then $M\subset H$.
\end{enumerate} \end{remark}


\begin{theorem}[Counting in cones] \label{thm:count-cone}
Suppose further that $\G$ is Zariski dense in $G$. 
Let $\Theta$ be a measurable subset of $V$. Let 
\[
\Omega_{\pm}= \{k\in K:w_{0}^{\pm\lambda}k\in \R^{+}\Theta\}.
\] 
If $\nu_{o}(\partial(\Omega_{\pm}\inv X_{0}^{-}))=0$, then for any norm $\|\cdot \|$ on $V$,
\begin{multline}
\lim_{T\to\infty}\frac{\#(w_{0}\G\cap B_{T}\cap \R^{+}\Theta)}{T^{\delta/\lambda}}
=\frac{1}{\delta\cdot \abs{m^{\BMS}}} \times \\
\begin{cases}
\muPS_{E}(E)\int_{k\in\Omega_{+} \inv}
\norm{w_0^\lambda k\inv}^{-\delta/\lambda}\, d\nu_o(kX_0^-),
\text{ if $\tE=HX_{0}$} \\
\sum
\muPS_{E}(E^\pm) \int_{k\in \Omega_{\pm}\inv}
\norm{w_{0}^{\pm \lambda}k\inv}^{-\delta/\lambda}\,d\bar\nu_{o}(kX_{0}^{-}), \text{ otherwise};
\label{eq:cone}
\end{cases}
\end{multline}
\end{theorem}

Note that if $\G$ is Zariski dense in $G$, and if $\partial(\Omega_\pm)$ is contained in a countable union of proper real 
algebraic subvarieties of $\partial\bH^n$ then $\nu_o(\partial(\Omega_\pm))=0$ (see 
\cite[Corollary~1.4]{FlaminioSpatzier1990} and \cite[Remark~1.7(2)]{OhShahcircle}). 

\subsection{Proof of the counting statements}
We follow the counting technique of \cite{DukeRudnickSarnak1993} and
\cite{EskinMcMullen1993}. For a Borel subset $\Omega\subset K$
satisfying the condition of Theorem \ref{thm:count-sector},
we set $$B_{T}(\Omega)=B_{T}\cap w_{0}A^{+}\Omega,$$
and define the following counting function on $\G\ba G$:
$$F_{B_T(\Omega)}(g):=\sum_{\gamma\in \G_{w_0}\ba\G}\chi_{B_T(\Omega)}(w_0\gamma
g).$$

We note that
\begin{equation} \label{eq:FTe}
F_{B_T(\Omega)}(e)=\# (w_0\G\cap B_T(\Omega))
= \#(w_0\G\cap B_T\cap (w_0 A^+\Omega)).
\end{equation}

For $\psi_1, \psi_2\in C_c(\G\ba G)$,
we set $\la \psi_1, \psi_2\ra:=\int_{\G\ba G} \psi_1(g) \psi_2(g)\, dg$.

Let $\psi\in C_c(\G\ba G)$. Then by \eqref{eq:haar:g=hak},
\begin{align}
&\la F_{B_T(\Omega)}, \psi\ra
= \int_{\G_{w_0}\ba G} \chi_{B_T(\Omega)} (w_0 g) \psi(g) \, dg
\notag \\
& =\int_{k\in \Omega} \int_{\{r\geq 0:\|w_0 a_r k\| <T\}}
\Bigl(\int_{[h]\in \G_{w_0}\ba H} \psi(h a_r k)\, dh\Bigr) \rho(r)\, dr dk
\label{eq:614}
\end{align}

For any $k\in K$ and $T>0$, define
\begin{equation} \label{eq:rkT}
r(k,T)=\sup\{r>0: \norm{w_{0}a_rk}<T\}.
\end{equation}

Let $\lambda_1$ be the $\log$ of the largest eigenvalue of $a_1$ on $V$ strictly less than
$e^{\lambda}$. Then by \eqref{eq:wlambda}
there exist $C_1\geq 1$ and $r_1\geq 0$ such that
\begin{equation} \label{eq:r}
\norm{w_0a_rk -e^{\lambda r} w_0^\lambda k}
\leq C_1e^{\lambda_1 r}, \text{ for all $k\in K$ and $r\ge r_1$.}
\end{equation}

Put $\epsilon_0=(\lambda-\lambda_1)/\lambda>0$ and $C_2=2C_1/\inf_{k\in K}\norm{w_0^\lambda k}$. Let $T_1\geq 1$ be 
such that $C_2T_1^{-\e_0}\leq 1/2$ and $(1/2)(T_1/\sup_{k\in K}\norm{w_0^\lambda k})^{1/\lambda} \geq e^{r_1}$. For $T\geq 
T_1$, we define
functions $r_{\pm}(k,T)$ via
\begin{equation} \label{eq:rpm}
e^{r_{\pm}(k,T)}=(T/\norm{w_0^\lambda k})^{1/\lambda}(1\pm C_2T^{-\e_0}).
\end{equation}
Then by elementary calculation using \eqref{eq:r}
\begin{equation} \label{eq:644}
r_-(k,T) \leq r(k,T) \leq  r_+(k,T), \text{ for all $T\geq T_1$ and $k\in K$}.
\end{equation}
By \eqref{eq:norm-cond},
\begin{equation} \label{eq:rmk=rk}
r_{\pm}(mk,T)=r_{\pm}(k,T), \text{ for all $m\in M$ and $k\in K$}.
\end{equation}

We note that by \eqref{eq:rpm}, given $\e>0$, for $T_1(\e)$ sufficiently large,
\begin{equation} \label{eq:rpm2}
 e^{\delta r_\pm(k,T)}=(1+O(\e))(T/\norm{w_0^\lambda k})^{\delta/\lambda} \text{ for all $T\geq T_1(\e)$}.
\end{equation}

\begin{proposition} \label{prop:H:E}
For any non-negative $\psi\in C_{c}(\G\bs G)$,
\begin{equation*} 
\begin{array}{l}
\int_{k\in \Omega} \int_{0}^{r_-(k,T)}  \rho(r)
\bigl(\int_{E^\ast} \psi_k(\gr(v)) d\mu_{E}^{\Leb}(v)\bigr)\, dr dk
 \le \la F_{B_T(\Omega)}, \psi\ra \\
\le
\int_{k\in \Omega} \int_0^{r_+(k,T)} \rho(r)
\Bigl(\int_{E^\ast} \psi_k(\gr(v)) d\mu_{E}^{\Leb}(v)\Bigr)\, dr dk,
\end{array}
\end{equation*}
where $\psi_k\in C_c(\G\ba G)^M\cong C_c(\T^1(\G\bs\bH^n))$ is given by
\[
\psi_k(g)=\int_{m\in M}\psi(gmk)dm.
\]
\end{proposition}

\begin{proof}
By \eqref{eq:dh}, \eqref{eq:haar:g=hak}, \eqref{eq:614}, \eqref{eq:644},
\eqref{eq:rmk=rk} and Lemma~\ref{lemma:N-Leb}, we get
\begin{equation*}
\begin{array}{ll}
&\la F_{B_T(\Omega)}, \psi\ra\\& =\int_{k\in \Omega} \int_{\{r\geq 0:\|w_0 a_r k\| <T\}}
\bigl(\int_{[h]\in \G_{w_0}\ba H} \psi(h a_r k)\, dh\bigr) \rho(r)\, dr dk
\notag\\
&\leq
\int_{k\in \Omega} \int_{0}^{r_+(k,T)}
\bigl(\int_{[h]\in \G_{w_0}\ba H} \psi(h a_r k)\, dh\bigr) \rho(r)\, dr dk
\notag\\
&= \int_{k\in \Omega}  \int_{0}^{r_+(k,T)}
\bigl(\int_{[h]\in \G_{w_0}\ba H} \int_{m\in M} \psi(h a_r m k)\,dm dh\bigr) \rho(r)\, dk,
\notag \text{ as $M\Omega =\Omega$} \\
&=\int_{k\in \Omega} \int_0^{r_+(k,T)}
\bigl(\int_{[h]\in\G_{w_0}\ba H} \psi_{k}(h a_r)\, dh\bigr) \rho(r)\, dr dk
\notag  \\
&= \int_{k\in \Omega} \int_0^{r_+(k,T)} \rho(r)
\bigl(\int_{E^\ast} \psi_k(\gr(v)) d\mu_{E}^{\Leb}(v)\bigr)\, dr dk.
\end{array}
\end{equation*}

The other inequality is proved similarly.
\end{proof}

\begin{proposition}\label{posone}\label{eq:540}
For any $\psi\in C_c(\G\ba G)$, we have
\begin{equation*} 
\lim_{T\to\infty} T^{-\delta/\lambda} \la F_{B_T(\Omega)}, \psi\ra
=\frac{\mu^{\PS}_{E}(E^\ast)}{\delta \cdot\abs{m^{\BMS}}} \cdot
m^{\BR}(\xi_{w_0}\ast_{\Omega} \psi),
\end{equation*}
where $\xi_{w_{0}}(k)=\norm{w_{0}^{\lambda}k}^{-\delta/\lambda}$.
 \end{proposition}

\begin{proof} Without loss of generality, we may assume that $\psi$ is non-negative.
For any $\e>0$ and $k\in K$, by Theorem~\ref{mainergint} and \eqref{eq:Xi},
there exists $r_0>0$ such that for any $r>r_0$:
\begin{gather}
\label{eq:main:plus}
 e^{(n-1-\delta)r} \int_{v\in  E^\ast
} \psi_k(\gr(v)) \,d\mu^{\Leb}_{E}(v) =
\frac{\mu^{\PS}_{E}(E^\ast)\cdot m^{\BR}(\psi_{k})}{\abs{m^{\BMS}} } +O(\e);
\\ \label{e2}
\rho(r)=(1+O(\e))e^{(n-1)r}.
\end{gather}
Since $\psi\in C_c(\G\bs G)$, the map $K\ni k\mapsto \psi_k$ is continuous with respect to the sup-norm on  $C_c(\T^1(\bH^n))
$. Therefore since $K$ is compact, we can choose $r_0>0$ independent of $k\in K$. Now for sufficiently large $T>1$,
\begin{equation}
\begin{array}{l}
 \int_{r_0}^{r^\pm(k,T)} \rho(r) \int_{E^\ast} \psi_k(\gr(v))
d\mu_{E}^{\Leb}(v) dr
\\
=\int_{r_0}^{r_{\pm}(k,T)} \rho(r) e^{(-n+1+\delta)r}
\bigl(e^{(n-1-\delta)r} \int_{E^{\ast}} \psi_k(\gr(v))\, d\mu_{E}^{\Leb}(v) \bigr)\,dr
\\
=\bigl(\frac{\mu^{\PS}_{E}(E^{\ast})\cdot m^{\BR}(\psi_k)}{|m^{\BMS}|} +O(\e)\bigr)(1+O(\e))\int_{r_0}^{r_{\pm}(k,T)} e^{\delta r}
\,dr
\\
= \frac{\mu^{\PS}_{E}(E^\ast)\cdot m^{\BR}(\psi_k)}{\abs{m^{\BMS}}}\cdot \frac{T^{\delta/\lambda}
\norm{w_{0}^{\lambda}k}^{-\delta/\lambda}}{\delta} + O(\e) T^{\delta/\lambda}  +O(e^{\delta r_0}),
\end{array}
\label{eq:634}
\end{equation}
where the last equation follows from \eqref{eq:rpm2} for sufficiently large $T$.

Since $E\subset \T^1(\G\ba \bH^n)$ is a closed subset,
$\psi\in C_c(\G\ba G)$ and $K$ is compact,  it follows that for fixed $r_0>1$,
we have
\[
\sup_{\abs{r}\leq r_{0},k\in K} \int_{E}
\psi_k(v a_r )\, d\mu_{E}^{\Leb}(v)=O(1).
\]
Hence
\begin{equation} \label{eq:550}
\int_{\{r:\|w_0 a_r k\|
<T, \abs{r} \le r_0\}} \rho(r) \int_{E}
\psi_k(\gr(v)) d\mu_{E}^{\Leb}(v) dr =O(e^{(n-1)r_0}).
\end{equation}

By Proposition~\ref{prop:H:E},  \eqref{eq:634} and \eqref{eq:550},
\begin{align*}
\lim_{T\to\infty}\frac{\la F_{B_T(\Omega)}, \psi\ra}{T^{\delta/\lambda}}=\frac{\mu^{\PS}_{E}(E^\ast)}{\delta\cdot \abs{m^{\BMS}}}
\cdot  \int_{k\in\Omega} \norm{w_{0}^{\lambda} k}^{-\delta/\lambda} m^{\BR}(\psi_{k})\,dk +O(\e).
\label{eq:520}
\end{align*}
Since $\e>0$ is arbitrary, we finish the proof.
\end{proof}

\begin{lemma}[Strong wavefront lemma]
\label{swl} There exist $\ell>1$ and $\e_0>0$ such
that for any $0<\e<\e_0$ and $g=hak\in HA^{+}K$ with $\norm{a}\geq 2$,
\[
gU_\e\subset h(H\cap U_{\ell \e}) a (A\cap U_{\ell \e}) k(K\cap U_{\ell \e}),
\]
where $\norm{g}$ denotes the distance of $g$ from
$e$ in $G$ which is $K$-invariant.
\end{lemma}

\begin{proof} If $H$ is symmetric, the result follows from \cite[Theorem~4.1]{GorodnikShahOh2010}.

Now suppose that $H=N$ is horospherical. We may assume that the distance from $e$ in $G$ is invariant under conjugation by 
elements of $K$. Let $u\in U_{\e}$. Then $kuk\inv\in U_{\e}$. Write $kuk\inv=h_{1}a_{1}k_{1}$, where $h_{1}\in H\cap U_{\ell\e}
$, $a_{1}\in A\cap U_{\ell\e}$ and $k_{1}\in K\cap U_{\ell \e}$ for some $\ell\geq 1$ independent of $\e$. Now
\[
gu=haku=ha(kuk\inv)k=(h(ah_{1}a\inv))(aa_{1})k(k\inv k_{1}k).
\]
Since $a\in A^{+}$ and $h_{1}\in H=N$, by \eqref{eq:Nhoro},  $\norm{ah_{1}a\inv}\leq \norm{h_{1}}$. Also
$\norm{k\inv k_{1}k}=\norm{k_{1}}$. Hence $gu$ has the required form.
\end{proof}

\begin{proof}[Proof of Theorem~\ref{thm:count-sector}(1)]
By the assumption that $\nu_o(\partial(\Omega^{-1}))=0$, for all sufficiently small $\e>0$,
 there exists an $\e$-neighborhood $K_\e$ of $e$ in $K$
  such that for $\Omega_{\e+}=\Omega K_\e $ and
  $\Omega_{\e-}=\cap_{k\in K_\e} \Omega k$,
 \begin{equation}\label{eq:600}
 \lim_{\e\to 0} \nu_o(\Omega_{\e +}^{-1} -  \Omega_{\e-}^{-1})= 0.
 \end{equation}

Let $\ell>1$ as in Lemma \ref{swl}. Then for $T\gg 1$,
$$B_T(\Omega) U_{\ell^{\inv} \e} \subset B_{(1+\e)T}(\Omega_{\e+}) \quad\text{and}\quad
B_{(1-\e)T}(\Omega_{\e-})\subset\cap_{u\in U_{\ell\inv \e}} B_T(\Omega) u .$$

Let $\psi_\e\in C_c(G)$ be a non-negative function supported on $U_{\ell\inv \e}$
and $\int\psi_\e dg=1$, and let $\Psi_\e\in C_c(\G\ba G)$ the
$\G$-average of $\psi_\e$:
\begin{equation} \label{eq:Psi_e}
\Psi_\e(g):=\sum_{\gamma\in \G}\psi_\e(\gamma g).
\end{equation}
Then  $F_{B_{(1-\e)T}(\Omega_{\e-})} (g)\le F_{B_T(\Omega)}(e)
\le F_{B_{(1+\e)T}(\Omega_{\e+})}(g)$ for all $g\in U_{\ell\inv \e}$.
Therefore, by integrating against $\Psi_\e$, we have
\begin{equation*} \label{eq:610}
\la F_{B_{(1-\e)T}(\Omega_{\e-})}, \Psi_\e \ra \le
F_{B_T(\Omega)}(e)\le \la F_{B_{(1+\e)T}(\Omega_{\e+})},
\Psi_\e\ra.
\end{equation*}

Let $\xi_{w_0}$ be as defined in Proposition \ref{posone}.
By Proposition \ref{corpt}, for any $\eta>0$, there exists $\e>0$ such that
\begin{align*}
m^{\BR}(\xi_{w_0} \ast_{\Omega} \Psi_\e)
= \tilde m^{\BR}(\xi_{w_0} \ast_{\Omega} \psi_\e)
= \int_{k\in \Omega^{-1}} \xi_{w_{0}}(k\inv)\,d\nu_o(kX_{0}^{-})+O(\eta).
\end{align*}

Therefore by Proposition~\ref{posone},
\begin{equation}
\begin{array}{l}
  \lim_{T\to\infty} T^{-\delta/\lambda}
 \cdot \la F_{B_{(1\pm\e)T}(\Omega_{\e{\pm}} )}, \Psi_\e\ra
 \\
 = \frac{{\mu^{\PS}_{E}(E^{\ast})}}{\delta \cdot\abs{m^{\BMS}}} \cdot
\int_{k\in \Omega_{\e\pm}^{-1}} \xi_{w_{0}}(k\inv)\,d\nu_o(kX_{0}^{-}) + O(\eta);
\label{eq:630}
\end{array}
\end{equation}

In view of \eqref{eq:600}, we get
\begin{equation*}
\lim_{T\to \infty}\frac{F_{B_{T}(\Omega)}(e)}{T^{\delta/\lambda}}
=\frac{{\mu^{\PS}_{E}(E^{\ast})}}{\delta \cdot\abs{m^{\BMS}}} \cdot
\int_{k\in \Omega^{-1}} \xi_{w_{0}}(k\inv)\,d\nu_o(kX_{0}^{-})+O(\eta).
\end{equation*}
Since $\eta>0$ is arbitrarily chosen, we finish the proof of (1).
\end{proof}

\begin{proposition} \label{prop:HA-K}
Suppose that $H=G_{w_{0}}$ is symmetric and that $G\neq HA^{+}K$. Let $\Omega\subset M\bs K$
such that $\nu_{o}(\partial(\Omega^{\inv}X_{0}^{-}))=0$.
Then
\begin{equation} \label{eq:661}
\begin{array}{c}
\lim_{T\to\infty}\frac{\#(w_{0}\G\cap B_{T}\cap w_{0}A^{-}\Omega)}{T^{\delta/\lambda}} \\ =
\frac{\muPS_{E}(E^{-})}{\delta\cdot \abs{m^{\BMS}}}
\int_{k\in\Omega \inv} \norm{w_0^{-\lambda} k\inv}^{-\delta/\lambda}\, d\nu_o(kX_0^-).
\end{array}
\end{equation}
\end{proposition}

\begin{proof}
For $k\in K$ and $T>0$, let $s(k,T)=\sup\{r>0:\norm{w_{0}a_{-r}k}<T\}$. Then there exist
$A_{0}>0$ and $T_{0}>0$ such that if we define $s_{\pm}(k,T)$ via
\[
e^{s_{\pm}(k,T)}=(1\pm A_{0}T^{-\e_{0}})(T/\norm{w_{0}^{-\lambda}k})^{1/\lambda},
\]
then for all $T\geq T_{0}$, we have
$s_{-}(k,T)\leq s(k,T)\leq s_{+}(k,T)$.

By Theorem~\ref{mainergint}, for any $\phi\in C_{c}(\G\bs\T^{1}(\bH^{n}))$, we have
\begin{equation*}
\begin{array}{l}
\lim_{r\to\infty} e^{(n-1-\delta)r}\int_{E^{+}}\phi(\grinv(v))\,d\mu_E^{\Leb}(v)
\\
=\lim_{r\to\infty} e^{(n-1-\delta)r}\int_{E^{-}}\phi(\gr(v))\,d\mu_E^{\Leb}(v)
=\frac{\muPS_{E}(E^{-})}{\delta \cdot \abs{m^{\BMS}}}\cdot m^{\BR}(\phi).
\end{array}
\end{equation*}

Let $B_{T}^{-}(\Omega)=B_{T}\cap w_{0}A^{-}\Omega$ and
$$F_{B_{T}^{-}(\Omega)}(g):=\sum_{\gamma\in \G_{w_0}\ba\G}\chi_{B_T^{-}(\Omega)}
(w_0\gamma g).$$

In view of these observations, by arguing as in the proof of Proposition~\ref{posone}, we get
that for any $\psi\in C_{c}(\G\bs G)$,
\begin{equation*}
\begin{array}{l}
\lim_{T\to\infty}T^{-\delta/\lambda}\la F_{B_T^{-}(\Omega)}, \psi\ra
\\
=\lim_{T\to\infty} T^{-\delta/\lambda}
\int_{k\in \Omega} \int_{\{r>0:\|w_0 a_{-r} k\| <T\}}
\bigl[\int_{[h]\in \G_{w_0}\ba H} \psi(h a_{-r} k)\, dh\bigr] \rho(r)\, dr dk
 \\
=\frac{\mu^{\PS}_{E}(E^{-})}{\delta \cdot \abs{m^{\BMS}}}
 \cdot \int_{k\in\Omega} \norm{w_{0}^{-\lambda}k}^{-\delta/\lambda} m^{\BR}(\psi_{k})\,dk.
 \end{array}
 \end{equation*}
Now \eqref{eq:661} follows from the arguments as in the proof of
Theorem~\ref{thm:count-sector} (1).
\end{proof}

\begin{remark} \label{rem:par0}
If $G_{\R w_{0}}$ is parabolic, then $w_{0}a_{r}\to 0$ as $r\to -\infty$. Since $w_{0}\G$ is
discrete,
\begin{equation} \label{eq:finiteset}
\#(w_{0}\G \cap w_{0}A^-K)<\infty.
\end{equation}
\end{remark}

\begin{proof}[Proof of Theorem~\ref{thm:counting1}(2)]
If $G=HA^{+}K$, then (2) follows from (1) by
putting $\Omega=K$.

If $H$ is symmetric and $G\neq HA^{+}K$, $G=HA^+K\sqcup HA^- K$, and then
(2) follows by combining (1) and Proposition~
\ref{prop:HA-K} and putting $\Omega=K$.

If $G_{\R w_{0}}$ is parabolic, \eqref{eq:main:counting} follows from
Theorem~\ref{thm:count-sector} and \eqref{eq:finiteset}.
\end{proof}

\subsection{Counting in bisectors of $HA^+K$ coordinates}
\ref{oc}. $K$ a maximal compact subgroup.
We state a counting result for bisectors in $HA^+K$ coordinates. For any $g\in HA^+K$, we
set $a(g)$ to be the $A^+$-component of $g$, which is unique.
Consider bounded Borel subsets $\Omega_1\subset H$ and $ \Omega_2\subset K$ with
$\Omega_1(H\cap M)=\Omega_1$ and $M\Omega_2=\Omega_2$.
Set $$N_T(\Omega_1, \Omega_2)=\# ( \G \cap \Omega_1A_T^+\Omega_2)$$
where $A_T^+=\{a_r\in A^+: e^r <T\} .$ For the sake of simplicity,
we assume that the projection map $\Omega_1 \to \G\ba G$ is injective.

\begin{theorem}\label{bisec}
 If
 $\mu_E^{\PS}(\partial(\Omega_1(X_0)))=\nu_o(\partial(\Omega_2^{-1}(X_0^{-})))=0$, then
\[
\lim_{T\to\infty} \frac{ N_T(\Omega_1, \Omega_2)}{T^\delta}=
\frac{1}{\delta\cdot \abs{m^{\BMS}}}
\mu_E^{\PS}(\Omega_{1}(X_0))\cdot\nu_o(\Omega_{2}^{-1}(X_0^-)).
\]
\end{theorem}

This result for $H=K$ was also obtained by Roblin \cite{Roblin2003} by a different approach.
 When $\G$ is a lattice in a semisimple Lie group $G$ and $H=K$, the analogue of Theorem~\ref{bisec} was obtained in 
\cite{GorodnikOh2007}.

\begin{proof}

We define the following function on $\G\ba G$:
$$F_{T, \Omega_1, \Omega_2}(g):=\sum_{\gamma\in \G}\chi_{\Omega_1A^+_T\Omega_2}
 (\gamma g) .$$

For $\psi\in C_c(\G\ba G)$, given $\e>0$, by Theorem \ref{m2} for sufficiently large $T>1$,
\begin{equation}
\begin{array}{l}
 \la F_{T,\Omega_1, \Omega_2},\psi\ra =\int_{g\in \Omega_1A^+_T\Omega_2}\psi(g)dg
\\
=\int_{k\in \Omega_2} \int_{1\le e^r<T} \int_{h\in \Omega_1} \psi(h a_r k) \rho(a_r)
dh dr dk
\\ =
\int_{k\in \Omega_2}\int_{T_{0}\le e^r <T} \rho(a_r)
\bigl(\int_{h\in \Omega_1.X_{0}}\psi_{k}(h a_r)
dh\bigr) dr dk + O_{T_{0}}(1)
\\  =
\Bigl(\frac{1}{\delta\cdot |m^{\BMS}|}\mu_E^{\PS}(\Omega_{1}X_0^{+})
\int_{k\in \Omega_2}  m^{\BR}(\psi_{k}) dk+O(\e)\Bigr)
\\
\qquad \times \bigl(\int_{0}^{\log T}e^{(r-n-1)\delta}\rho(r) dr \bigr) + O_{T_{0}}(1)
\\ =
\frac{T^{\delta}}{\delta\cdot |m^{\BMS}|}\mu_E^{\PS}(\Omega_{1}X_0)\cdot
m^{\BR}(\chi_{K}{\ast} _{\Omega_2} \psi)+O(\e)T^{\delta}+O_{T_{0}}(1),
\label{fto1}
\end{array}
\end{equation}
where $\chi_K*_{\Omega_2}\psi(g)=\int_{k\in \Omega_2} \psi(gk) dk$.

By the assumptions on $\Omega_1$ and $\Omega_2$, for every $\e>0$ there exist
$\e$-neighborhoods $H_\e$ and $K_\e$ of $e$ in $H$ and $K$ respectively  such that
 for $\Omega_{1,\e^-}:=\cap_{h\in H_\e(H\cap M)} \Omega_1 h$,
$\Omega_{1,\e^+}:=\Omega_1 H_\e(H\cap M)$, $\Omega_{2,\e^-}:=\cap_{k\in K_\e} \Omega_2 k$ and
$\Omega_{2,\e^+}:=\Omega_2 K_\e$, as $\e\to 0$,
\begin{equation*} \label{eq:bdrypm}
\mu_E^{\PS}(\Omega_{1, \e^ +}(X_0)\setminus \Omega_{1,\e ^-}(X_0))\to 0, \
\nu_o(\Omega^{-1}_{2, \e^+}(X_0^-)\setminus\Omega_{2,\e^-}^{-1}(X_0^-))\to 0.
\end{equation*}

By Lemma \ref{swl}, for $\ell>1$ as therein,
there exists an $\e$-neighborhood $U_\e$ of $G$ such that for all $T\gg 1$,
\[
\begin{array}{c}
\Omega_1 A_T^+
 \Omega_2 U_{\ell\inv \e}
  \subset \Omega_{1,\e^+}A_{(1+\e)T}^+ \Omega_{2, \e^+}\\
\Omega_{1,\e^-}A_{(1-\e)T}^+ \Omega_{2, \e^-}\subset \cap_{g\in U_{\ell\inv \e}} \Omega_1 A_T^+ \Omega_2 g.
\end{array}
\]
 Let $\psi_\e\in
C_c(G)$ be a non-negative function supported on $U_{\ell\inv \e}$
and $\int\psi_\e dg=1$, and let $\Psi_\e\in C_c(\G\ba G)$ the
$\G$-average of $\psi_\e$:
$$\Psi_\e(g)=\sum_{\gamma\in \G} \psi_\e (\gamma g) .$$
It follows that
\begin{equation} \label{eq:FTpm}
\la F_{(1-\e)T,\Omega_{1,\e^-}, \Omega_{2,\e^-}}, \Psi_\e\ra
   \le F_{T, \Omega_1, \Omega_2}(e)\le
\la  F_{(1+\e)T,\Omega_{1,\e^+}, \Omega_{2,\e^+}}, \Psi_\e \ra.
\end{equation}

On the other hand, by Proposition~\ref{corpt},
\[
\lim_{\e\to 0} m^{\BR}(\chi_K  *_{\Omega_{2,\e^\pm}} \Psi_\e)
=\nu_o(\Omega_{2,\e^\pm}^{-1}(X_0^-)).
\]
Therefore by \eqref{fto1},
\[
\lim_{T\to\infty} T^{-\delta}\la F_{(1\pm \e)T,\Omega_{1,\e^\pm}, \Omega_{2,\e^\pm}}, \Psi_\e\ra
= \frac{\mu_E^{\PS}(\Omega_{1,\e^\pm}(X_0) ) \nu_o(\Omega_{2,\e^\pm}^{-1}(X_0^{-}))}{\delta\cdot |m^{\BMS}|}.
\]
By \eqref{eq:FTpm} we get
$$\lim_{T\to\infty} \frac{ F_{T,\Omega_1, \Omega_2}(e)}{T^{\delta}}
=\frac{1}{\delta\cdot |m^{\BMS}|}\mu_E^{\PS}(\Omega_{1,\e^+}(X_0))
\nu_o(\Omega_{2,\e^+}^{-1}(X_0^-)).$$
\end{proof}

\subsection{Counting theorems for $\G$ Zariski dense} \label{zde}
In the case when $\G$ Zariski dense,
Theorem \ref{thm:count-sector}
 holds for any norm on $V$ and for any $\Omega$ without the $M$-invariance condition.
Similarly, Theorem~\ref{bisec} holds without the $M$-invariance assumption
on $\Omega_1$ and $\Omega_2$.

The reason that this generalization is possible is because
for $\G$ Zariski dense, we use Theorem~\ref{Zdapp} instead of Theorem \ref{mainergint}.  In proving 
Theorem~\ref{thm:count-sector}, the place where we needed the $M$-invariance of $\Omega$ 
is  Proposition \ref{prop:H:E}; for general $\Omega$, we replace this proposition by:
\begin{equation*}
\begin{array}{l}
\int_{k\in \Omega} \int_{0}^{r_-(k,T)}  \rho(r)
\Bigl(\int_{\G_{w_0}\ba H} \psi_k(ha_r) dh \Bigr)\, dr dk \\
 \le \la F_{B_T(\Omega)}, \psi\ra 
\le
\int_{k\in \Omega} \int_0^{r_+(k,T)} \rho(r)
\Bigl(\int_{\G_{w_0}\ba H} \psi_k(ha_r) dh\Bigr)\, dr dk,
\end{array}
\end{equation*}
where $\psi_k(g):=\psi(gk)\in C_c(\G\ba G)$ is simply the translation of $\psi$ by $k$.

Applying Theorem \ref{Zdapp} to the inner integral in the above,
we deduce in the same way as in the proof of Proposition \ref{posone}
 that for any $\psi\in C_c(\G\ba G)$, we have
\begin{equation}\label{final1} 
\lim_{T\to\infty} T^{-\delta/\lambda} \la F_{B_T(\Omega)}, \psi\ra
=\frac{\mu^{\PS}_{E}(E^\ast)}{\delta \cdot\abs{m^{\BMS}}} \cdot
\bar{m}^{\BR}(\xi_{w_0}\ast_{\Omega} \psi),
\end{equation}
where $\mBR$ defined as in subsection \ref{A1} and
 $\xi_{w_{0}}(k):=\norm{w_{0}^{\lambda}k}^{-\delta/\lambda}$.
Now for a general norm $\|\cdot \|$ on $V$,
note that the function $\xi_{w_0}(k)$ 
is not necessarily $M$-invariant. 
However, for an approximate identity $\{\psi_\e\}_{\e>0}$ on $G$ and any $f\in C(K)$,
the proof of Proposition~\ref{corpt} can be easily modified to
prove \begin{equation}\label{final2} \lim_{\e\to 0}
 {\mBR}(f*_\Omega \psi_\e)=
 \int_{k\in \Omega^{-1}} f(k^{-1})\, d{\nu}_o(kX_0^-).
 \end{equation}

Hence applying \eqref{final1} to $\psi=\psi_\e$ and \eqref{final2}
to $f=\xi_{w_0}$ and by sending $\e \to 0$,
we obtain
\begin{equation}\label{final3}
\lim_{T\to\infty} { T^{-\delta/\lambda} }\cdot { F_{B_T(\Omega)}(e)}
=\frac{\mu^{\PS}_{E}(E^\ast)}{\delta \cdot\abs{m^{\BMS}}} \cdot
\int_{k\in \Omega^{-1}} \norm{w_{0}^{\lambda}k^{-1}}^{-\delta/\lambda} d\nu_o(kX_0^-).
\end{equation}
This explains the generalization of Theorem  \ref{thm:count-sector} (1).
The generalization for  Theorem  \ref{thm:count-sector} (2)
and Theorem~\ref{bisec} can be done similarly.

\begin{proof}[Proof of Theorem~\ref{thm:count-cone}]
In view of the above explanation, the result can be deduced from Theorem~\ref{thm:count-sector} (or its combination with 
Proposition~\ref{prop:HA-K} or Remark~\ref{rem:par0}) via elementary arguments; see
\cite{GorodnikOhShah2009}.
\end{proof}

\section{Appendix: Equality of two Haar measures} \label{a2}
Let $H$ be a symmetric group as in \S\ref{subsec:symmetric}.
As in Notation~\ref{not:approxid}(1), consider the Haar measure on $G$ corresponding to the Iwasawa decomposition 
$G=NAK$ given by 
\[
dg=e^{(n-1)t}\, dn\,dt\, dq, \text{ for $g=na_tq$, $n\in N$, $a_t\in A$, $q\in K$}.
\]
Corresponding to the generalized Cartan decomposition $G=HAK$, by \eqref{eq:haar:g=hak} the Haar measure on $G$ can be 
expressed as 
\[
dg=c_0\cdot  \rho(r)\,dh\, dr \, dk, \text{ for $g=ha_rk\in HAK$},
\]
where $c_0>0$ is a constant. We note that $dn$ is defined by Lemma~\ref{lemma:N-Leb} and $dh$ is determined by 
\eqref{eq:dh}. 

\begin{theorem} 
$c_0=1$. 
\end{theorem}

\begin{proof}
Let the notation be as in \S\ref{cmbr}. Let $N^-=\{g\in G: a_{-r} g a_{r}\to e \text{ as $r\to\infty$}\}$.
Then for $y\in \Lie(N^-)$ we have $a_{-r}\exp(y)a_{r}=\exp(e^{-r}y)$. In view of $NAN^-M$-decomposition of a small 
neighborhood of $e$ in $G$, for $h$ in such a neighborhood we write
\[
h=n(x(h))a_{b(h)}v(y(h))m(h)
\]
where $x(h)\in \Lie(N)\cong \R^{n-1}$ and $n(x(h))=\exp(x(g))$, $y(h)\in\Lie(N^+)\cong\R^{n-1}$ and 
$v(y(h))=\exp(y(h))$, $b(h) \in\R$ and $m(h)\in M$. In particular,
\begin{equation} \label{eq:xhXo}
hX_{0}^{+}=n(x(h))a_{b(h)}v(y(h))m(h)X_{0}^{+}=n(x(h))X_{0}^{+}
\end{equation}

In view of the decompositions $G=HAK$ and $G=NAK$, For $h\in H$, $r>0$ and $k\in K$, we express
\[
ha_rk=n(z(h,r,k))a_{t(h,r,k)}q(h,r,k), \text{ where $q(h,r,k)\in K$.}
\]
Now for $h$ in a small neighborhood of $e$ in $H$, we have
\begin{align*}
ha_rk=n(x(h))a_{b(h)}v(y(h))m(h)a_rk =n(x(h))a_{r+b(h)}v(e^{-r}y(h))(m(h)k).
\end{align*}
In view of $G=NAK$ decomposition,
\begin{gather*}
v(e^{-r}y(h))=n(x_1(h,r))a_{b_1(h,r)}k_1(h,r), \text{ with } \\
\max(\norm{x_1(h,r)},\norm{b_1(h,r)},\norm{k_1(h,r)})=O(e^{-r}\norm{x(h)}).
\end{gather*}
Therefore
\begin{align*}
ha_{r}k&=n(x(h))a_{r+b(h)}n(x_1(h,r))a_{b_1(h,r)}(k_1(h,r)m(h)k) \\
&=n(x(h)+x_2(h,r))a_{r+b(h)+b_1(h,r)}(k_1(h,r)m(h)k),
\end{align*}
where $x_2(h,r)=e^{-r-b(h)}x_1(h,r)$. So
\begin{equation} \label{eq:x2}
\norm{x_2(h,r)}=e^{-2r}O(\norm{x(h)}).
\end{equation}
Therefore
\begin{equation}
\begin{array}{l}
z(h,r,k)=n(x(h)+x_2(h,r)), \quad t(h,r,k)=r+b(h)+b_1(h,r),   \label{eq:zhrk} \\
q(h,r,k) =k_1(h,r)m(h)k.
\end{array}
\end{equation}
Since $z(h,r,k)=z(hm,r,e)$ and $t(h,r,k)=t(hm,r,e)$ for any $k\in K$ and $m\in M\cap H=G_{X_0^+}\cap H$, we can write
$z(h,r,k)=z([h],r)$ and $t([h],r,k)=t([h],r)$, where $[h]=h(M\cap H)=hX_{0}^+$. Moreover, for any fixed $h$ and $r$, since $dk$ is 
$K$-invariant, we have that $dq(h,r,k)=dk$.

For $h$ in a small neighborhood of $e$ in $H$, $r>0$ and $k\in K$,
\begin{equation*}
\begin{array}{ll}
c_0
&=\frac{e^{(n-1)t(h,r,k)}\, dn(z(h,r,k))\,dt(h,r,k)\,dq(h,r,k)}{\rho(r)\,dh\,dr\,dk}\\
&=\frac{e^{(n-1)t([h],r)}\,dn(z([h],r))\,dt([h],r)}{\rho(r)\,dh\,dr}\cdot \frac{dq(h,r,k)}{dk} \\
&=\frac{e^{(n-1)t([h],r)}\,dn(z([h],r))\,dt([h],r)}{\rho(r)\,dh\,dr},
\end{array}
\end{equation*}
because $z$ and $t$ do not depend on $k$ and for fixed $(h,r)$ we have $d(q(h,r,k))=dk$. 
Now the numerator depends only on $[h]=hM$ and $\int_{m\in H\cap M}1\, dm=1$. Therefore
\begin{multline}
c_0=\frac{e^{(n-1)t([h],r)} e^{(n-1)\beta_{n(z([h],r))X_{0}^+}(o,n(z(h,r,k))o)}}
{\rho(r)e^{(n-1)\beta_{[h]}(o,[h]o)}} \times \\ 
\frac{dm_o(n(z([h],r))X_{0}^+)\, dt([h],r)} {dm_o([h])\, dr}.
\label{eq:RNd}
 \end{multline}

To compute $c_0$, we evaluate the Radon-Nikodym derivative at the point $([h],r)=([e],s)=(X_{0}^+,s)$ for any fixed $s>0$.
Then we consider the upper half space model $\R^{n-1}\times \R_{>0}$ for $\bH^{n}$ with $o=(0,1)$ and $X_{0}^{-}=\infty$. 
Then  $X_{0}^{+}=0\in\R^{n-1}=\partial\bH^{n}\setminus\{\infty\}$. Since $m_o$ is equivalent to the Lebesgue measure, let
\begin{equation} \label{eq:mo0}
0<C:=\frac{dm_{o}(x)}{dx}\Bigr |_ {x=0} \text{; also $n(x)X_{0}^{+}=x$ for all $x\in\R^{n-1}$}.
\end{equation}

We define a map $\Phi$ from a small neighborhood of $(0,s)$ in $\R^{n-1}\times\R$ to $\R^{n-1}\times \R$ by
\[
\Phi([h],r)=(n(z(h,r,k)X_{0}^{+},t([h],r))).
\]
To compute the Jacobian of $\Phi$ at the point $(X_{0}^{+},s)=(0,s)$,
we write $\Phi=(\Phi_{1},\Phi_{2})$ and $([h],r)=(z_{1},z_{2})$.

Fixing $[h]=[e]$, we get $z([e],r)=0$, $t([e],r)=r$. Therefore $\partial_{z_{2}}(\Phi_{1}, \Phi_{2})= (0,1)$.
Hence the Jacobian of $\Phi$ at $([h],r)=(0,s)$ is
\begin{equation*}
\begin{array}{ll}
J(\Phi)(0,s) & =\abs{\partial_{z_{1}}\Phi_{1}(0,s)} \\
& = \frac{dm_o(n(x_{2}([h],s)+x(h))X_0^+)}{dm_o([h])}
\text{ at $[h]=0$,   by \eqref{eq:zhrk}}
\notag \\
& =\frac{dm_{o}(n(x_{2}([h],s)+x([h]))X_{0}^+)}{dm_{o}(n(x([h]))X_{0}^{+})}, \text{ by  \eqref{eq:xhXo}} \\
&=\frac{d(x_{2}([h],s)+x([h]))}{d(x([h]))} \text{ at $[h]=0=x([h])$, by \eqref{eq:mo0}}  \notag \\
&=1+ \frac{d(x_{2}([h],s))}{d(x([h]))} \text{ at $[h]=0=x([h])$} \notag \\
&=1+O(e^{-2s(n-1)}), \text{ by \eqref{eq:x2}};
\end{array}
\end{equation*}
note that for a fixed $s$, due to \eqref{eq:xhXo} and \eqref{eq:mo0}, $x_2([h],s)$ is a smooth function of $x([h])$. By 
\eqref{eq:RNd}, 
the Radon-Nikodym derivative at $([h],r)=([e],s)$ is
\begin{align*}
c_{0}&=\frac{e^{(n-1)t([e],s)} e^{(n-1)\beta_{n(z([e],s))X_0^+}(o,n(z([e],s))o)}}
{\rho(s)e^{(n-1)\beta_{[e]}(o,[e]o)}}\cdot J(\Phi)(0,s)\\
&=(e^{(n-1)s}/\rho(s))(1+O(e^{-2s(n-1)})).
\end{align*}
Since $\rho(s)/e^{(n-1)s}\to 1$ as $s\to\infty$,  we have  $c_{0}=1$.
\end{proof}

\end{document}